\documentclass[11pt, preprint,natbib]{imsart}
\usepackage[plainpages=false]{hyperref}

\usepackage{xcolor}
\definecolor{ultramarine}{RGB}{0,32,96}
\definecolor{hookersgreen}{rgb}{0.0, 0.44, 0.0}
\hypersetup{
            colorlinks,
            linkcolor=hookersgreen,
            linktoc=blue,
            citecolor=ultramarine,
            urlcolor=black,
            filecolor=black
            }

\usepackage{nameref}

\startlocaldefs
\endlocaldefs

\usepackage[utf8]{inputenc}
\usepackage{amsmath}
\usepackage{amssymb}
\usepackage{nicefrac}
\usepackage{bm}
\usepackage{bbm} 
\usepackage{stackengine}
\usepackage{amsthm}
\usepackage{algorithm}
\usepackage{algorithmic}
\usepackage{dsfont}
\usepackage{graphicx}
\usepackage{textgreek}

\newcommand{\appropto}{\mathrel{\vcenter{
  \offinterlineskip\halign{\hfil$##$\cr
    \propto\cr\noalign{\kern2pt}\sim\cr\noalign{\kern-2pt}}}}}

  \newcommand\numleq[1]%
  {\stackrel{\scriptscriptstyle(\mkern-1.5mu#1\mkern-1.5mu)}{\leq}}
\newcommand{\defined}{\mathrel{\mathop:}=}

\newcommand{\kom}[1]{}

\newcommand\numgeq[1]%
  {\stackrel{\scriptscriptstyle(\mkern-1.5mu#1\mkern-1.5mu)}{\geq}}

\newtheorem{thm}{Theorem}
\newtheorem{prop}{Proposition}
\newtheorem{rmk}{Remark}
\newtheorem{lem}{Lemma}

\newtheorem{defn}{Definition}
\newtheorem{obs}{Observation}

\begin{document}

\begin{frontmatter}

\title{Adaptive Weights Community Detection\thanksref{t1}}
\runtitle{Adaptive Weights Community Detection}

 \author{\fnms{Franz} \snm{Besold}\corref{}}
 \thankstext{t1}{Financial support by German Ministry for Education via the Berlin Center for
Machine Learning (01IS18037I) is gratefully acknowledged.} 
\address{Weierstrass Institute, Mohrenstr. 39, 10117 Berlin, Germany\\ \href{mailto:franz.besold@wias-berlin.de}{\textcolor{ultramarine}{franz.besold@wias-berlin.de}}}
\and
\author{\fnms{Vladimir} \snm{Spokoiny}\thanksref{t2}}
 \thankstext{t2}{The research was supported by the Russian Science Foundation grant No. 18-11-00132.} 
\address{Weierstrass Institute and HU Berlin,
HSE University, IITP RAS,\\
Mohrenstr. 39, 10117 Berlin, Germany \\ \href{mailto:vladimir.spokoiny@wias-berlin.de}{\textcolor{ultramarine}{vladimir.spokoiny@wias-berlin.de}}}

\runauthor{F. Besold and V. Spokoiny}

\begin{abstract}
Due to the technological progress of the last decades, Community Detection has become a major topic in machine learning. However, there is still a huge gap between practical and theoretical results, as theoretically optimal procedures often lack a feasible implementation and vice versa. This paper aims to close this gap and presents a novel algorithm that is both numerically and statistically efficient. Our procedure uses a test of homogeneity to compute adaptive weights describing local communities. The approach was inspired by the Adaptive Weights Community Detection (AWCD) algorithm by \cite{AWCD}. This algorithm delivered some promising results on artificial and real-life data, but our theoretical analysis reveals its performance to be suboptimal on a stochastic block model. In particular, the involved estimators are biased and the procedure does not work for sparse graphs. We propose significant modifications, addressing both shortcomings and achieving a nearly optimal rate of strong consistency on the stochastic block model. Our theoretical results are illustrated and validated by numerical experiments.
\end{abstract}


\end{frontmatter}

 \newpage
\tableofcontents

\section{Introduction}
\subsection{Community Detection}
Community detection has become a major topic in modern statistics with applications in various fields. A very illustrative example is social graphs. Originally discussing relatively small examples such as the famous Zachary's network of karate club members \cite{karate_club} consisting of only 34 vertices, it is nowadays possible to process huge data with millions of vertices such as the Facebook graph \cite{facebook_example1, facebook_example2}, the amazon purchasing network \cite{amazon_example}, mobile phone networks \cite{mobile_phone_example} or most recently, data of COVID-19 infections \cite{CD_example_COVID}. Moreover, community detection has various applications in biology and bioinformatics \cite{survey_biology_cd}. Other popular examples are citation networks \cite{citation_example} and the world wide web \cite{www_example}. For a more exhaustive list of applications we refer to \cite{newman_survey, fortunato_survey, cd_another_survey_with_some_examples}.

The topics of community detection and clustering are clearly related. In fact, after embedding the nodes in a metric space, the problem of community detection can be reduced to the problem of clustering. Nevertheless, the underlying data of a graph is fundamentally different from a point cloud in \(\mathbb R^n\) and there is no canonical method to embed the vertices in a metric space. Consequently, the theoretical analysis and underlying models differ for each field. This has motivated the development of many methods that deal specifically with community detection. 

Similar to the topic of clustering, the task of community detection lacks a clear definition leading to a vast amount of algorithms with varying objectives. Most generally, the goal of community detection can be described as recovering groups of vertices with a similar connection pattern from a given graph. The graph may be weighted or directed. The most common case is that vertices inside each community are more densely connected to each other than to vertices of other communities. The set of groups may be a partition, a set of overlapping communities or a hierarchical structure.
One of the earliest and best-known methods is the Kernighan-Lin algorithm \cite{kernighan-lin-algorithm} which aims to find a bisection of a graph with a minimal number of edges connecting the two components through successively swapping vertices between the two communities. Applying it iteratively, this method produces a partition of the vertices. The number and size of the communities need to be known. A similar but faster method is the spectral graph bisection \cite{fiedler_spectral_graph_partitioning, barnes_spectral_bisection, pothen_some_spectral_graph_partitioning_algorithms}. However, the \emph{cut size}, i.e. the number of edges connecting the different components, is not a suitable quality measure of a community structure where the number and size of the communities are unknown. Instead, the so-called \emph{modularity} \cite{girvan_newman_modularity} has become the most popular quality function. It measures the discrepancy between the given number of edges between different communities and the expectation of this term for a similar random graph without the community structure. Originally, it was introduced as a stopping criterion for the algorithm of \cite{girvan_newman_algorithm}. However, there have been introduced many algorithms since that directly maximize modularity, starting with the greedy optimization \cite{newman_greedy_modularity_optimization, amazon_example, mobile_phone_example} and including for example spectral optimization \cite{newman_spectral_modularity_optimization} and simulated annealing \cite{simulated_annealing}. Modularity can be extended to the case of weighted graphs \cite{modularity_weighted_graph}, directed graphs \cite{modularity_directed_graph} and overlapping communities \cite{modularity_overlapping_communities}. There are plenty of other methods available besides modularity maximization, such as the \emph{Clique Percolation
Method} \cite{clique_percolation} constructing each community as a union of heavily overlapping cliques and thus allowing for overlapping communities, or most recently, also neural networks \cite{survey_sbm_deeplearning}. Hierarchical methods can be separated into \emph{agglomerative} and \emph{divisive} methods. Divisive algorithms iteratively split the communities into smaller communities. An example is the algorithm by \cite{girvan_newman_algorithm} starting with the original graph and successively removing edges based on a certain measure of \emph{betweenness}, e.g. of the number of shortest paths containing a given edge. Conversely, agglomerative algorithms iterative merge communities into larger communities. An example is the greedy optimization of modularity proposed by \cite{newman_greedy_modularity_optimization} starting by removing all edges from the graph and successively adding edges based on the impact on the modularity. Considering the communities as connected components, a hierarchical structure can be obtained. For a comprehensive survey on community detection algorithms, we refer to \cite{fortunato_survey}. 

A big challenge in Community Detection is the gap between theoretical results and practically relevant algorithms. Many of those lack a rigorous statistical analysis even on the most basic models, while statistically efficient algorithms are often not numerically feasible. For example, the rate-optimal procedure suggested by \cite{sbm_general_minimax} offers no implementation of polynomial complexity, whereas surprisingly little is known about the performance of the popular Louvain algorithm \cite{mobile_phone_example} on the stochastic block model despite some recent progress \cite{louvain_on_sbm}.

\subsection{Stochastic Block Model}
The \emph{stochastic block model (SBM)} \cite{original_paper_sbm} is the simplest and by far the most studied model for community detection. Under the SBM, edges are generated by independent Bernoulli variables, with the parameters only depending on the corresponding communities for each pair of vertices. The minimax rates are well known for different forms of recovery \cite{survey_sbm}. In this paper we will only discuss \emph{exact recovery}, also known as \emph{strong consistency}, i.e. we want to recover the entire partition with large probability and without any misclassified vertices.

The SBM is of course not a very realistic model, as the degree distribution inside communities is usually not uniform in applications. It has for example been shown, that the SBM provides a poor fit for the famous karate club network \cite{sbm_on_karate_club}. However, there has been recently a lot of progress on the theoretical foundations of community detection that go beyond the SBM, such as the study of \emph{degree-corrected block models} \cite{degree-corrected-sbm} or \emph{graphons} \cite{graphons1, graphons2}. And also for practical simulation, there have been various models introduced that are much more realistic than the standard SBM such as the LFR benchmark \cite{benchmark_sbm1} or \cite{benchmark_sbm2} where the degree distribution follows a power law. 

Nonetheless, the study of the SBM for a new method can be very useful, as this model already captures a major information theoretic bottleneck for random graphs. We will start discussing a very simple version of the SBM, which is also called symmetric SBM. We consider the size of each community to be deterministic.
\begin{defn}\label{defn_sbm}
 Suppose \(K \in\mathbb Z_{\geq 2}\) is fixed,  \(n> 0\) and \(0\leq \rho < \theta \leq 1\). By SBM(\(n, K, \theta, \rho\)) we denote a random graph with \(nK\) vertices that are divided into \(K\) communities of size \(n\) and whose edges are generated according to independent Bernoulli variables of mean \(\theta\) inside the communities and mean \(\rho\) between different communities.
\end{defn}
We will later also discuss the generalization of our results to a more general SBM.

\subsection{AWCD revisited}
This paper follows up on a proposal by Larisa Adamyan, Kirill Efimov and Vladimir Spokoiny for a novel community detection method based on a hypothesis test of homogeneity for an SBM called \emph{adaptive weights community detection} (AWCD) \cite{AWCD}. The idea of this test originates from a likelihood-ratio test for local homogeneity \cite{Polzehl} with applications to image processing. Later, the same test was also used for the problem of adaptive clustering \cite{AWC}. Unfortunately, no theoretical analysis was provided for the AWCD algorithm. We will provide the first theoretical study in this paper and demonstrate that significant modifications are necessary to achieve a good performance on the SBM. We start by introducing the original algorithm and the idea behind it: Let us consider a general SBM with two disjoint communities \(\mathcal C^*_1\) and \(\mathcal C^*_2\). The communities may have different sizes. The edges are independent and follow Bernoulli distributions. For \(i=1, 2\), edges inside community \(\mathcal C^*_i\) follow a Bernoulli distribution with parameter \(\theta_i\), the remaining edges between the two communities follow a parameter \(\rho\). We consider the null hypothesis
\begin{itemize}
 \item[\(H_0\):]  \(\theta_1=\theta_2=\rho \in [0, 1]\)
\end{itemize}
against the alternative
\begin{itemize}
 \item[\(H_1\):] \(\theta_1,\theta_2,\rho \in [0, 1].\)
\end{itemize}
The likelihood-ratio test statistic turns out to be 
\[T = N_{11} \mathcal K\left(\widetilde\theta_{11}, \widetilde\theta_{1\lor 2}\right) + N_{22} \mathcal K\left(\widetilde\theta_{22}, \widetilde\theta_{1\lor 2}\right) + N_{12} \mathcal K\left(\widetilde\theta_{12}, \widetilde\theta_{1\lor 2}\right)\text{,}\]
where 
\begin{itemize}
 \item \(\widetilde\theta_{ii}\) denotes MLE of \(\theta_i\) under \(H_1\),
 \item \(\widetilde\theta_{12}\) denotes MLE of \(\rho\) under \(H_1\),
 \item \(N_{ij}\) denotes the sample size of the respective MLEs and
 \item \(\widetilde\theta_{1\lor 2}\) denotes MLE of \(\rho\) under \(H_0\).
\end{itemize}
Given data in form of an adjacency matrix \(\mathcal Y\), the AWCD algorithm applies this test iteratively on \textit{local communities} \(\mathcal C_i\) that are computed for each node \(i\) and represented by a weight matrix \(W\) such that \(\mathcal  C_i = \{j : W_{ij} = 1\}\). An exact description is given in Algorithm \ref{algorithm}. \cite{AWCD} propose to use the usual graph theoretic neighborhood \(\mathcal Y\) as a starting guess.
Ideally, for some tuning parameter \(\lambda\), the local community structure \(W\), which is updated at each step of the procedure, converges towards the true underlying community structure \(W^*\). 
Note that the proposed starting guess violates the setup of the likelihood ratio test w.r.t. the following points:
\begin{itemize}
 \item Local Communities are \textit{overlapping}.
 \item Local Communities are \textit{small} compared to true communities.
 \item Local Communities are of \textit{low precision}, i.e. they contain many false members.
\end{itemize}
As we will discuss in the following, these violations lead to practical and theoretical limitations of the procedure and motivated us to propose substantial modifications.

In this paper, we address the current gap between theoretical and practical results in Community Detection. Our contributions include:
\begin{itemize}
 \item A novel procedure based on \cite{AWC} with the following modifications:
\begin{itemize}
\item Reducing the bias of the involved estimators.
\item Increasing the size of the initial communities.
\end{itemize}
\item Rates of strong consistency of the new algorithm on the SBM: Contrary to \cite{AWC}, the rate is nearly optimal in the most common case where the quotient between two Bernoulli parameters of a symmetric stochastic block model is constant.
\item Numerical illustrations and validation for our theoretical results.
\end{itemize}
The rest of the paper is organized as follows. In section \ref{sec_results} we present our main results. We start in subsection \ref{subsec_improvement_bias_correction} by discussing rates of strong consistency for both the original version of the algorithm as well as a modified version that removes bias terms. In subsection \ref{subsec_larger_starting_guess} we will show that these results can be extended to sparse graphs by increasing the size of the initial starting guess for each neighborhood. In particular, this extension of the algorithm achieves a nearly optimal rate of strong consistency for stochastic block models using two Bernoulli parameters having a constant quotient. For simplicity, these results are stated for the symmetric stochastic block model. We discuss the generalization to a more general stochastic block model in subsection \ref{subsec_general_sbm}. In the following section \ref{sec_experiments} we present numerical results illustrating the main results of section \ref{sec_results}. All proofs are collected in section \ref{sec_proofs}. In appendix \ref{ref_appendix} we discuss further details on the presented rates of consistency.

\begin{algorithm}
\caption{AWCD algorithm from \cite{AWCD}}
\label{algorithm}
\begin{algorithmic} [1]
\STATE\textbf{input}: adjacency matrix \(\mathcal Y\), number of iterations \(l_{\max}\) and a test threshold \(\lambda \in \mathbb R\)
\STATE initialize the weight matrix \(W^{(0)}=  \mathcal Y\)
\FOR{\(l\) from \(1\) to \(l_{\max}\)}
\STATE \(S =  \left( W^{(0)} \mathcal Y W^{(0)}\right)_{i, j} \)
\STATE \(N_{ij} = |\mathcal C_i|  \left(|\mathcal C_j| - \mathbbm1(i = j)\right)\)
\STATE \(\displaystyle\widetilde\theta_{ij} =   \frac{S_{ij}}{N_{ij}}\)
\STATE \( \widetilde\theta_{i\lor j}= \frac{S_{ii}+2S_{ij}+S_{jj}}{N_{ii}+2N_{ij}+N_{jj}}\)
\STATE compute test statistic \(T= N_{11} \mathcal K\left(\widetilde\theta_{11}, \widetilde\theta_{1\lor 2} \right) + N_{22} \mathcal K\left(\widetilde\theta _{22}, \widetilde\theta_{1\lor 2} \right) + N_{12}  \mathcal K\left(\widetilde\theta_{12}, \widetilde\theta_{1\lor 2}\right)\)
\STATE update weight matrix \(W= \left(\mathbbm{1}\left(T_{ij} \leq \lambda \right)\right)_{i, j}\)
\STATE update initial weight matrix \(W^{(0)} =W\)
\ENDFOR
\STATE \textbf{output}: matrix of binary weights \(W \)
\end{algorithmic}
\end{algorithm}

\section{Results}\label{sec_results}
\subsection{Improvement of rates via bias correction}\label{subsec_improvement_bias_correction}
For simplicity, we start by considering the very simple model SBM(\(n, K, \rho, \theta\)) introduced in Definition \ref{defn_sbm}. As we are interested in the asymptotics \(n \to \infty\), the parameters \(\theta\) and \(\rho\) may depend on \(n\). However, we will not use the index \(n\) explicitly to simplify the notation. In contrast to \(\theta\) and \(\rho\), we consider the number of communities \(K\) to be fixed. We denote the corresponding adjacency matrix by \(\mathcal Y\).
Moreover, we will only consider the first step of algorithm \ref{algorithm} for our theoretical analysis, i.e. \(l_{\max} = 1\).

The first modification of the algorithm that we propose is related to the definition of \(S_{ij}\). This sum is supposed to count the number of edges connecting the initial communities \(\mathcal C_i\) and \(\mathcal C_j\). In particular, after conditioning on \(\mathcal C_i\) and \(\mathcal C_j\), we would like to have a sum of independent Bernoulli variables. However, if we also condition on \(\mathcal Y_{ij} = 1\), the sum contains a deterministic part of size \(|\mathcal C_i|+|\mathcal C_j|-1\) which is with large probability of order \(\theta n\):
 \begin{align*}
  S_{ij} &= \sum_{k, l\notin \{i, j\}} \mathcal Y_{il} \mathcal Y_{lk} \mathcal Y_{kj}+\mathcal Y_{ij} \sum_{k\notin\{i, j\}} \left(\mathcal Y_{ik} \mathcal Y_{ki}+ \mathcal Y_{jk} \mathcal Y_{kj} \right) + \mathcal Y_{ij}^3\\
  &=  \underbrace{\sum_{k, l\notin \{i, j\}} \mathcal Y_{il} \mathcal Y_{lk}\mathcal Y_{kj}}_{\mathcal O( \theta^3 n^2)} + \mathbbm 1(\mathcal Y_{ij} = 1) \left(\underbrace{|\mathcal C_i| + |\mathcal C_j| - 1}_{\mathcal O(  \theta n)}  \right)
 \end{align*}
 As the sum \(S_{ij}\) is of order \(\theta^3 n^2\), it can only be informative as long as \(\theta \gg n^{-\frac{1}{2}}\). This problem can be easily fixed by redefining 
\begin{align*}
S_{ij}& \defined \left(W^{(0)}\mathcal YW^{(0)}\right)_{i, j} - \mathbbm 1(\mathcal Y_{ij}=1) \left(|\mathcal C_i|+|\mathcal C_j| - 1\right)\\
& = \sum_{\mathcal Y_{il}=1, \mathcal Y_{kj} = 1, k, l\notin\{i, j\}} \mathcal Y_{kl}.
\end{align*}
Let us call this debiased version of the algorithm AWCD$_1$ and the original version AWCD\(_1^\circ\). Indeed, this modification improves the rate of \(\theta - \rho\) significantly:
\begin{thm}\label{thm_original_AWCD}
For both versions of the algorithm we consider the asymptotics \(\frac{\theta}{\rho}\equiv C\) as well as the more general condition \(\frac{\theta}{\rho} \leq C\) for the symmetric SBM\((n, K, \theta, \rho)\) with a fixed number of communities \(K\). Consistent exact recovery, i.e. \[\mathbb P (\forall i, j : W_{ij} = W^*_{ij} ) \to 1,\] is achieved after the first step as long as \(\theta \leq \frac{1}{2}\) and \\
\renewcommand{\arraystretch}{1.2}
\begin{tabular}[]{|l|l|l|}
\hline
 & \(\frac{\theta}{\rho} \equiv C\)  & \(\frac{\theta}{\rho} \leq C \)\\
\hline
AWCD\(_1^\circ\) & \( \theta \gg n^{-\frac{1}{2}}\) &\(\theta - \rho \gg \max \{ n^{-\frac{1}{3}}\theta^\frac{1}{3}, n^{-\frac{1}{6}}\theta^\frac{5}{6} (\log n)^\frac{1}{6}\}\)  \\
 AWCD\(_1\)&  \(\theta \gg n^{-\frac{2}{3} }(\log n)^\frac{1}{3} \) &\(\theta - \rho\gg \max\{n^{-\frac{1}{3}} \theta^\frac{1}{2}(\log n)^\frac{1}{6},   n^{-\frac{1}{6}} \theta^\frac{5}{6} (\log n)^\frac{1}{6}\}.\)\\
  \hline
\end{tabular}\\
\end{thm}
\begin{rmk}
In the above table, we use the notation \(x\gg y\) to describe the following: There exists a constant \(c\) such that the statement of the theorem holds as long as \(c x > y\).
\end{rmk}
\begin{rmk}
The constant \(\frac{1}{2}\) may be replaced by any positive constant smaller than 1.
\end{rmk}

We have visualized the consistency regime for \(\frac{\theta}{\rho}\leq C\) in a log$_n$-log$_n$ plot of the quotient \(\frac{\theta-\rho}{\theta}\) against the parameter \(\theta\) in Figure \ref{fig_regime_k_1}. Note that in the case \(\frac{\theta}{\rho}\equiv C\), the minimax rate of consistent exact recovery is \((\log n) n^{-1}\), matching the minimal rate ensuring consistent connectivity of the graph, c.f. \cite{survey_sbm}.

\begin{figure}[h]
\begin{center}
   \includegraphics[scale = 0.75]{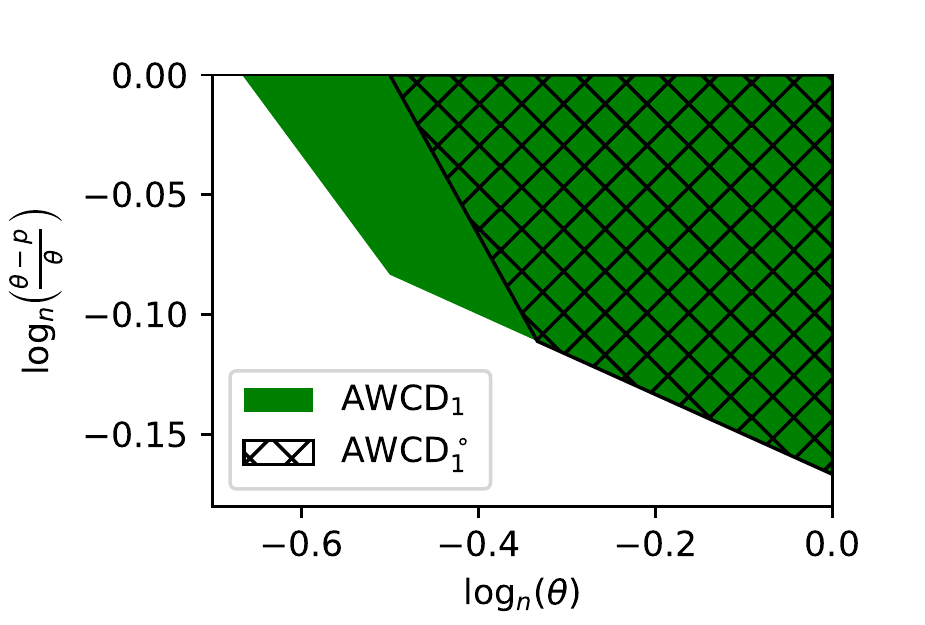}

   \caption{A log$_n$-log$_n$ plot of the regime of strong consistency for AWCD\(_1\) and AWCD\(_1^\circ\)}
   \label{fig_regime_k_1}
   \end{center}
\end{figure}

\subsection{Increase of starting neighborhood for sparse graphs}\label{subsec_larger_starting_guess}
Although the introduced modification improves the rate of strong consistency, the obtained rate is still far from optimal. As previously mentioned, the algorithm violates the likelihood-ratio test setup with respect to the assumptions on the starting guess. In particular, the local communities are relatively small and also contain many members from different communities. The following observation shows, that a better starting guess does indeed improve the rate further. 
\begin{obs} \label{obs_improved_starting_guess}
Suppose instead of \(W^{(0)}=\mathcal Y\) we are given an improved starting guess for each \(i\) in form of a random subset of its true community of the same size \(\asymp \theta n\) as before. We assume \(\frac{\theta}{\rho} \equiv C\) and \(\theta < \frac{1}{2}\). Then AWCD\(_1\) achieves consistent exact recovery after the first step as long as 
\[\theta \gg n^{-\frac{2}{3}} (\log n)^\frac{1}{3}.\]
If the size of the above starting guess is increased to \(\asymp \theta^\frac{1}{2} n\), then this rate improves to
\[\theta \gg (\log n)^\frac{1}{2} n^{-1}.\]
\end{obs}
Considering that the reconstruction of the community structure from such a good starting guess is trivial, it is surprising, that the rate from Theorem \ref{thm_original_AWCD} does not improve from increased precision of the starting guess (at least in the case \(\frac{\theta}{\rho} \equiv C\)). However, we observe the expected improvement after additionally increasing the size of the starting guess. The reason for this phenomenon is that for the computation of each test statistic \(T_{ij}\) we only use a part of the available data \(\mathcal Y\). The larger the starting guess is, the larger part of the data we use. If the size of the starting guess is too small, we cannot expect the algorithm to recover the community structure correctly, even if the starting guess is a subset of the true community.

As we will see in the following, the algorithm also benefits from an increased starting guess if the precision is not increased or even slightly worse. A simple way to increase the size of the starting guess is to take into account members that are connected via a path of minimal path length \(k\). We call this set of members \emph{\(k\)-neighborhood of \(i\)}  and denote it by \(\mathcal C_i^{k}\). Similarly, we can also work with the neighborhood \(\mathcal C_i^{\leq k}\) of members that are connected via a path of length \(k\) or smaller. If the graph is sparse enough, these neighborhoods are almost of the same size and we expect very similar results. To simplify some technical details, we will focus on using the neighborhood \(\mathcal C_i^{k}\).

\begin{figure}[]
\begin{minipage}{0.45\linewidth}
   \includegraphics[width = \linewidth]{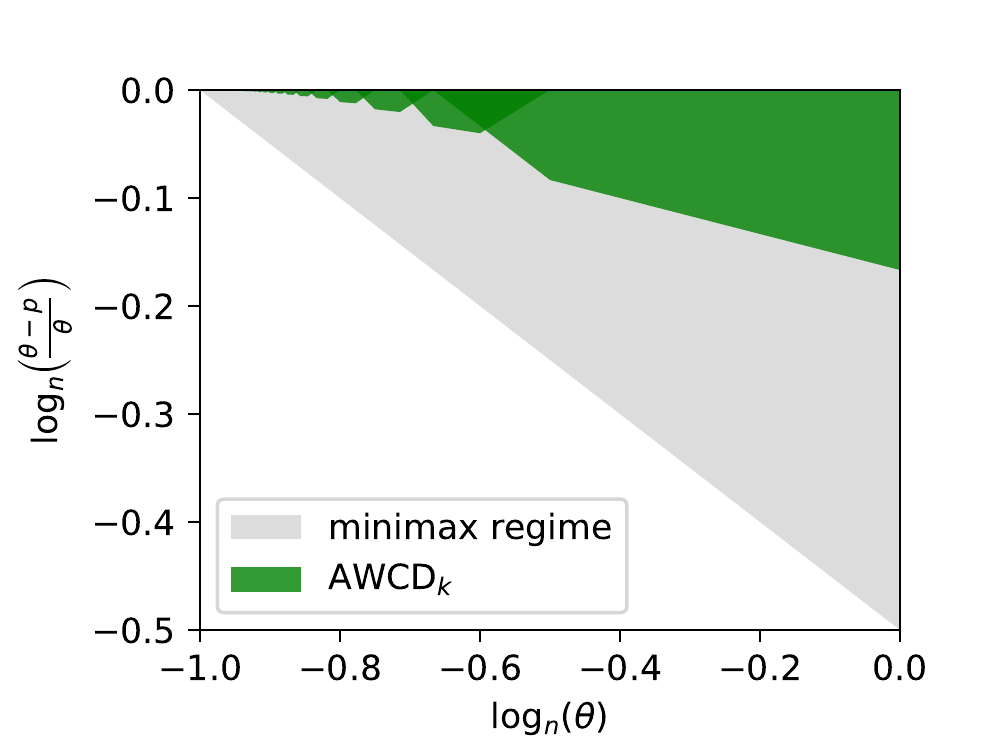}
\end{minipage}
\begin{minipage}{0.45\linewidth}
  \includegraphics[width = \linewidth ]{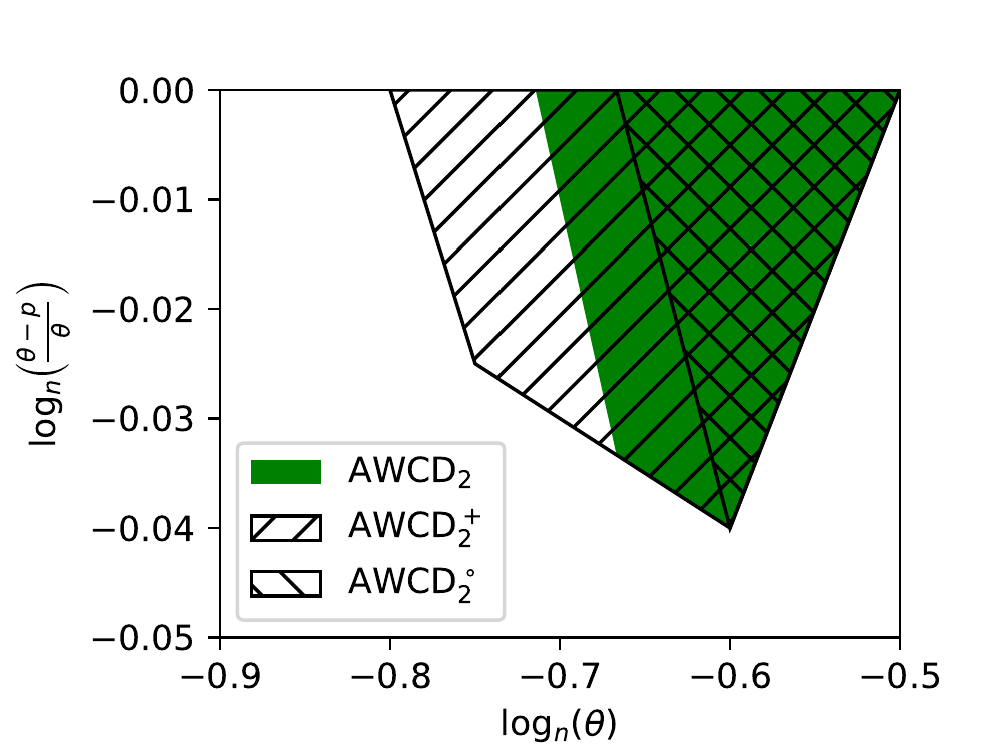}
\end{minipage}
\caption{Left: A log$_n$-log$_n$ plot of the strong consistency regime of AWCD$_k$ under a proper choice of the neighborhood parameter \(k\) in comparison to the minimax regime \cite{sbm_general_minimax}. Right: A log$_n$-log$_n$ plot compares the consistency regime of the three versions of the algorithm in case \(k=2\).}
\label{fig_regime_k>1}
\end{figure}

Moreover, an alternative but similar approach to increase the part of the data \(\mathcal Y\) used for the computation of test statistic \(T_{ij}\) is to instead (or additionally) change the way how we count the number of connections between two different local communities: Instead of only counting the edges that directly connect the two local communities, we can also count connections via paths of a certain (maximum) length. This approach should lead to very similar results to those presented in the following. 

Let us modify the definitions in the first step of algorithm \ref{algorithm} to
\begin{align*}
 S_{ij}  &\defined \sum_{v_1\in\mathcal C_i^k, v_2\in\mathcal C_j^k} \mathcal Y_{v_1 v_2} \text{ and}\\
 N_{ij} &\defined  |\mathcal C_i^k | \left( |\mathcal C_j^k| - \mathbbm1(i=j)\right).
\end{align*}
We will call this version of the algorithm AWCD\(_k^\circ\). Moreover, we also consider an analogous bias correction term as before:
\begin{align*}
 S_{ij}  &\defined \sum_{v_1\in\mathcal C_i^k, v_2\in\mathcal C_j^k} \mathcal Y_{v_1 v_2} - \mathbbm1(i=j)\left(|\mathcal C_i^k| + |\mathcal C_j^k|\right)
 \end{align*}
We denote the corresponding algorithm by AWCD\(_k\). We will show that the bias correction does significantly improve the rates. However, this is only an approximate bias correction. A completely unbiased version AWCD\(_k^+\) can for example be defined via
\begin{align*}
 S_{ij}  &\defined \sum_{v_1\in\mathcal C_i^k\setminus \mathcal C_j^{k-1}, v_2\in\mathcal C_j^k\setminus \mathcal C_i^{k-1}} \mathcal Y_{v_1 v_2}\text{ and} \\
 N_{ij} & \defined  \sum_{v_1\in\mathcal C_i^k\setminus \mathcal C_j^{k-1}\neq v_2\in\mathcal C_j^k\setminus \mathcal C_i^{k-1}} 1.
 \end{align*}
leading to even better rates. Unfortunately, we do not know if there exists an efficient implementation. Conversely, the other two versions can be implemented efficiently using matrix multiplications: The starting neighborhood guesses $C_i^k$ can be represented by the weight matrix 
  \[W^{(0)} = \left(\mathbbm 1\left(\left(\sum_{l=1}^k \mathcal Y^l \right)_{ij}  > 0\right) - \mathbbm 1\left(\left(\sum_{l=1}^{k-1} \mathcal Y^l \right)_{ij}  > 0\right)\right)_{i, j}\]
  allowing us to compute \(S\) as previously as a product of matrices \(S=W^{(0)}\mathcal YW^{(0)}\). Computation of \(N\) and the bias correction only requires \(\mathcal O(n^2)\) operations. Thus, the overall complexity is the same as the complexity of the involved (sparse) matrix multiplications. 
Consequently, AWCD\(_k\) is the main algorithm that we propose in this paper. The other two versions are just considered for the sake of comparison.

This extension of the algorithm adds a tuning parameter \(k\) to the procedure. However, this does not increase the parameter space much, as \(k\) will only take very few different values in practice, e.g. \(k\in\{1, 2,3, 4\}\). Moreover, it may be estimated from the sparsity of the given adjacency matrix. 

\cite{AWCD} demonstrated, that starting from the 1-neighborhood, multiple iterations of the algorithm can improve the quality of the result further if the first step already produces a reasonably good estimation of the community structure. This idea of using the output of the previous iteration as a new starting guess can of course also be applied to the algorithms introduced above. However, this approach will not improve the rate any further, see section \ref{sec_experiments}.

In the following theorem, we collect the rates of strong consistency for the cases $\frac{\theta}{\rho}\equiv C$ and $\frac{\theta}{\rho} \leq C$ for each of the three algorithms introduced above. The rates are visualized in a log$_n$-log$_n$ plot in Figure \ref{fig_regime_k>1}. 
\begin{thm}\label{thm_consistency_AWCD_general_k}
Suppose \(k\geq 2\). Consistent exact recovery is achieved on SBM($n, K, \theta, \rho$) after the first step as long as \\

\renewcommand{\arraystretch}{1.2}
\begin{tabular}[]{|l|l|l|}
\hline
 & \(\frac{\theta}{\rho} \equiv C\)  & \(\frac{\theta}{\rho} \leq C \)\\
\hline
AWCD\(_k^\circ\) & \(n^{-\frac{k}{k+1}} \ll \theta \ll n^{-\frac{k-1}{k}} \) &\(\theta - \rho \gg \max \{ C_1, C_2^\circ\}\)  \\
 AWCD\(_k\)&  \( n^{-\frac{2k+1}{2k+3}} (\log n)^\frac{1}{2k+3} \ll \theta \ll n^{-\frac{k-1}{k}}\) &\(\theta - \rho \gg \max \{ C_1, C_2\}\) \\
  AWCD\(_k^+\)&   \(n^{-\frac{2k}{2k+1}} (\log n)^\frac{1}{2k+1} \ll \theta \ll n^{-\frac{k-1}{k}}\) &\(\theta - \rho \gg \max \{ C_1, C_2^+\}\) \\
  \hline
\end{tabular}\\

for
\begin{align*}
 C_1 &= \theta^\frac{3k+1}{2k+1} n^\frac{k-1}{2k+1} + \theta^\frac{4k+1}{4k+2}n^{-\frac{1}{4k+2}} (\log n)^\frac{1}{4k+2},\\
 C_2^\circ &=\theta^\frac{k}{2k+1}n^{-\frac{k}{2k+1}},\\
 C_2&= \theta^\frac{2k-1}{4k+2}n^{-\frac{1}{2}} (\log n)^\frac{1}{4k+2}\text{ and}\\
 C_2^+ &= \theta^\frac{1}{2}n^{-\frac{k}{2k+1}} (\log n)^\frac{1}{4k+2}.
\end{align*}
\end{thm}
\begin{rmk}
 In particular, if \(k\) is tuned correctly, the algorithm AWCD$_k$ nearly achieves the optimal rate of strong consistency for the case \(\frac{\theta}{\rho}\equiv C\) as the regimes for different \(k\) are overlapping. 
\end{rmk}

\subsection{General SBM}\label{subsec_general_sbm}
We suspect that any of the above results may be generalized to an SBM having communities of different sizes and more Bernoulli parameters. Indeed, continuing the study of \(\text{AWCD}_1\), we can extend the results of Theorem \ref{thm_original_AWCD} to more general stochastic block models.

\begin{prop}\label{Proposition_different_block_size}
We consider the SBM with communities of different sizes \(n_1, \dots, n_K\) and only two Bernoulli parameters \(\theta > \rho\) as before. We denote the minimal and maximal community size by \(n_{\min}\) and \(n_{\max}\). Then \(\text{AWCD}_1\) achieves consistent exact recovery after the first step (under the asymptotics \(n_{\min}\to\infty\) and \(\frac{\theta}{\rho}\leq C\)) as long as 
\[\theta - \rho\gg \left(\frac{n_{\max}}{n_{\min}}\right)^\frac{2}{3}\max\{n_{\min}^{-\frac{1}{3}} \theta^\frac{1}{2}(\log n_{\min})^\frac{1}{6},   n_{\min}^{-\frac{1}{6}} \theta^\frac{5}{6} (\log n_{\min})^\frac{1}{6}\}\]
and \(\theta < \frac{1}{2}\) (or any constant smaller than 1).
\end{prop}
\begin{rmk}
The lower bound no longer simplifies to a single term in the case \(\frac{\theta}{\rho}\equiv C\) but rather to 
\[\theta \gg \max\left\{\left(\frac{n_{\max}}{n_{\min}}\right)^\frac{4}{3}n_{\min}^{-\frac{2}{3}} (\log n_{\min})^\frac{1}{3}, \left(\frac{n_{\max}}{n_{\min}}\right)^4  n_{\min}^{-1} \log n_{\min}\right\}.\]
Only for \(n_{\max}^{10} \lesssim n_{\min}^{11}\) does the condition further simplify to
\[\theta \gg \left(\frac{n_{\max}}{n_{\min}}\right)^\frac{4}{3}n_{\min}^{-\frac{2}{3}} \log n_{\min},\]
similarly to the rate of Theorem \ref{thm_original_AWCD}.
\end{rmk}

\begin{prop}\label{Proposition_different_block_size_and_parameter}
 We consider a stochastic block model with two blocks of block sizes \(n_1\) and \(n_2\) with parameters \(\theta_1, \theta_2 > \rho\) under the asymptotics \(n_{\min} \to \infty\). Then \(AWCD_1\) achieves consistent exact recovery after the first step as long as
 \[\theta_{\min} - \rho\gg \left(\frac{\theta_{\max}}{\theta_{\min}}\right)^\frac{7}{3}\left(\frac{n_{\max}}{n_{\min}}\right)^2\max\left\{n_{\min}^{-\frac{1}{3}} \theta_{\max}^\frac{1}{2}(\log n_{\min})^\frac{1}{6},   n_{\min}^{-\frac{1}{6}} \theta_{\max}^\frac{5}{6} (\log n_{\min})^\frac{1}{6}\right\}\] 
 and \(\theta_{\max}<\frac{1}{2}\)(or any constant smaller than 1).
\end{prop}
\begin{rmk}
Note that we do in general allow for \(\frac{n_{\max}}{n_{\min}} \to \infty\) or \(\frac{\theta_{\max}}{\theta_{\min}}\to\infty\). If  \(\frac{n_{\max}}{n_{\min}} \) and \(\frac{\theta_{\max}}{\theta_{\min}}\) are bounded by a constant \(M\), then the rates in Proposition \ref{Proposition_different_block_size} and \ref{Proposition_different_block_size_and_parameter} are exactly the same as in Theorem \ref{thm_original_AWCD}. However, the respective implicit constant factor in the lower bound depends on \(M\).
\end{rmk}

\section{Experiments}\label{sec_experiments}

\begin{figure}
\begin{minipage}{0.32\linewidth}
   \includegraphics[width = \linewidth]{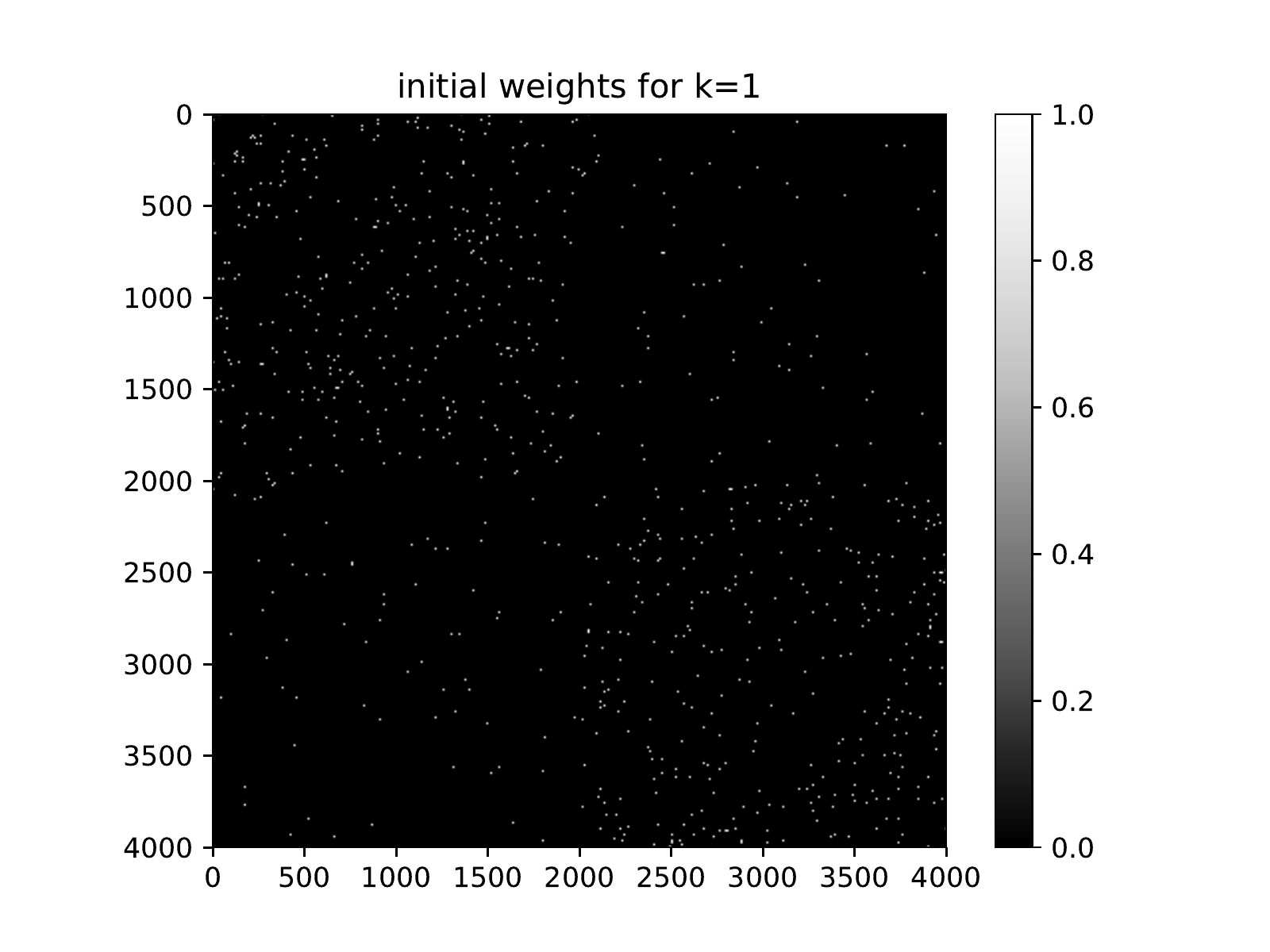}
\end{minipage}
\begin{minipage}{0.32\linewidth}
   \includegraphics[width = \linewidth]{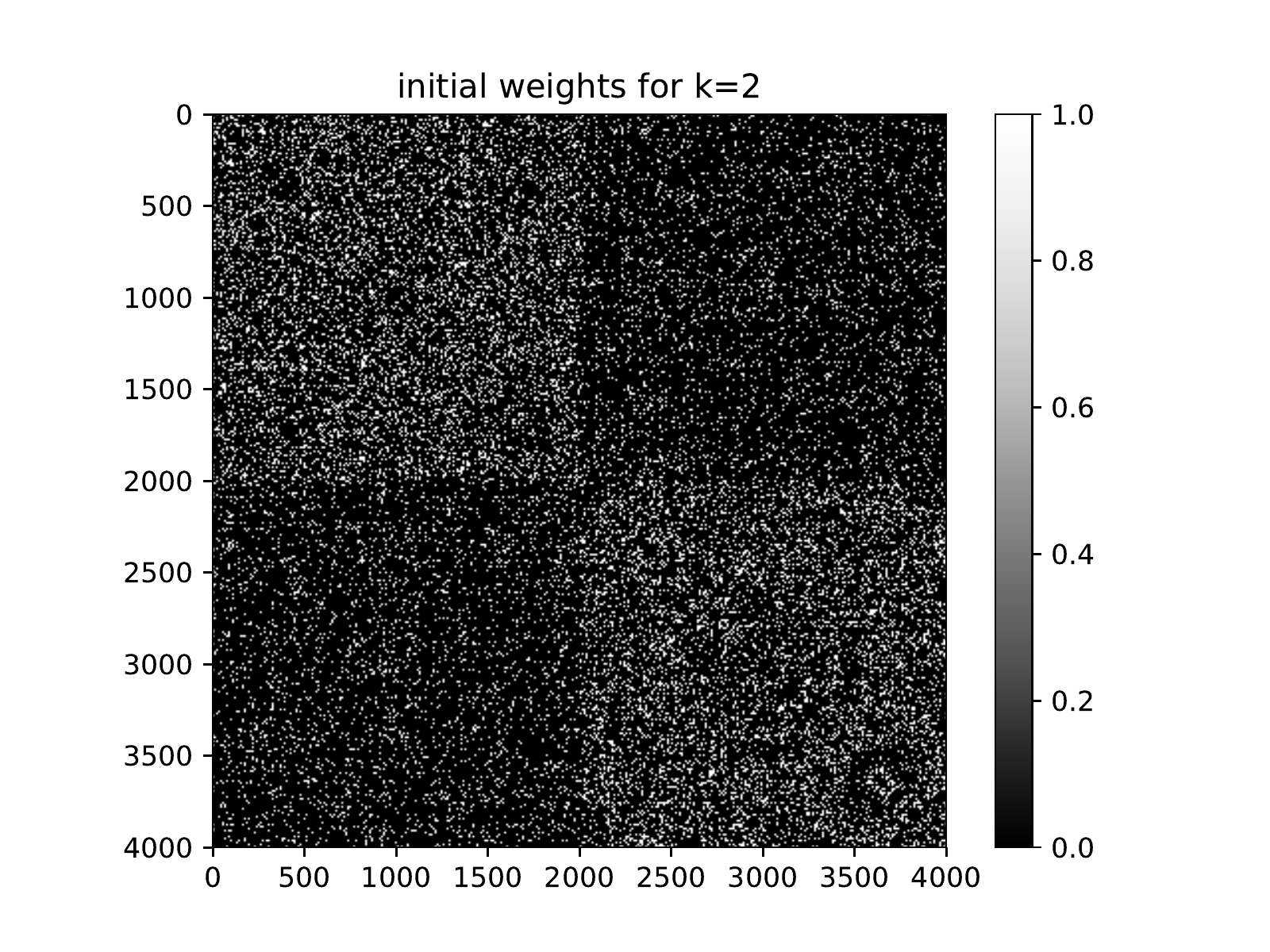}
\end{minipage}
\begin{minipage}{0.32\linewidth}
   \includegraphics[width = \linewidth]{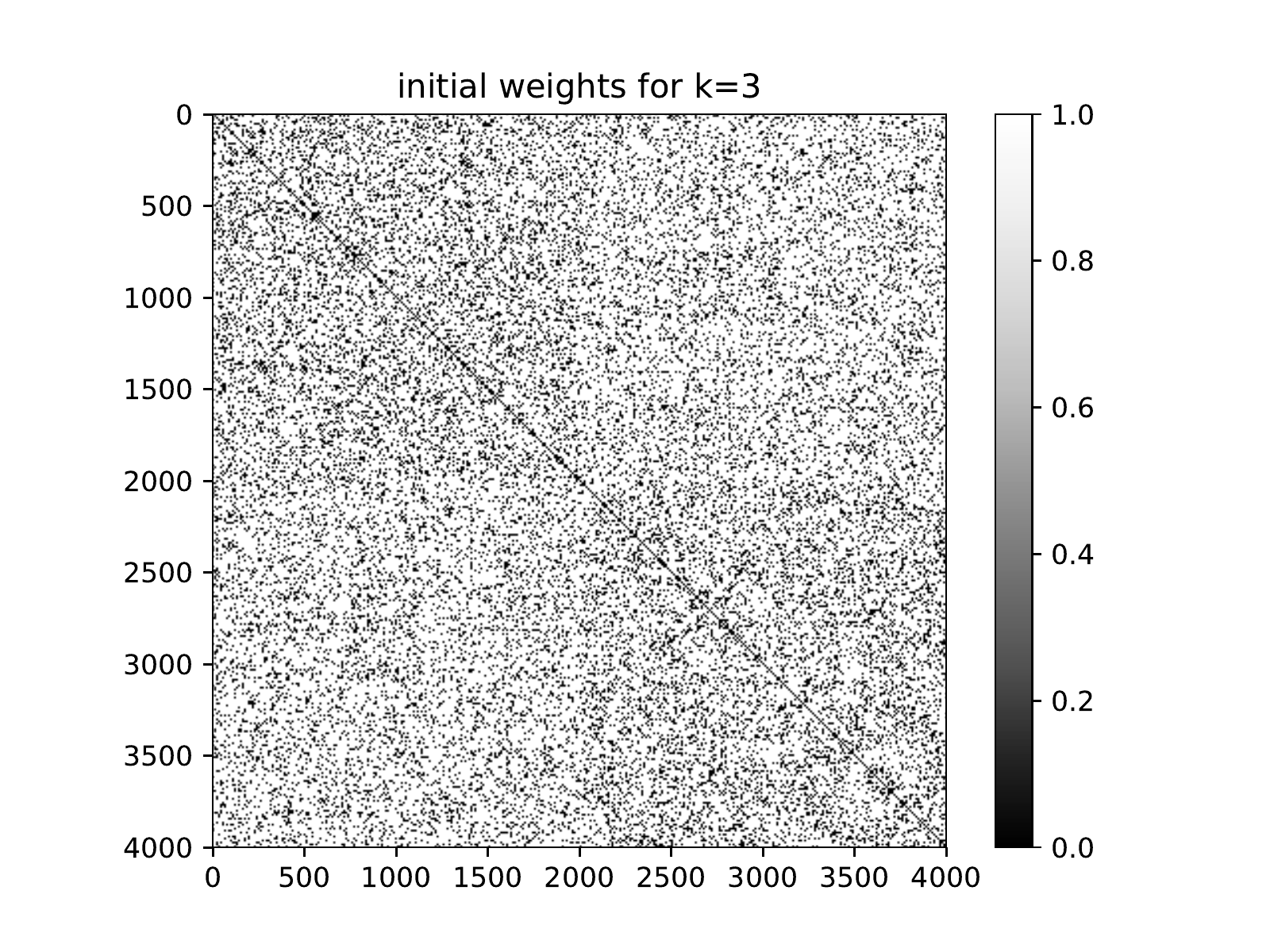}
\end{minipage} \\
\begin{minipage}{0.32\linewidth}
   \includegraphics[width = \linewidth]{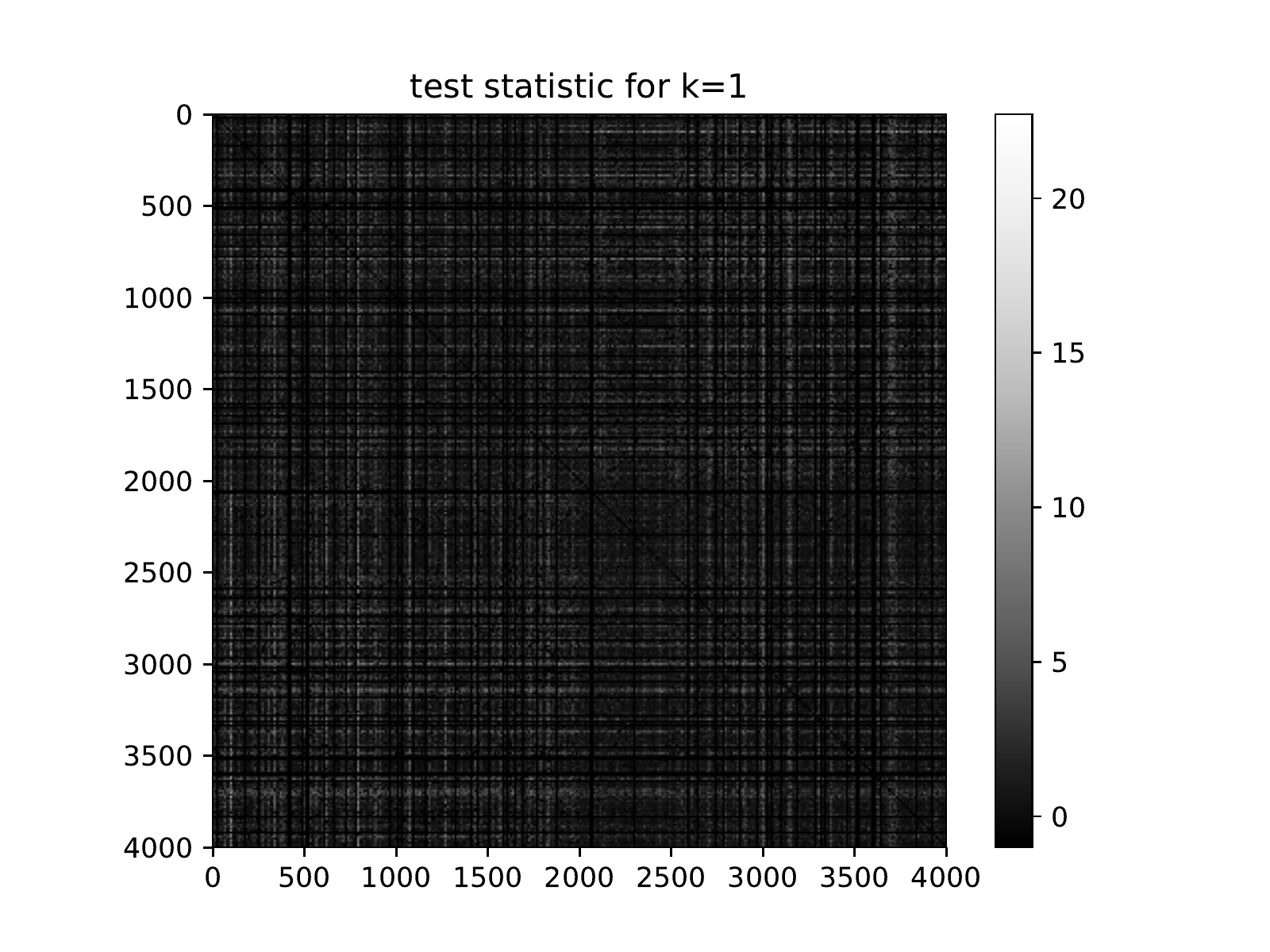}
\end{minipage}
\begin{minipage}{0.32\linewidth}
   \includegraphics[width = \linewidth]{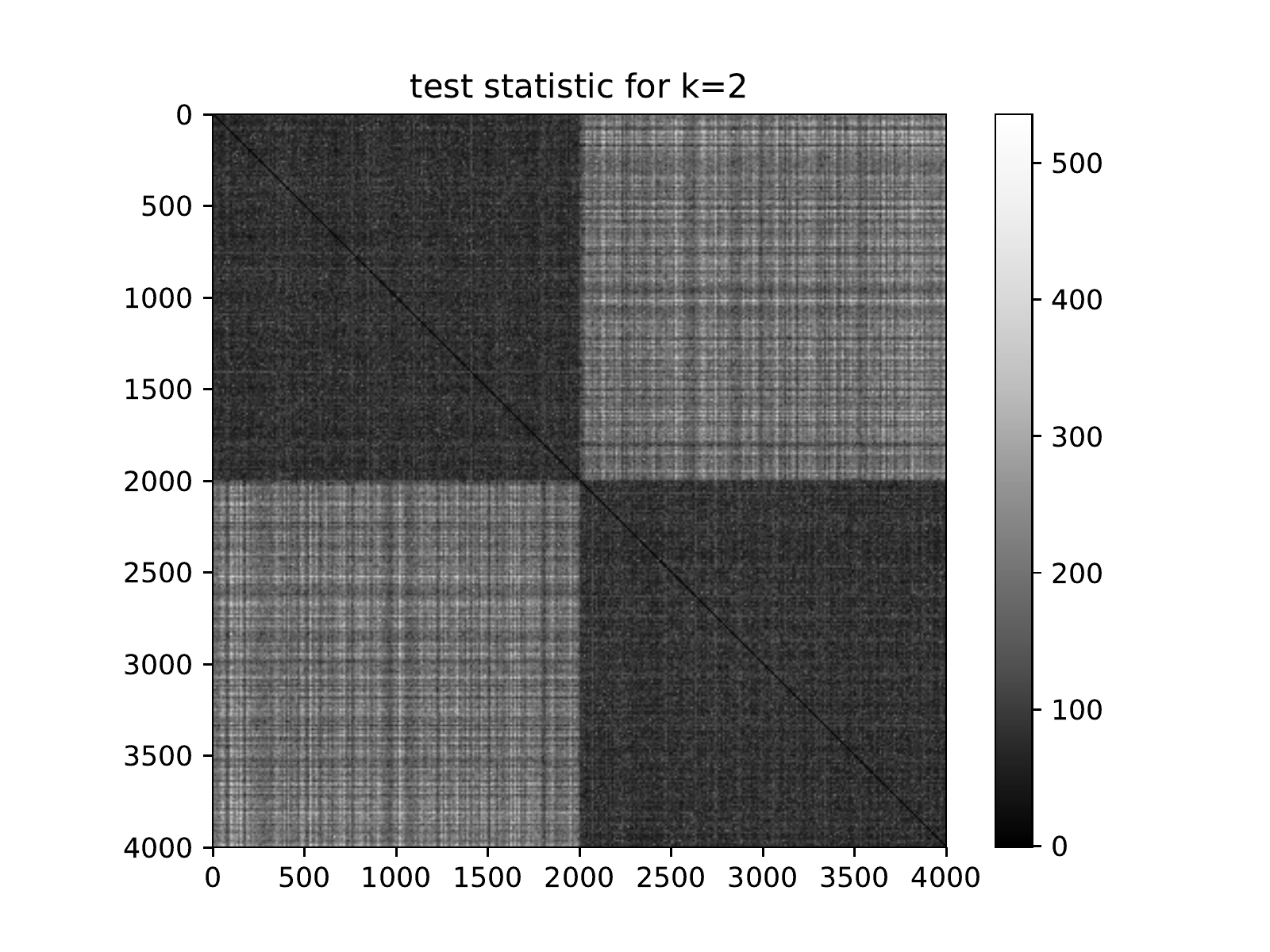}
\end{minipage}
\begin{minipage}{0.32\linewidth}
   \includegraphics[width = \linewidth]{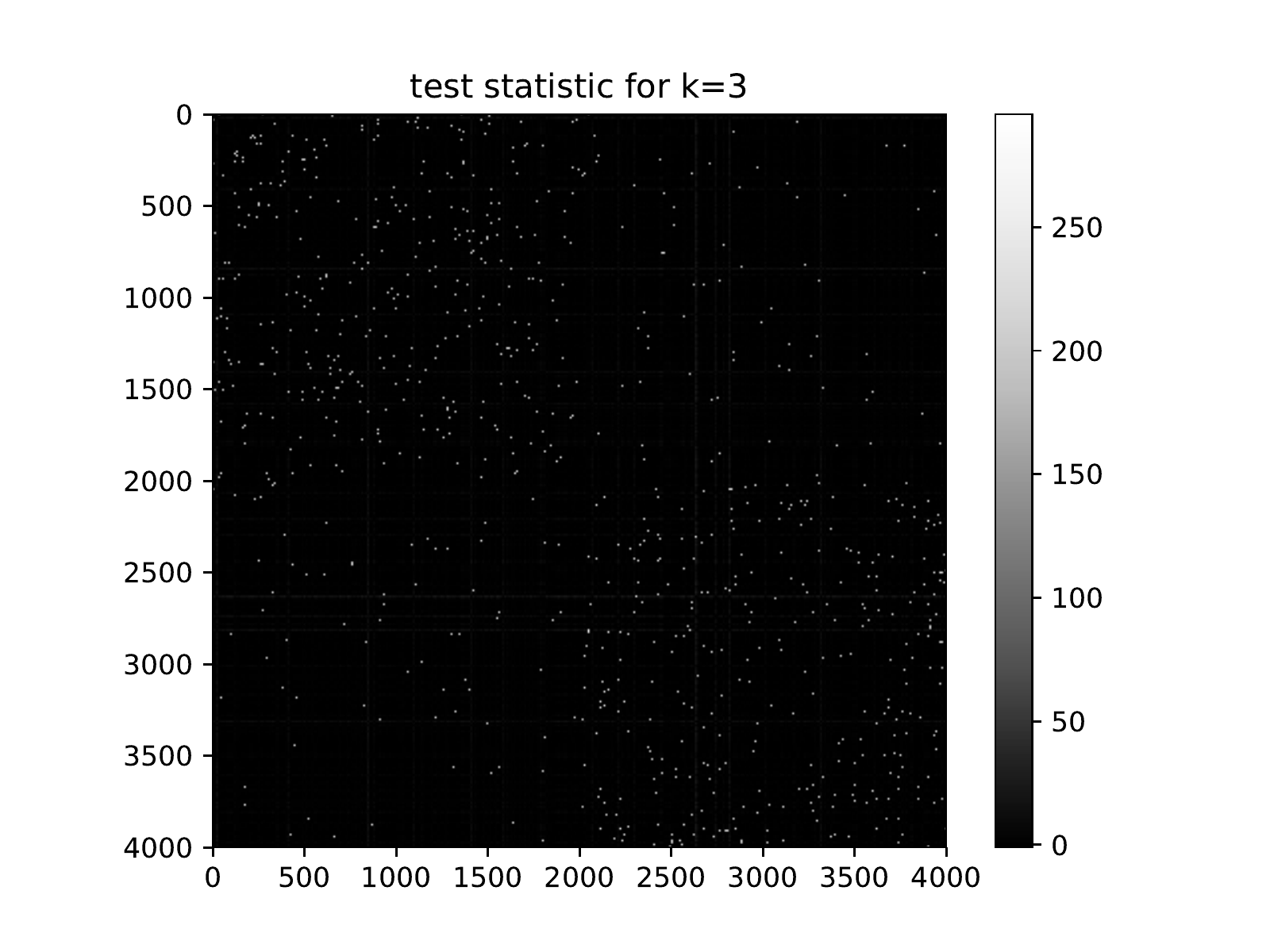}
\end{minipage}
\caption{We consider a realization of SBM(2000, 2, 0.01, 0.0025). Top: Adjacency matrix of the starting guess $W^{(0)}$ of the community structure from AWCD$_k$  for $k=1, 2, 3$. Bottom: The corresponding test statistics.}
\label{fig_matrices}
\end{figure}

\cite{AWCD} already demonstrated a state-of-the-art performance of the original algorithm on the LFR benchmark \cite{benchmark_sbm1} for a total sample size of \(n=1000\) as well as on smaller real-life examples such as the famous Zachary's karate club network \cite{karate_club}. Moreover, it has been shown that optimizing modularity is a reasonable method to choose the tuning parameter $\lambda$. In this section, we focus on validating our theoretical results and consider only the simple symmetric stochastic block model SBM($n, 2, \theta, \rho$) consisting of two communities of identical size \(n\) and having two Bernoulli parameters $\theta > \rho$. We want to study two questions:\begin{enumerate}
                                                                                                                                                                                                                                                                                                                                                                                                                                                                                                                 \item How does the performance of the procedure depend on the starting neighborhood via the parameter \(k\)?
                                                                                                                                                                                                                                                                                                                                                                                                                                                                                                                 \item Can we observe the stated rates of consistency?
                                                                                                                                                                                                                                                                                                                                                                                                                                                                                                                                                                                                                                                                                                                                                                                \end{enumerate}

\begin{figure}
\begin{minipage}{0.32\linewidth}
   \includegraphics[width = \linewidth]{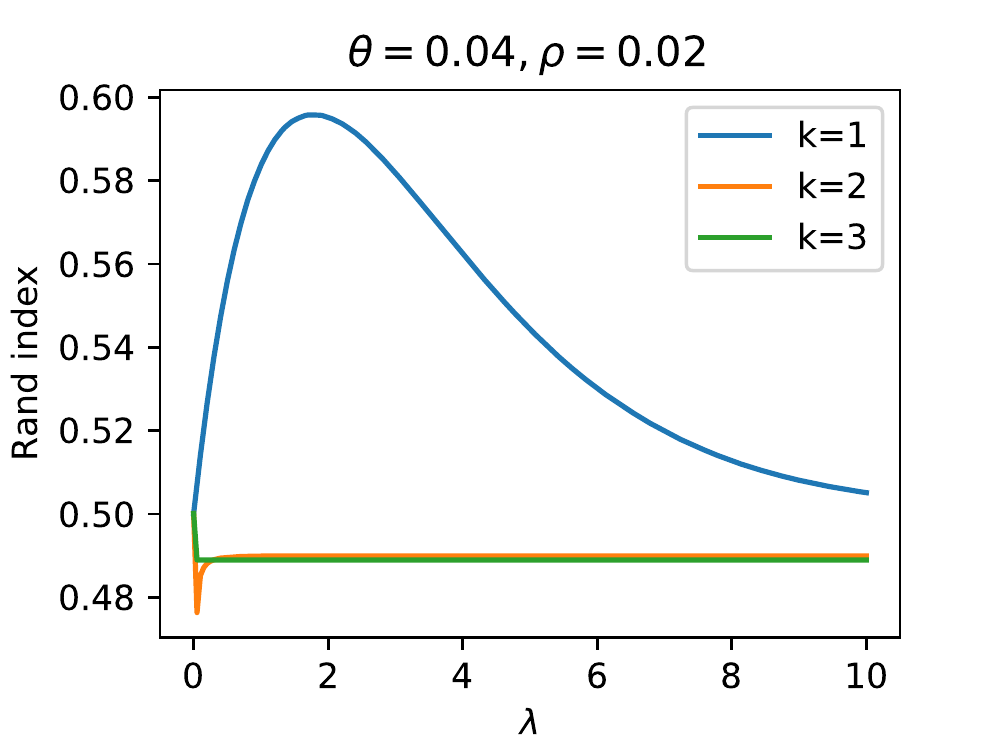}
\end{minipage}
\begin{minipage}{0.32\linewidth}
   \includegraphics[width = \linewidth]{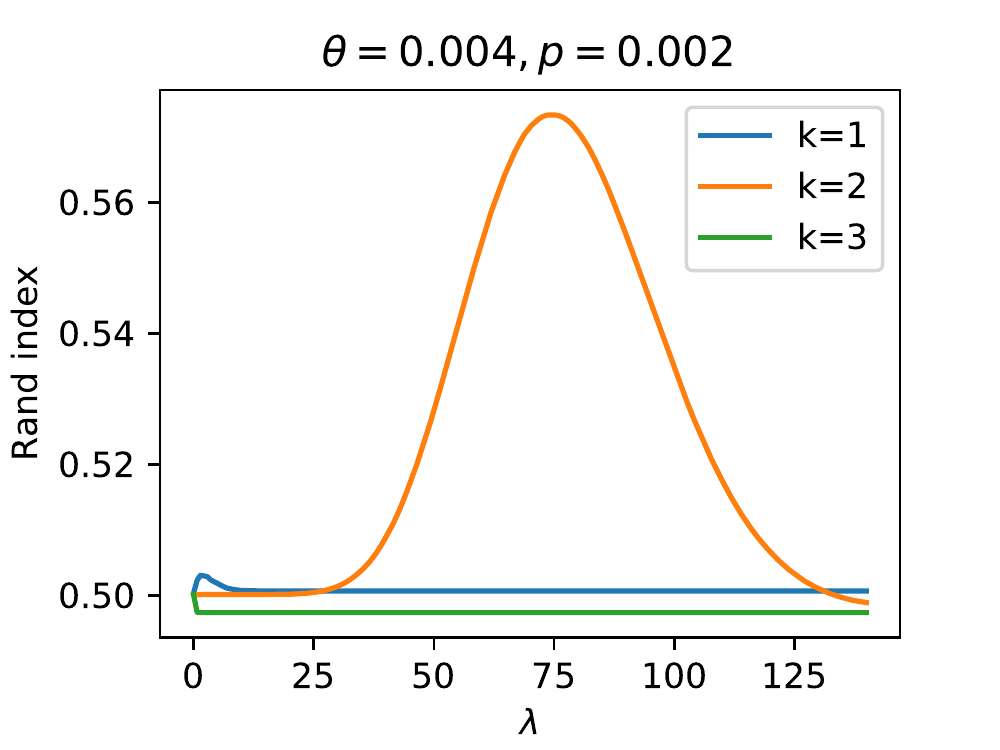}
\end{minipage}
\begin{minipage}{0.32\linewidth}
   \includegraphics[width = \linewidth]{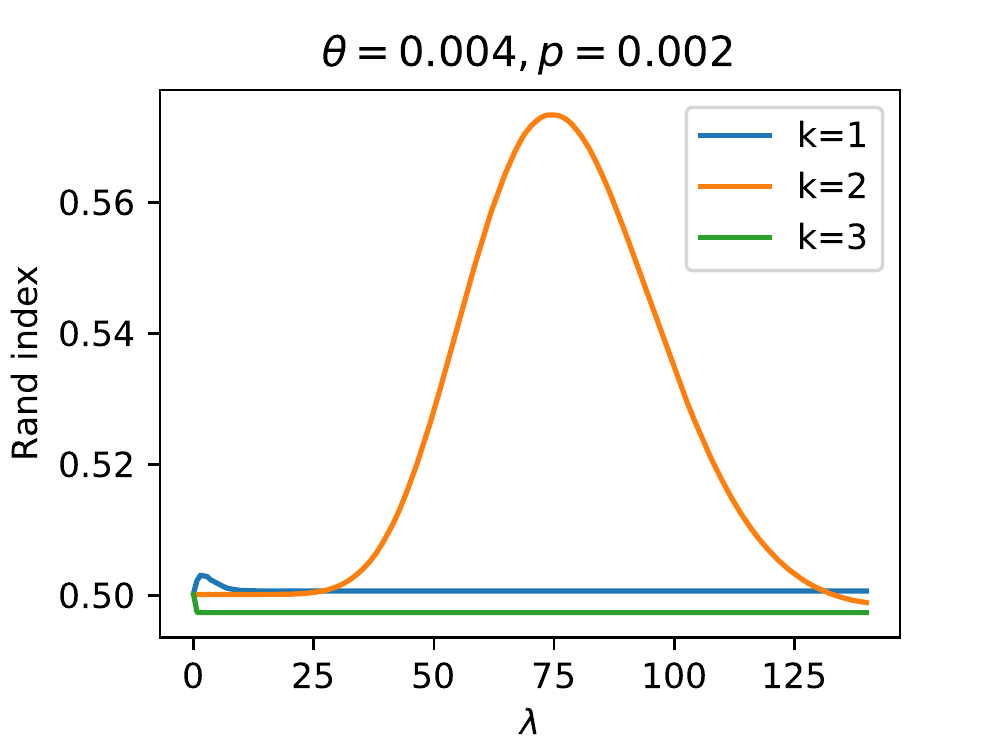}
\end{minipage} \\
\begin{minipage}{0.32\linewidth}
   \includegraphics[width = \linewidth]{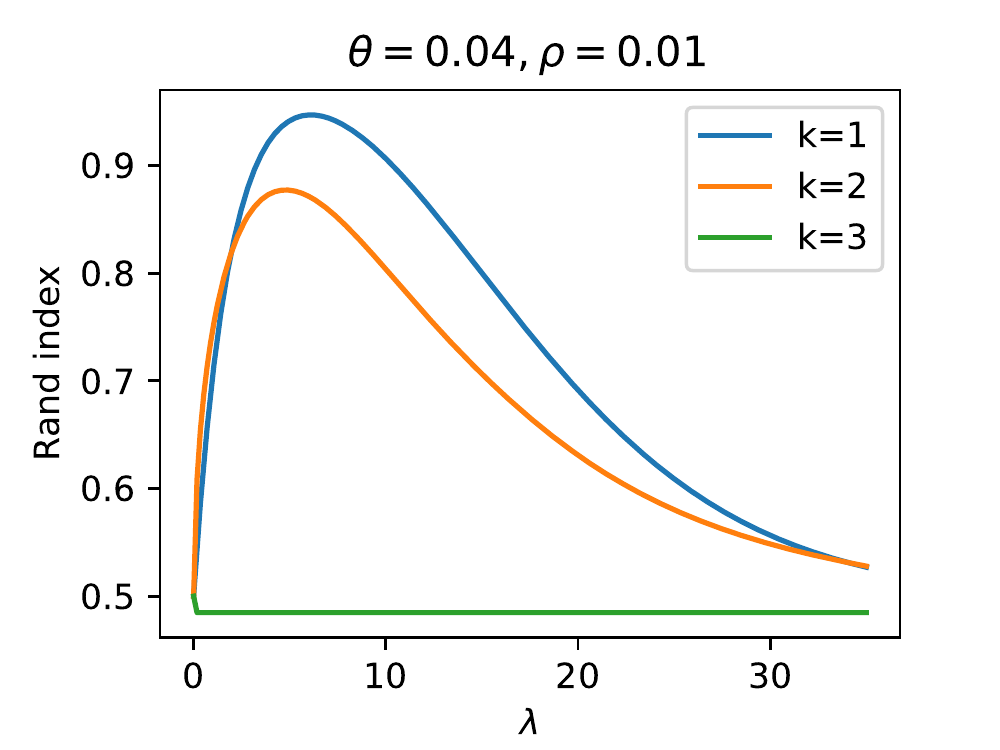}
\end{minipage}
\begin{minipage}{0.32\linewidth}
   \includegraphics[width = \linewidth]{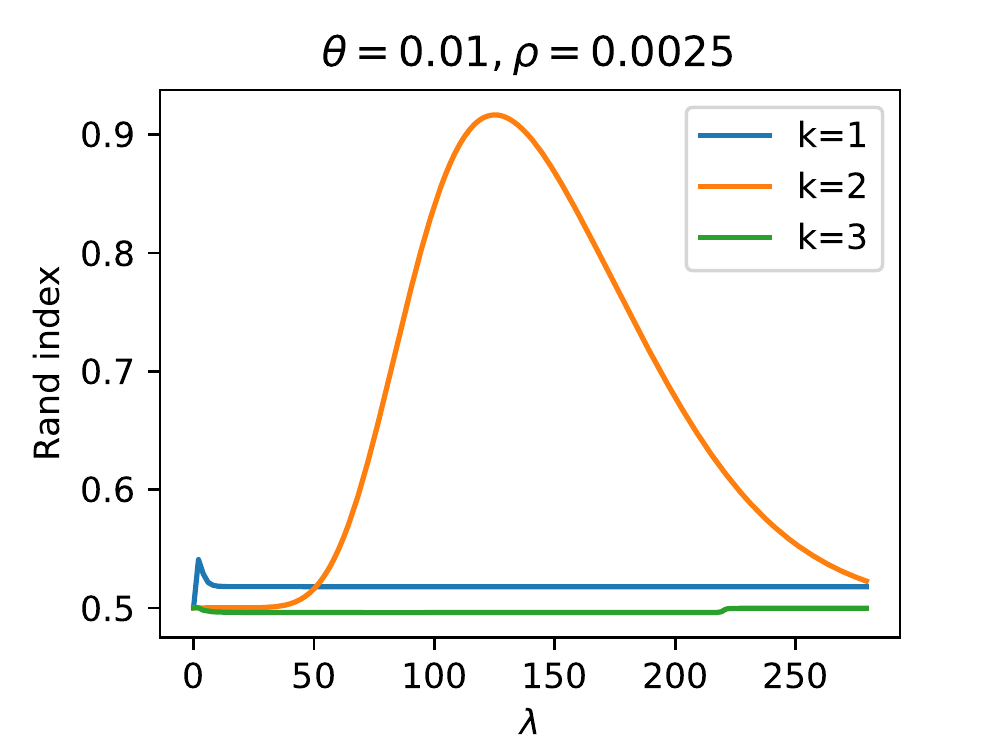}
\end{minipage}
\begin{minipage}{0.32\linewidth}
   \includegraphics[width = \linewidth]{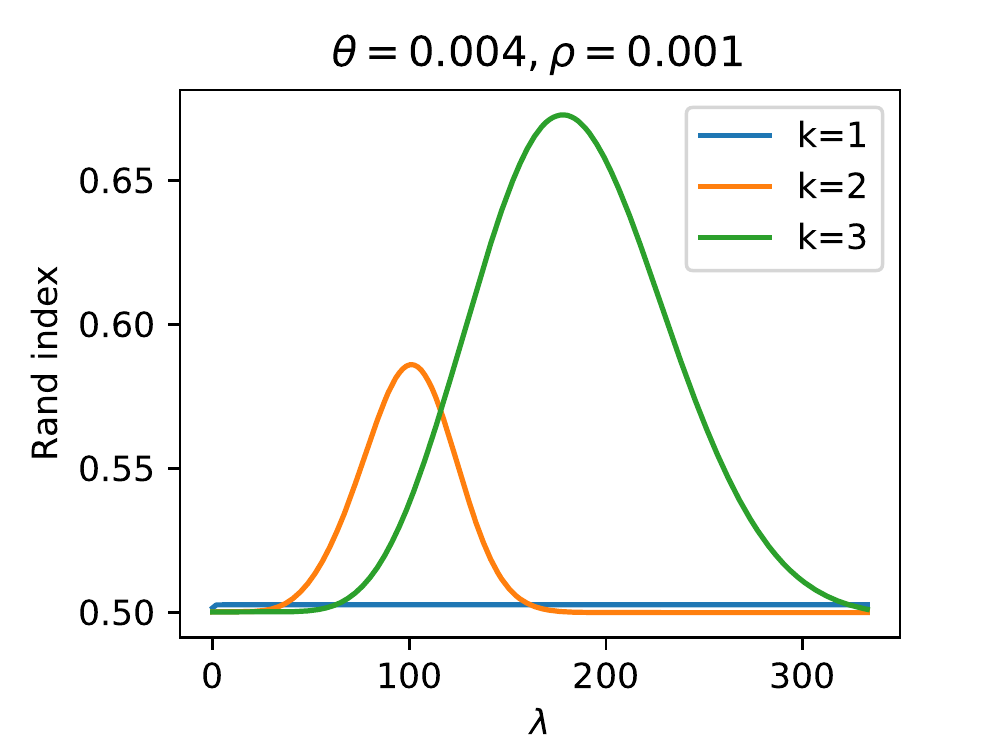}
\end{minipage}
\caption{Rand index achieved by AWCD$_k$ on SBM($2000, 2, \theta, \rho$) for different algorithm parameters $(k, \lambda)$ and for several Bernoulli parameter combinations $(\theta, \rho)$}
\label{fig_performance_different_k}
\end{figure}

We do not study the effect of the bias correction term. This aspect of the algorithm is important for strong consistency. However, in terms of the accuracy of the final weight matrix, also known as \emph{Rand index} \cite{rand}, the effect is rather small: If the matrix is relatively dense, then the bias term is also relatively small, whereas the bias correction only applies to a relatively small fraction of the test statistics $T_{ij}$, if the matrix is sparse.

We start by studying one realization of the SBM as described above with the parameters $n=2000$, $\theta = 0.01$ and $\rho = 0.0025$. In Figure \ref{fig_matrices} we see the original adjacency matrix, the corresponding 2- and 3-neighborhood adjacency matrix as well as the test statistics of the algorithm AWCD$_k$ for $k=1, 2, 3$. We can see that the test statistic is only informative for $k=2$: For $k=1$, each sum $S_{ij}$ consists only of a few summands, so the test statistic cannot be reliable, whereas, for $k=3$, each starting neighborhood contains almost all members of the network. This is confirmed by a plot of the Rand index of the final weight matrices corresponding to these test matrices and different tuning parameters \(\lambda\) in Figure \ref{fig_performance_different_k} (bottom, middle): Only for $k=2$ does the algorithm correctly identify the community structure. \cite{AWCD} have demonstrated in the case $k=1$ that multiple iterations of the algorithm can improve the final weight matrix further if enough information on the community structure is recovered after the first step. However, in this example, we do not benefit from these additional iterations, as the data is too sparse and the weight matrix after the first step is not informative enough. This is demonstrated in Figure \ref{fig_many_rounds}.

We repeated the above experiment for several other parameter combinations $(\theta, \rho)$ and plotted the resulting Rand indices in Figure \ref{fig_performance_different_k}. We can see that the performance of the algorithm is stable with respect to the tuning parameter $\lambda$. Moreover, for any parameter combination of the SBM, we find parameters $k$ and $\lambda$ which recover the community structure with a non-negligible accuracy. As expected by Theorem \ref{thm_consistency_AWCD_general_k}, increasing the sparsity of the data significantly enough requires also increasing the parameter $k$. Note that the overall sparsity depends on both parameters $\theta$ and $\rho$: In the case of $\theta = 0.004$ we see on the right-hand side of Figure \ref{fig_performance_different_k} that the optimal parameter $k$ depends on $\rho$. Unsurprisingly, we also observe that for a given $\theta$, an increase of the quotient $\frac{\theta}{\rho}$ allows a more accurate recovery of the community structure.

\begin{figure}
\begin{center}
    \includegraphics[scale = 0.55]{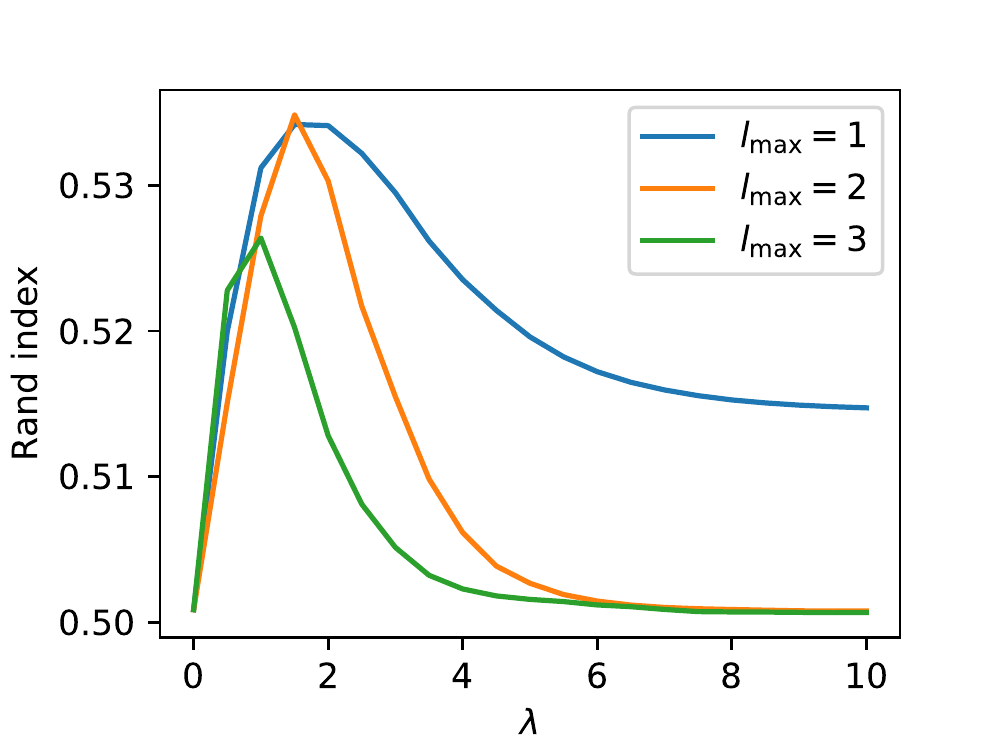}
    \caption{Rand index achieved after the first three iterations of algorithm \ref{algorithm} for a realization of SBM($2000, 2, 0.01, 0.0025$)}
\label{fig_many_rounds}
\end{center}
\end{figure}

Lastly, we discuss the second question raised above: Can we observe the stated rates of consistency in numerical experiments? To minimize the computation time, we will focus on the case $k=1$. We have repeated the above experiment for many different parameter combinations $(n, \theta)$ with a fixed quotient $\frac{\theta}{\rho}\equiv 4$. For each parameter combination, we computed the maximum Rand index that can be achieved by optimizing the tuning parameter $\lambda$. We repeated this experiment ten times for each parameter combination with \(n\leq 2000\) and averaged the resulting maximum Rand index. The results are shown in Figure \ref{fig_experiment_rate}. According to Theorem \ref{thm_original_AWCD}, AWCD$_1$ achieves strong consistency at a rate of $n^{-\frac{2}{3}}$ up to logarithmic factors. Indeed, it seems that the function yielding the minimum $\theta$ necessary to achieve a certain level of accuracy for a given community size $n$ can be reasonably well approximated by a function of the form $Cn^{-\frac{2}{3}}$.

\begin{figure}
\begin{minipage}{0.49\linewidth}
   \includegraphics[width = \linewidth]{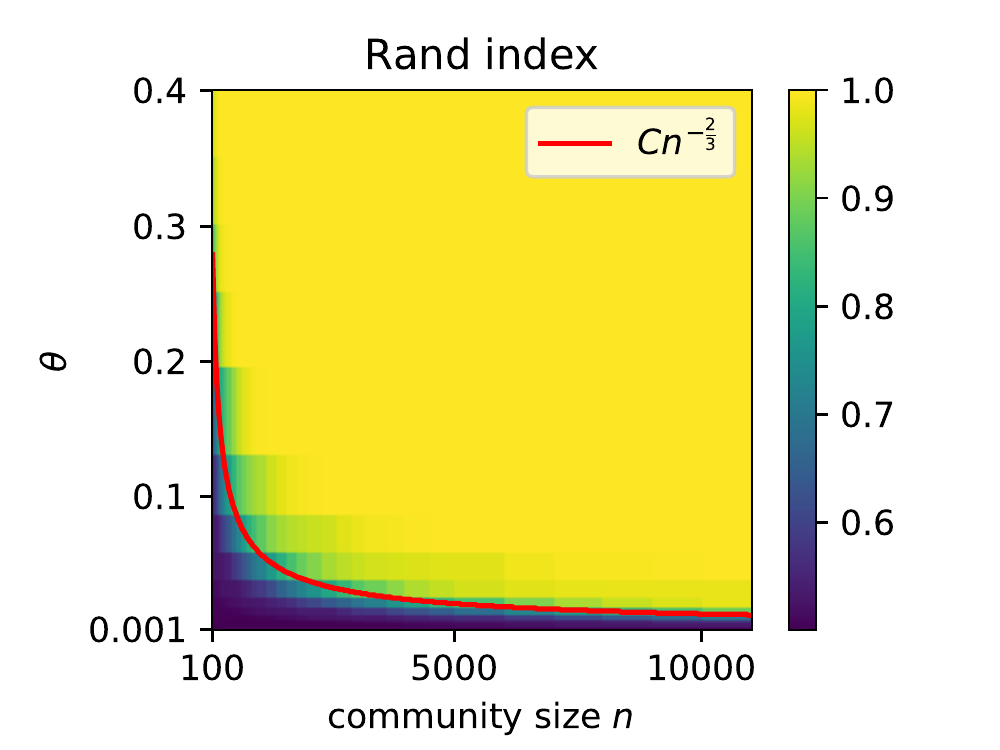}
\end{minipage}
\begin{minipage}{0.49\linewidth}
   \includegraphics[width = \linewidth]{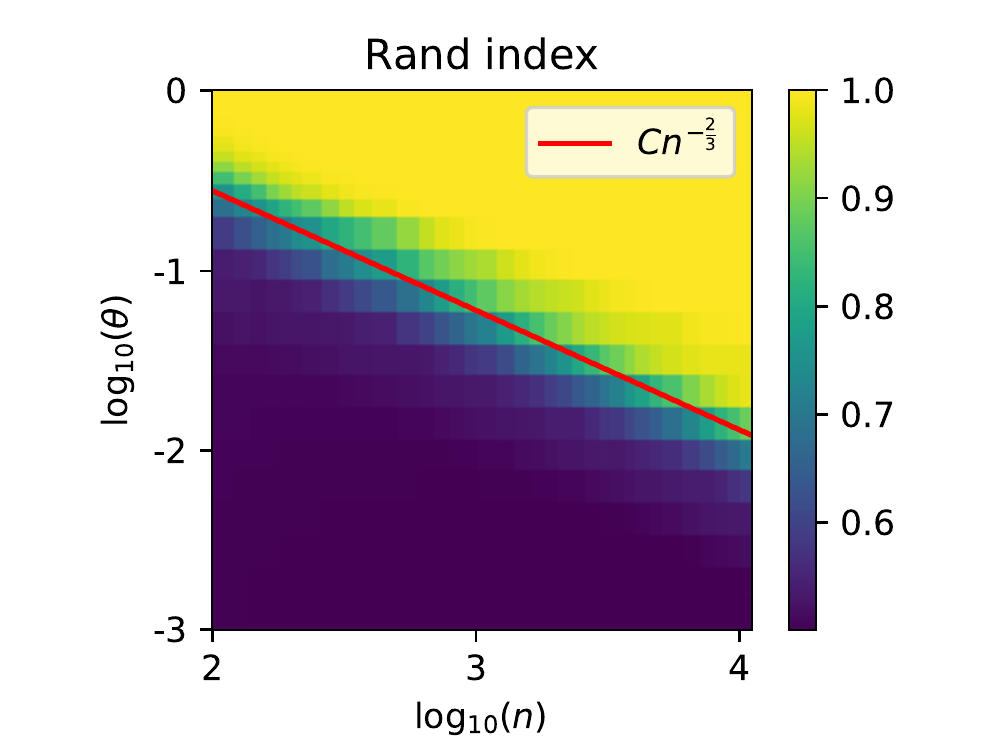}
\end{minipage}
\caption{Two plots visualize the Rand index achieved by AWCD$_1$ for different SBM parameters $(n, \theta)$ with $\frac{\theta}{\rho}\equiv 4$ using the optimal tuning parameter $\lambda$. The right plot only differs in the logarithmic scale on both axes.}
\label{fig_experiment_rate}
\end{figure}


\section{Proofs}\label{sec_proofs}
\begin{lem}\label{lem_bernstein_homogeneous}
 For \(X\sim \text{Binom}(m, q)\) and \(a>0\) we have
\begin{align*}
 \mathbb P\left(|X-\mathbb E X| \geq am \right) \leq 2 \exp\left(-\frac{a^2m}{2\left(q + \frac{a}{3}\right)}\right)\text{.}
\end{align*}
\end{lem}
\begin{proof}
 This is a special case of Bernstein's inequality for bounded variables \cite[Theorem 2.8.4]{bernstein_ineq}.
\end{proof}

 \begin{lem}\label{lem_bernstein}
Suppose \(X_i\) are independent Bernoulli variables of mean \(\theta_i \leq \theta\) and \(a>0\). Then for \(c_i\in\{1, 2\}\) and \(m\defined \sum_{i=1}^N c_i\) we have
\begin{align*}
 \mathbb P\left(\left|\sum_{i=1}^N c_i X_i  -\mathbb E c_ i X_i\right| \geq am \right) \leq 2 \exp\left(-\frac{a^2m}{6\theta + \frac{4a}{3}}\right)\text{.}
 \end{align*}
 \end{lem}
 \begin{proof}
This is a special case of Bernstein's inequality for bounded variables \cite[Theorem 2.8.4]{bernstein_ineq}.
 \end{proof}

\begin{proof}[Proof of Theorem \ref{thm_original_AWCD}]
Let us first consider the AWC\(^D\) algorithm. 
Suppose \(\epsilon > 0\) and \(\alpha\in\{1, 2, \dots, K\}\). We will denote the corresponding true community by \(\mathcal C_\alpha^*\). According to Lemma \ref{lem_bernstein_homogeneous} we have
\begin{align} |\mathcal C_i\cap \mathcal C^*_\alpha| = \theta_i^\alpha n + \mathcal O\left(\epsilon n + \theta \right)   \label{eqn_size_one_nbhd}\end{align}
for \begin{align}\theta_i^\alpha = \begin{cases}
                         \theta &, i \in \mathcal C^*_\alpha\\
                         \rho &, i\notin \mathcal C^*_\alpha
                        \end{cases}
\nonumber
\end{align}
on an event of probability at least
\[1 - 2\exp \left( - \frac{\epsilon^2 (n - 2)}{2\theta + \epsilon}\right)\text{.}\]
Similarly, we conclude for \(i\neq j\) 
\begin{align}
|\mathcal C_i \cap \mathcal C_j|= \mathcal O(\theta^2 n + \epsilon n) \label{eqn_size_overlap} 
\end{align}
on an event of the probability of at least the same lower bound. The above identities are still valid if we replace \(\mathcal C_i\) by \(\mathcal C_i^j\defined \mathcal C_i \setminus \{j\}\). In the following, we will restrict to an event where \eqref{eqn_size_one_nbhd} and \eqref{eqn_size_overlap} are satisfied for any \(i \neq j\). By union bound, this event is of asymptotic probability 1 as long as
\begin{align}
 \min \left\{ \frac{\epsilon^2 n}{\theta}, \epsilon n \right\} \gg \log n.
 \label{condition_large_probability}
\end{align}
Note that according to our assumptions \(\theta n\to \infty\). We conclude from \eqref{eqn_size_one_nbhd} and \eqref{eqn_size_overlap} for \(\epsilon < \theta\) 
\begin{align}
 \mathbb E [S_{ij}|\mathcal C_i^j, \mathcal C_j^i] 
 &= \theta \sum_\alpha |\mathcal C_i^j \cap \mathcal C_\alpha^*| |\mathcal C_j^i \cap \mathcal C_\alpha^*| + p \sum_{\alpha\neq \beta} |\mathcal C_i^j\cap \mathcal C_\alpha^*||\mathcal C_j^i\cap \mathcal C_\beta^*| + \mathcal O(\theta |\mathcal C_i\cap \mathcal C_j|\nonumber)\\
 &=\theta\sum_\alpha \theta_i^\alpha \theta_j^\alpha n^2 + \rho \sum_{\alpha\neq\beta} \theta_i^\alpha \theta_j^\beta n^2 +\mathcal O(\theta^3 n+ \theta^2 \epsilon n^2)\nonumber \\
 &= \begin{cases}
 a   +\mathcal O(\theta^3 n+ \theta^2 \epsilon n^2)& ,\exists \alpha: i, j\in \mathcal C_\alpha^*\\
    c +\mathcal O(\theta^3 n+ \theta^2 \epsilon n^2) & ,\text{otherwise}
    \end{cases}\label{conditional_expectation_S}
\end{align}
for 
\begin{align*}
 a &\defined  n^2[ \theta ^3 +3(K-1)\theta \rho^2 + (K-1)(K-2)\rho^3]\text{ and}\\
c &\defined n^2\left[ 3\theta^2 \rho + 3(K-2)\theta \rho^2 + ((K-1)(K-2) + 1)\rho^3\right].
\end{align*}
Note that we have to add the term \(\mathcal O(\theta |\mathcal C_i\cap \mathcal C_j|)\) because for any point \(l\) in the overlap, we have \(\mathcal Y_{il}\mathcal Y_{ll}\mathcal Y_{lk} = 0\). To simplify the notation, we will drop the superindex in \(\mathcal C_i^j\) from now on. Similarly, we have
\begin{align*}
 \mathbb E [S_{ii}|\mathcal \mathcal C_i]  &=\theta \sum_\alpha |\mathcal C_i \cap \mathcal C_\alpha^*| ^2+ p \sum_{\alpha\neq \beta} |\mathcal C_i\cap \mathcal C_\alpha^*||\mathcal C_i\cap \mathcal C_\beta^*| + \mathcal O(\theta |\mathcal C_i|\nonumber)\\
 &=\theta\sum_\alpha (\theta_i^\alpha)^2 n^2 + p \sum_{\alpha\neq\beta} \theta_i^\alpha \theta_j^\beta n^2 + \mathcal O(\theta^2 n + \theta^2 \epsilon n^2)\nonumber \\
 &= a + \mathcal O(\theta^2 n + \theta^2 \epsilon n^2)\nonumber
\end{align*}
as well as
\begin{align*}
\mathbb E[N_{ij}|\mathcal C_i, \mathcal C_j]
&= N_{ij}\nonumber \\
&= \left(\sum_\alpha |\mathcal C_i\cap \mathcal C_\alpha^*|\right)\left(\sum_\alpha |\mathcal C_j\cap \mathcal C_\alpha^*|\right)\nonumber\\
&=d + \mathcal O(\theta^2 n + \epsilon \theta n^2)
\end{align*}
for
\[d \defined n^2\left[\theta^2 + 2(K-1) \theta \rho  + (K-1)^2 \rho ^2\right]\]
and 
\begin{align*}
\mathbb E[N_{ii}|\mathcal C_i]
&=d + \mathcal O(\theta  n + \epsilon \theta n^2).
\end{align*}
For \begin{align}
    R\defined \theta n + \epsilon \theta n^2 \label{defn_R} 
    \end{align}
 we conclude
\begin{align*}
\mathbb E[ \widetilde\theta_{ii}|\mathcal C_i] &=  \frac{a +\mathcal O\left(\theta R \right)}{d +\mathcal O\left(R\right)}.
\end{align*}
For \begin{align}
  \delta\defined \max\left\{\frac{1}{\theta n}, \frac{\epsilon}{\theta} \right\}\label{defn_delta}   
    \end{align}
 under the assumption \(\delta \ll 1\) we can rewrite 
\begin{align*}
\mathbb E[\widetilde \theta_{ii}|\mathcal C_i]&=\frac{a}{d}\left(1+\mathcal O(\delta)\right)
\end{align*}
and analogously 
\begin{align*}
\mathbb E[\widetilde \theta_{ij}|\mathcal C_i, \mathcal C_j] &= \begin{cases}\frac{a}{d}\left(1+\mathcal O(\delta)\right)& ,\exists \alpha: i, j\in \mathcal C_\alpha^* \\
                    \frac{c}{d} \left(1+\mathcal O(\delta)\right)&  ,\text{otherwise}
                    \end{cases}\\
\text{and }\mathbb E[ \widetilde \theta_{i\lor j} | \mathcal C_i, \mathcal C_j] &=\begin{cases} \frac{a}{d}\left(1+\mathcal O(\delta)\right) & ,\exists \alpha: i, j\in \mathcal C_\alpha^*\\
   \frac{a + c}{2d} \left(1+\mathcal O(\delta)\right)&  ,\text{otherwise.}
                    \end{cases}
\end{align*}
In the case of \(i\) and \(j\) belonging to different communities we conclude for large enough \(n\)
 \begin{align}
\mathbb E[\widetilde \theta_{i\lor j} - \widetilde\theta_{ij} | \mathcal C_i, \mathcal C_j]&= \frac{a+c}{2d}\left(1+\mathcal O(\delta)\right) - \frac{c}{d}\left(1+\mathcal O(\delta)\right)\nonumber\\
  &=\frac{a-c}{2d}  + \mathcal O(\delta \theta) \nonumber\\
  & \propto  \frac{(\theta-\rho)^3}{\theta^2 }  + \mathcal O(\delta \theta) \label{ineq_diff_lower_bd_simplified}
 \end{align}
and otherwise
\begin{align}
\mathbb E[ \widetilde\theta_{i\lor j} - \widetilde\theta_{ij} | \mathcal C_i, \mathcal C_j]  &= \frac{a}{d}\left(1+\mathcal O(\delta)\right) - \frac{a}{d}\left(1+\mathcal O(\delta) \right)\nonumber\\
  &=\mathcal O\left(\delta \theta\right). \label{ineq_diff_upper_bd_simplified}
\end{align}
Recall that with large probability, \(N_{ij} = d(1+\mathcal O(\delta))\). From Lemma \ref{lem_bernstein} we deduce that on an event of large probability,
\begin{align*}
\log \left(\frac{1}{2}\mathbb P\left(|S_{ij} -   \mathbb E [S_{ij} \vert \mathcal C_i^j, \mathcal C_j^i] |\geq \frac{ \rho d}{3} \bigg\vert \mathcal C_i^j, \mathcal C_j^i\right)\right) &\lesssim -\rho d \\
& \lesssim - \theta^3 n^2.
\end{align*}
As the RHS does not depend on \(\mathcal C_i^j\) or \(\mathcal C_j^i\), the bound is valid for the unconditional probability as well. From our assumption \(\theta^3 n^2 \gg \log n\) and \(\theta < \frac{1}{2}\) we deduce that with large probability for some \(x\in\{a, c\}\subseteq [\rho d, \theta d]\)
\begin{align}
 \widetilde\theta_{ij} &\in \left[\frac{x-\frac{\rho}{3}}{d}(1+\mathcal O(\delta)),\frac{x+\frac{\rho}{3}}{d}(1+\mathcal O(\delta)) \right] \nonumber\\ 
&\subseteq  \left[\frac{\rho}{2}, \theta + \frac{\rho}{2}\right].\label{bounded_interval}
\end{align}
Similarly, with large probability
\begin{align}\widetilde \theta_{i\lor j} \in \left[\frac{\rho}{2}, \theta + \frac{\rho}{2}\right]\text{.}\nonumber\label{condition_interval}\end{align}
In particular, the Fisher information of a Bernoulli variable with a mean parameter bounded as above is up to bounded constants given by \(\theta^{-1}\).

Analogously, the condition \(\theta \gg (\log n)^\frac{1}{3} n^{-\frac{2}{3}}\) and Lemma \ref{lem_bernstein} imply that with large probability
\begin{align*}
|S_{ij} -   \mathbb E [S_{ij}|\mathcal C_i, \mathcal C_j] |&\lesssim (\log n)^\frac{1}{2} \theta^\frac{3}{2} n 
\end{align*}
and
\begin{align*}
 |S_{i\lor j} -   \mathbb E [S_{i\lor j}|\mathcal C_i, \mathcal C_j] |&\lesssim (\log n)^\frac{1}{2} \theta^\frac{3}{2}n.
\end{align*}
In particular,
\begin{equation}
 |\widetilde \theta_{ij} -   \mathbb E [\widetilde \theta_{ij}|\mathcal C_i, \mathcal C_j] |\lesssim (\log n)^\frac{1}{2} \theta^{-\frac{1}{2} }n^{-1} \label{bound_stochastic_part}
\end{equation}
and
\begin{equation}
 |\widetilde \theta_{i\lor j} -   \mathbb E [\widetilde \theta_{i\lor j}|\mathcal C_i, \mathcal C_j] | \lesssim (\log n)^\frac{1}{2} \theta^{-\frac{1}{2}}n^{-1}. \nonumber
\end{equation}
Using the quadratic Taylor expansion of the Kullbach-Leibler and \eqref{ineq_diff_lower_bd_simplified}, we conclude in the case where \(i\) and \(j\) belong to different communities 
\begin{align}
\theta^{\frac{1}{2}}\mathcal K^\frac{1}{2}(\widetilde \theta_{ij}, \widetilde \theta_{i\lor j})&\gtrsim    |\widetilde\theta_{ij}- \widetilde\theta_{i\lor j}| \nonumber\\
&\geq \left|\mathbb E[\widetilde \theta_{ij} | \mathcal C_i, \mathcal C_j]-\mathbb E[\widetilde \theta_{i\lor j} | \mathcal C_i, \mathcal C_j]\right| -\left\vert\widetilde \theta_{ij} - \mathbb E[\widetilde \theta_{ij} | \mathcal C_i, \mathcal C_j]\right\vert - \left\vert\widetilde \theta_{i\lor j} - \mathbb E[\widetilde \theta_{i\lor j} | \mathcal C_i, \mathcal C_j]\right\vert\nonumber\\
&\gtrsim (\theta - \rho )^3\theta^{-2} + \mathcal O\left(\delta \theta + (\log n)^\frac{1}{2}\theta^{-\frac{1}{2}} n^{-1}\right), \label{KL_final_lower_bd}
\end{align}
whereas in the other case we conclude analogously from \eqref{ineq_diff_upper_bd_simplified}
\begin{align}
\theta^{\frac{1}{2}}\mathcal K^\frac{1}{2}(\theta_{ij}, \theta_{i\lor j})&\lesssim \delta\theta + (\log n)^\frac{1}{2}\theta^{-\frac{1}{2}} n^{-1}. \label{KL_final_upper_bd}
\end{align}
The same inequalities hold for \(\mathcal K^\frac{1}{2}(\widetilde\theta_{ii}, \widetilde \theta_{i\lor j})\) and \(\mathcal K^\frac{1}{2}(\widetilde\theta_{jj}, \widetilde \theta_{i\lor j})\). 
To ensure that the Kullback-Leibler divergence in the second case \eqref{KL_final_upper_bd} is significantly smaller than in the first case \eqref{KL_final_lower_bd}, it will suffice to check
\begin{align}
 \delta \theta&\ll \frac{(\theta - \rho )^3}{\theta^2}\nonumber\\
\Leftrightarrow \qquad \max\left\{\frac{1}{\theta n}, \frac{\epsilon}{\theta} \right\}&\ll \left(\frac{\theta-\rho}{\theta}\right)^3\label{condition1}
\end{align}
and
\begin{align}
 \sqrt{\frac{\log n}{\theta n^2}} &\ll \frac{(\theta - \rho)^3}{\theta^2} .\label{condition2}
\end{align}
With \(\cdot \ll\cdot \) we denote that inequality is satisfied for small enough yet not specified constant. 
These are all conditions in the proof that remain to be checked except the large probability assumption 
\begin{align}
 \min \left\{ \frac{\epsilon^2 n}{\theta \log n}, \frac{\epsilon n}{\log n} \right\} \gg 1 \label{condition3}
\end{align}
given in \eqref{condition_large_probability}.
An \(\epsilon > 0\) satisfying \eqref{condition1} and \eqref{condition3} exists as long as 
\begin{align}
 \frac{\log n}{n} + \sqrt{\frac{\theta \log n}{n}} \ll \frac{(\theta - \rho )^3}{\theta^2}, \label{condition21}
\end{align}
whereas the \(\epsilon\)-independent part of conditions \eqref{condition1} and \eqref{condition2} can be summarized by 
\begin{align}
 \frac{1}{n} + \sqrt{\frac{\log n}{\theta n^2}} \ll \frac{(\theta - \rho)^3}{\theta^2}.\label{condition22}
\end{align}
By simple calculus, we can verify that our assumption 
\begin{align}
 \theta - \rho\gg \max\{n^{-\frac{1}{3}} \theta^\frac{1}{2}(\log n)^\frac{1}{6},   n^{-\frac{1}{6}} \theta^\frac{5}{6} (\log n)^\frac{1}{6}\} \nonumber
\end{align}
implies conditions \eqref{condition21} and \eqref{condition22}. In view of the fact that test is scaled with \(N_{ij} = d(1+\mathcal O(\delta))\) and \(\delta \ll 1\) this also ensures consistency of the test (provided a proper threshold is given).

Next, let us consider the original AWCD algorithm. The proof deviates because we need to add an additional summand \(\mathcal O(|\mathcal C_i| + |\mathcal C_j|) =\mathcal O(\theta n) \) to the conditional expectation of \(S_{ij}\) in \eqref{conditional_expectation_S}. Therefore we also need to modify the definitions \(R\defined n + \epsilon \theta n^2\) and \(\delta \defined \max \{\frac{1}{\theta^2 n}, \frac{\epsilon}{n}\}\), cf. \eqref{defn_R} and \eqref{defn_delta}. Otherwise, we can follow the above proof, although \(S_{ij}\) is, after conditioning on \(\mathcal C_i\) and \(\mathcal C_j\), not any longer a sum of independent Bernoulli variables. However, it can be split into a deterministic part and the same sum of independent Bernoulli variables as above. By bounding the deterministic part, we can still establish inequality \eqref{bounded_interval}. Similarly to \(\eqref{condition21}\) and \(\eqref{condition22}\) we end up with the sufficient condition
\begin{align*}
 \frac{1}{\theta n} + \sqrt{\frac{\theta \log n}{n}} + \sqrt{\frac{\log n}{\theta n^2}}& \ll \frac{(\theta - \rho )^3}{\theta^2}
\end{align*}
which is satisfied by our condition
\begin{align}
 \theta - \rho \gg \max \{ n^{-\frac{1}{3}}\theta^\frac{1}{3}, n^{-\frac{1}{6}}\theta^\frac{5}{6} (\log n)^\frac{1}{6}\}. \nonumber
\end{align}
\end{proof}

\begin{proof}[Proof of Observation \ref{obs_improved_starting_guess}]
We can follow the proof of theorem \ref{thm_original_AWCD} by modifying \eqref{eqn_size_one_nbhd} to
\[|\mathcal C_i\cap\mathcal C_\alpha^* | = \mathbbm 1(i\in\mathcal C_\alpha) \theta n\]
in the first case and in the second case to \[|\mathcal C_i\cap\mathcal C_\alpha^* | = \mathbbm 1(i\in\mathcal C_\alpha) \theta^\frac{1}{2} n.\] In both cases, we end up with
 \begin{align}
\mathbb E[\widetilde \theta_{i\lor j} -\widetilde \theta_{ij} | \mathcal C_i, \mathcal C_j]& \propto \theta - \rho + \mathcal O(n^{-1}). \nonumber
 \end{align}
 However, using the larger starting guess, the stochastic bound \eqref{bound_stochastic_part} can be improved to 
 \begin{align*}
   |\widetilde\theta_{ij} - \mathbb E [\widetilde\theta_{ij}|\mathcal C_i, \mathcal C_j] |&\lesssim (\log n)^\frac{1}{2} n^{-1}.
 \end{align*}
 Finally, we end up with the sufficient condition
 \[\theta - \rho \gg (\log n)^\frac{1}{2} \theta^{-\frac{1}{2}} n^{-1} \]
in the case of the smaller starting guess and in the other case with
  \[\theta - \rho \gg (\log n)^\frac{1}{2} n^{-1} .\]
\end{proof}


\begin{lem}\label{Lemma_concentration_result_general_pathlength}
Suppose \(K\), \(k\) and \(\epsilon>0\) are fixed. We assume
\begin{itemize}
 \item \( \min\left\{\frac{\epsilon^2 n}{\theta}, \epsilon n\right\} \gg \log n\),
 \item \(\epsilon \ll \theta\),
 \item \(\theta n \gg 1\), and
 \item \(\theta^k n^{k-1} \ll 1\).
\end{itemize}
Then with large probability, we have
\[ |\mathcal C_{i}^k \cap \mathcal C_\alpha^* |  = a_k^{i, \alpha}n^k\left(1+\mathcal O\left(\frac{\epsilon}{\rho}+\theta^kn^{k-1}\right)\right),\]
where we define recursively
\begin{align*}
 a_1 &=  \theta,\\
 b_1 &= \rho,\\
 a_k &=  \theta a_{k-1} + (K-1)\rho b_{k-1},\\
 b_k &=    \rho a_{k-1} + \theta b_{k-1} +(K-2)\rho b_{k-1}\text{ and}\\
 a_k^{i, \alpha} &=\begin{cases}
                   a_k &, i\in \mathcal C^*_\alpha\\
                   b_k &,i \notin \mathcal C^*_\alpha.
                  \end{cases}
\end{align*}
The upper bound can be improved to
\[ |\mathcal C_{i}^k \cap \mathcal C_\alpha^* |  \leq a_k^{i, \alpha}n^k\left(1+\mathcal O\left(\frac{\epsilon}{\rho}\right)\right).\]
\end{lem}
\begin{rmk}
 As \(|\mathcal C_{i}^{k-1} \cap \mathcal C_\alpha^* |\) is significantly smaller than \(|\mathcal C_{i}^{k} \cap \mathcal C_\alpha^* |\), exactly the same concentration result holds for \(|\mathcal C_{i}^{\leq k} \cap \mathcal C_\alpha^* |\).
 The increase corresponds to an additional factor \(\left(1 + \mathcal O\left(\frac{1}{\theta n}\right)\right)\) - however in view of \(\frac{\epsilon}{\rho} \gg \frac{1}{\theta n}\) this factor can be omitted.
\end{rmk}
\begin{rmk}\label{rmk_difference_explicit_general_k}
By induction, we can compute an explicit formula for the difference
 \begin{align*}
a_k-b_k &= (\theta- \rho)(a_{k-1}-b_{k-1})\\
&=(\theta - \rho)^k.
\end{align*}
\end{rmk}

\begin{proof}[Proof of Lemma \ref{Lemma_concentration_result_general_pathlength}]
Applying Lemma \ref{lem_bernstein_homogeneous} and the union bound, a simple induction yields the following upper bound
\begin{align}
 |\mathcal C_{i}^k\cap \mathcal C_\alpha^*| &\leq a_k^{i, \alpha}n^k\left(1+\frac{\epsilon}{\rho}\right)^k
 \label{upper_concentration_result_general_k}
 \end{align}
as long as 
\begin{equation}
 \min\left\{\frac{\epsilon^2 n}{\theta}, \epsilon n\right\} \geq c(k) \log n.
 \label{condition_w_h_p}
\end{equation}
The lower bound is not as simple because of the potential overlap of the involved 1-neighborhoods. First of all, our assumptions \(\epsilon \ll \theta\), \(\theta n \gg 1\) as well as \(\theta^k n^{k-1} \ll 1\) are designed to ensure the upper bound
\begin{align*}
 |\mathcal C_{i}^{\leq k}\cap \mathcal C_\alpha^*| &\numleq{\ref{upper_concentration_result_general_k}} \sum_{l=1}^k K^{l-1}\theta^l n^l \left(1+\frac{\epsilon}{\rho}\right)^l\\
 & \leq  \frac{n}{2}.
\end{align*}
This allows us to apply Lemma \ref{lem_bernstein_homogeneous} to the set of members not contained in \(\mathcal C_{i, j}^{\leq k}\) while the large probability of the bound is still ensured by condition \eqref{condition_w_h_p}. The obtained lower bounds on the size of the 1-neighborhoods around points in \(\mathcal C_i^{k-1}\) take into account the potential overlaps and can thus be summarized.
To be precise, let us condition on the event that the (\(k-1\))-neighborhood is equal to \(\{v_1, v_2, \dots v_m\}\). Let us fix \(v\in\mathcal C^{k-1}\) and denote by \(\mathcal S(\mathcal C_i^{\leq l})\) the set of edges contained by a path of length at most \(l\) starting from \(i\).
We condition additionally on \(\mathcal S(\mathcal C_i^{\leq k-1})\cup \left(\bigcup_{v_i\neq v}\mathcal S(\mathcal C_{v_i})\right)\).
Now we apply Lemma \ref{lem_bernstein_homogeneous} to get a lower bound of the number of members in \(\mathcal C_\alpha^*\) that are connected to \(v\) and not contained in \(\mathcal C_i^{\leq k-1}\cup \left( \bigcup_{v_i\neq v}\mathcal C_{v_i}\right)\). This is in fact a lower bound on the number of members in \(\mathcal C_i^k \cap \mathcal C_\alpha^*\) that are only connected to \(i\) via a k-path containing \(v\). This lower bound is given by
\begin{align*}
& (\theta_v^\alpha - \epsilon) \left| \mathcal C_\alpha^* \setminus \left(\{i\} \cup \mathcal C_i^{\leq k}\cap \mathcal C_\alpha^*\right) \right|\\
 \geq & (\theta_v^\alpha - \epsilon)\left(n-1-a_1^{i, \alpha} n^1 \left(1+\frac{\epsilon}{\rho}\right)^1 - \dots - \alpha_k^{i, \alpha}n^{k} \left(1+\frac{\epsilon}{\rho}\right)^{k}\right)\\
=& (\theta_v^\alpha - \epsilon)\left(n-1-\sum_{l=1}^k a_l^{i, \alpha} n^l \left(1+\frac{\epsilon}{\rho}\right)^l\right)\\
\geq &(\theta_v^\alpha - \epsilon)\left(n-1-\sum_{l=1}^k K^{l-1} \theta^l n^l \left(1+\frac{\epsilon}{\rho}\right)^l\right)\\
= &(\theta_v^\alpha - \epsilon)\left(n- \frac{1}{K} \frac{K^{k+1} \theta^{k+1} n^{k+1} \left(1+\frac{\epsilon}{\rho}\right)^{k+1}- 1}{K\theta n \left(1+\frac{\epsilon}{\rho}\right) - 1}\right)\\
=&(\theta_v^\alpha - \epsilon) n(1-\delta_k)
\end{align*}
for  \[\delta_k = \frac{1}{n}\left(\frac{1}{K} \frac{K^{k+1} \theta^{k+1} n^{k+1} \left(1+\frac{\epsilon}{\rho}\right)^{k+1}- 1}{K\theta n \left(1+\frac{\epsilon}{\rho}\right) - 1}\right).\]
This bound is of course valid in the unconditional form as well. By induction, we conclude
\begin{align*}
|\mathcal C_i^k\cap \mathcal C_\alpha^*| & \geq \sum_\beta |\mathcal C_i^{k-1} \cap \mathcal C_\beta^*| \left[\theta \mathbbm 1(\alpha = \beta) + \rho \mathbbm1(\alpha\neq\beta) - \epsilon\right]n(1-\delta_k)\\
& \geq \alpha_i^k n^k \left(1-\frac{\epsilon}{\rho}\right)^k \prod_{l=1}^k (1-\delta_l).
\end{align*}
So we get in view of \(\delta_{k-1} < \delta_k\)
\begin{align*}
|\mathcal C_{i}^k \cap \mathcal C_\alpha^* | \in \left [a_k^{i, \alpha}n^k \left(1-\frac{\epsilon}{\rho}\right)^k(1-\delta_k)^{k}, a_k^{i, \alpha}n^k \left(1+\frac{\epsilon}{\rho}\right)^k\right ].
\end{align*}
Note that \(\delta_k = \mathcal O( \theta^k n^{k-1})\), so our assumptions ensure \(\delta_k \ll 1\). Also taking into account \(\epsilon \ll \theta\), we conclude 
\begin{align}
|\mathcal C_{i}^k \cap \mathcal C_\alpha^* | = a_k^{i, \alpha}n^k\left(1+\mathcal O\left(\frac{\epsilon}{\rho}+\theta^k n^{k-1}\right)\right). \nonumber 
\end{align}
\end{proof}


\begin{lem}\label{lem_overlap_k_nbhoods}
Under the same assumptions as in Lemma \ref{Lemma_concentration_result_general_pathlength} it holds on an event of large probability
 \begin{align*}
\left| \mathcal C_i^k \cap \mathcal C_j^k \cap \mathcal C_\alpha^* \right| \lesssim \theta^{2k} n^{2k-1} + \epsilon \theta^{k-1}n^k.
 \end{align*}
If we further assume that there exists no path of length at most \(k - l\) between \(i\) and \(j\) for some \(1 \leq l < k\), we have 
\begin{align*}
\left| \mathcal C_i^l \cap \mathcal C_j^k \cap \mathcal C_\alpha^* \right| \lesssim  \theta^{l+k} n^{l+k - 1} + \epsilon \theta^{l-1} n^l.
 \end{align*}
\end{lem}

\begin{proof}
According to Lemma \ref{Lemma_concentration_result_general_pathlength} 
\begin{align}
 |\mathcal C_i^{\leq k} \cap \mathcal C_\alpha^*|  \leq a_k^{i, \alpha} n^k \left( 1+ \mathcal O \left( \frac{\epsilon}{\theta} + \theta^k n^{k-1} \right) \right).
\label{upper_bd_for_intersec_bd} \end{align}
From our assumptions, we conclude further for any \(\beta\)
\begin{align*}
  |\mathcal C_i^{\leq k} \cap \mathcal C_\beta^*|  < 2 K^{k-1} \theta^{k} n^{k}.
\end{align*}
Let us introduce the notation \[m \defined n-2 K^{k-1} \theta^{k} n^{k}.\]
Note that our assumptions imply \(m > \frac{n}{2}\).
After conditioning on \(\mathcal C_j^1\), we can apply Lemma \ref{lem_bernstein_homogeneous}
to get a lower bound on the number of edges between \(i\) and the remaining points in \(\mathcal C_\beta^*\):
\begin{align}
|\mathcal C_\beta^*\cap \mathcal C_i^1\setminus \mathcal C_j^1| &\geq a_1^{i, \beta}m-\mathcal O(\epsilon n)\nonumber\\
&= a_1^{i, \beta} m\left(1-\mathcal O\left(\frac{\epsilon}{\theta}\right)\right)\label{lower_bd_ci1}
\end{align}
Note that this holds for any \(\beta\) with large probability due to the lower bound on \(\epsilon\) and \(m > \frac{n}{2}\).
Let us denote by \(\mathcal S(\mathcal C_i^{\leq k})\) the set of all edges contained by a path starting at \(i\) of length at most \(k\).
After additionally conditioning on \(\mathcal C_i^1= \{v_1, \dots, v_{k'}\}\) and \(\mathcal S(\mathcal C_j^{\leq 2})\) as well as \(\mathcal C_{v_2}^1\), ... and \(\mathcal C_{v_{k'}}^1\), the same argument yields for any \(\alpha\)
\begin{align}
 \left| \mathcal C_\alpha^* \cap \mathcal C_{v_{1}}^1 \setminus \left(\mathcal C_j^{\leq 2} \cup ( \mathcal C_{v_2}^1 \cup \cdots \cup \mathcal C_{v_{k'}}^1)\right) \right|
 \geq a_1^{v_1, \alpha} m\left(1-\mathcal O\left(\frac{\epsilon}{\theta}\right)\right).\label{lower_bd_cv1}
\end{align}
Summarizing the lower bound \eqref{lower_bd_cv1} yields in view of \eqref{lower_bd_ci1} and \(\frac{\epsilon}{\theta} \ll 1\)
\begin{align*}
 | \mathcal C_\alpha^* \cap \mathcal C_i^2\setminus \mathcal C_j^{\leq 2}| & \geq a_2^{i, \alpha}  m^2 \left(1-\mathcal O\left(\frac{\epsilon}{\theta}\right)\right).
\end{align*}
Iterating the argument further, we end up with
\begin{align}
  | \mathcal C_\alpha^* \cap \mathcal C_i^{k}\setminus  \mathcal C_j^{\leq k}| & \geq a_{k}^{i,  \alpha}  m^{k} \left(1-\mathcal O\left(\frac{\epsilon}{\theta}\right)\right).
\label{lower_bd_for_intersec_bd}
  \end{align}
From \eqref{upper_bd_for_intersec_bd}, \eqref{lower_bd_for_intersec_bd} and \(\theta^k n^{k-1} \ll 1\) we conclude
\begin{align*}
  |   \mathcal C_\alpha^*\cap \mathcal C_i^{k} \cap  \mathcal C_j^{\leq k}| & \leq a_{k}^{i,  \alpha}  \left(n^k - m^{k}\right) + \mathcal O\left(\epsilon \theta^{k-1}n^k\right) \\
  &= a_k^{i, \alpha} n^k (1 - (1 - \mathcal O(\theta^{k}n^{k-1}))^k) + \mathcal O\left(\epsilon \theta^{k-1}n^k\right)\\
  & \lesssim \theta^{2k} n^{2k-1} + \epsilon \theta^{k-1}n^k.
  \end{align*}
This is of course also an upper bound for \( | \mathcal C_i^{k} \cap \mathcal C_j^{k}\cap \mathcal  C_\alpha^*| \). 
Next, we discuss the overlap of two neighborhoods of different sizes \(|\mathcal C_i^l \cap \mathcal C_j^k \cap \mathcal C_\alpha^*|\) with \(l<k\)
in case there exists no path of length at most \(k-l\) between \(i\) and \(j\).
We can then argue very similar (after conditioning on \( \mathcal S( \mathcal C_j^{\leq k - l  + 1})\)) as in the case of \(l=k\) and start the induction with a lower lower bound analogous to \eqref{lower_bd_ci1}:
\begin{align*}
| \mathcal C_\beta^*\cap  \mathcal C_i^1\setminus \mathcal  C_j^{\leq k - l  + 1}| &\geq a_1^{i, \beta} m\left(1-\mathcal O\left(\frac{\epsilon}{\theta}\right)\right)
\end{align*}
Analogously to \eqref{lower_bd_cv1} and \eqref{lower_bd_for_intersec_bd} we iterate this argument further and end up with
\begin{align*}
|   \mathcal C_\alpha^* \cap \mathcal C_i^l\setminus  \mathcal C_j^{\leq k }| &\geq a_l^{i, \alpha} m^l\left(1-\mathcal O\left(\frac{\epsilon}{\theta}\right)\right).
\end{align*}
Together with the upper bound from Lemma \ref{Lemma_concentration_result_general_pathlength}
\begin{align*}
 |\mathcal C_i^l \cap \mathcal C_\alpha^*| \leq \alpha_l^{i, \alpha} n^l \left( 1+ \mathcal O\left( \frac{\epsilon}{\theta} + \theta^l n^{l-1} \right)\right)
\end{align*}
we conclude 
\begin{align*}
  |\mathcal C_i^l \cap \mathcal C_j^k \cap  \mathcal C_\alpha^*| &\leq \alpha_l^{i, \alpha} \left(n^l - m^l\right) + \mathcal O \left( \epsilon \theta^{l-1} n^l \right)\\
  & \lesssim \theta^{l+k} n^{l+k - 1} + \epsilon \theta^{l-1} n^l.
\end{align*}
\end{proof}

\begin{proof}[Proof of Theorem \ref{thm_consistency_AWCD_general_k}]
 We start discussing the algorithm AWCD$_k$ and consider \(i\neq j\). Let us condition on \(\mathcal S(\mathcal C_i^{\leq k}) \cup \mathcal S(\mathcal C_j^{\leq k})\) and write 
\begin{align*}
 \mathbb E [S_{ij} | \mathcal S(\mathcal C_i^{\leq k}) \cup \mathcal S(\mathcal C_j^{\leq k})] &= \sum_{v_1\in\mathcal C_i^k, v_2\in\mathcal C_j^k} \mathcal Y_{v_1 v_2} \\
 &=D + S,
\end{align*}
where \(D\) denotes the deterministic part of the sum and \(S\) denotes the stochastic part (i.e. an independent sum of Bernoulli variables). Note that a summand \(\mathcal Y_{v_1 v_2}\) is deterministic and \(\equiv 1\) if and only if \(v_1 \in \mathcal C_j^{k-1}\) or  \(v_2 \in \mathcal C_i^{k-1}\). First of all, let us consider the case \(\mathcal Y_{ij} = 0\). According to Lemma \ref{lem_overlap_k_nbhoods} we have w.h.p.
\begin{align}
|\mathcal C_i^{k-1} \cap \mathcal  C_j^{k} \cap \mathcal C_\alpha^*| & \lesssim \theta^{2k-1}n^{2k-2} + \epsilon \theta^{k-2} n^{k-1} \label{upper_bd_overlap_k_k-1}.
\end{align}
Moreover, the 1-neighborhood around any member is of size at most \(\mathcal O(\theta n (1 + \frac{\epsilon}{\theta}) )\), so 
\begin{align*}
D & \lesssim \theta^{2k} n^{2k-1} + \epsilon \theta^{k-1} n^k.
\end{align*}
Next, let us discuss the case where \(\mathcal Y_{ij} = 1\). We start studying the case when we do not correct the bias at all, i.e. with the algorithm AWCD$_k^\circ$. Then the upper bound \eqref{upper_bd_overlap_k_k-1} is no longer valid, instead we get
\begin{align}
|\mathcal C_i^{k-1} \cap \mathcal  C_j^{k} \cap \mathcal C_\beta^*| & \leq |\mathcal C_i^{k-1} \cap \mathcal C_\beta^*| \nonumber \\
& \leq a_{k-1}^{i, \beta} n^{k-1} \left( 1+ \mathcal O \left( \frac{\epsilon}{\theta}\right) \right) \label{upper_bd_overlap_k_k-1_yij_is1}.
\end{align}
Again, using the upper bound \(\theta_v^\alpha n (1 + \mathcal O(\frac{\epsilon}{\theta}))\) on size of the intersection of \(\mathcal C_\alpha^*\) with the 1-neighborhood around any member \(v\) we conclude from \eqref{upper_bd_overlap_k_k-1_yij_is1}
\begin{align*}
D & \leq \sum_\alpha \left( a_{k}^{i, \alpha} + a_{k}^{j, \alpha} \right)  n^k \left( 1+ \mathcal O \left( \frac{\epsilon}{\theta}\right) \right).
\end{align*}
Next, we want to compute a lower bound on \(D\). Note that if \(v_2\in \mathcal C_j^k \cap  \mathcal C_i^{k-1}\), then any neighbor \(v\) of \(v_2\) that is not contained by \(\mathcal C_i^{\leq k-1}\) contributes to a deterministic summand \(\mathcal Y_{v v_2} \equiv 1\) of \(D\). From the proof of Lemma \ref{lem_overlap_k_nbhoods} we already have the lower bound \eqref{lower_bd_for_intersec_bd}
\begin{align}
 |\mathcal C_j^k\cap  \mathcal C_i^{k-1} \cap \mathcal C_\beta^*| &=|\mathcal C_\beta^* \cap \mathcal C_i^{k-1} \setminus \mathcal C_j^{\leq k-1}| \nonumber\\
 &\geq a_{k-1}^{i,  \beta}(n-2K^{k-2}\theta^{k-1}n^{k-1})^{k-1} \left(1-\mathcal O\left(\frac{\epsilon}{\theta}\right)\right)\nonumber\\
 &=a_{k-1}^{i,  \beta}n^{k-1} \left( 1- \mathcal O\left(\theta^{k-1}n^{k-2}\right)\right) \left(1-\mathcal O\left(\frac{\epsilon}{\theta}\right)\right). \label{lower_bd_intersec_k_k-1}
\end{align}
At the same time, any member \(v\) has at least \(\theta_v^\alpha n\left(1-\mathcal O\left(\theta^{k-1}n^{k-2}\right) \right) (1-\mathcal O(\frac{\epsilon}{\theta}))\) neighbors inside \(\mathcal C_\alpha^* \setminus \mathcal C_i^{\leq k-1}\). Combining this with \eqref{lower_bd_intersec_k_k-1} yields
\begin{align*}
 D &\geq \sum_\alpha (a_k^{i, \alpha} + a_k^{j, \alpha})  n^k \left( 1- \mathcal O\left(\theta^{k-1}n^{k-2}\right)\right) \left(1-\mathcal O\left(\frac{\epsilon}{\theta}\right)\right).
\end{align*}
Considering that according to Lemma \ref{Lemma_concentration_result_general_pathlength}
\begin{align*}
|\mathcal C_i^k | + |\mathcal C_j^k| = \sum_\alpha (a_k^{i, \alpha} + a_k^{j, \alpha})  n^k \left( 1+ \mathcal O\left( \frac{\epsilon}{\theta} + \theta^kn^{k-1} \right)\right),
\end{align*}
the deterministic part \(D\) is significantly smaller after we adjust the definition of \(S_{ij}\) to
\begin{align*}
  \sum_{v_1\in\mathcal C_i^k, v_2\in\mathcal C_j^k}\mathcal  Y_{v_1 v_2} - \mathbbm1(\mathcal Y_{ij} = 1)(|\mathcal C_i^k|+\mathcal C_j^k|).
\end{align*}
To be precise, considering again AWCD$_k$, we now have the same bound as in the case \(\mathcal Y_{ij} = 0\)
\begin{align}
 D &=\mathcal O\left( \left(\frac{\epsilon}{\theta} + \theta^k n^{k-1} \right) \theta^k n^k\right)\nonumber \\
 &= \mathcal O\left(\theta^{2k} n^{2k-1} + \epsilon \theta^{k-1}n^k \right) \label{upper_bd_D_general_k}.
\end{align}
Next, let us consider the stochastic part \(S\). Note that if \(v_1\in\mathcal C_i^k \setminus \mathcal C_j^{\leq {k-1}}\) and \(v_2 \in\mathcal C_j^k \setminus \mathcal C_i^{\leq k-1}\), then even after conditioning on \(\mathcal S(\mathcal C_i^{\leq k})\cup \mathcal S(\mathcal C_j^{\leq k})\), the term \(\mathcal Y_{ij}\) is still a Bernoulli variable and not deterministic. In view of 
\begin{align*}
|\mathcal C_\alpha^* \cap \mathcal C_i^k \setminus \mathcal C_j^{\leq k-1}| & =| \mathcal C_i^k \cap \mathcal C_\alpha^* |\left(1 + \mathcal O\left( \frac{|\mathcal C_j^{\leq k-1}\cap \mathcal C_\alpha^* |}{|\mathcal C_i^k\cap \mathcal C_\alpha^* |}\right)\right)\\
&=|\mathcal C_i^k\cap \mathcal C_\alpha^* | \left(1+\mathcal O\left(\frac{1}{\theta n}\right)\right)
\end{align*}
we conclude similar to \eqref{conditional_expectation_S} from Lemma \ref{Lemma_concentration_result_general_pathlength} and Lemma \ref{lem_overlap_k_nbhoods}
\begin{align}
S =& \theta \sum_\alpha |\mathcal C_i^k \cap \mathcal C_\alpha^*| | \mathcal C_j^k \cap  \mathcal C_\alpha^*|   \left(1+\mathcal O\left(\frac{1}{\theta n}\right)\right) \nonumber \\
& +  \rho \sum_{\alpha\neq \beta} |\mathcal C_i^k\cap  \mathcal C_\alpha^*||\mathcal C_j^k\cap  \mathcal C_\beta^*|  \left(1+\mathcal O\left(\frac{1}{\theta n}\right)\right) + \mathcal O\left(\theta |\mathcal C_i^k\cap  \mathcal C_j^k|\nonumber\right)\nonumber \\
 &= \begin{cases}
 an^{2k}\left( 1+ \mathcal O\left( \frac{\epsilon}{\theta} + \theta^k n^{k-1} + \frac{1}{\theta n} \right)\right)   +\mathcal O(\theta^{2k+1} n^{2k-1} + \epsilon \theta^{k}n^k + D)& ,\exists \alpha: i, j\in \mathcal C_\alpha^* \\
 cn^{2k}\left( 1+ \mathcal O\left( \frac{\epsilon}{\theta} + \theta^k n^{k-1} + \frac{1}{\theta n} \right)\right)   +\mathcal O(\theta^{2k+1} n^{2k-1} + \epsilon \theta^{k}n^k + D)&  ,\text{otherwise}
    \end{cases}\nonumber \\
     &= \begin{cases}
 an^{2k} + \mathcal O\left(\theta R\right)& ,\exists \alpha: i, j\in \mathcal C_\alpha^* \\
 cn^{2k}  + \mathcal O\left(\theta R\right)&  ,\text{otherwise}
    \end{cases} \label{concentration_S_general_k}
\end{align}
for
\begin{align*}
a  &=   \theta \left[a_k^2 + (K-1)b_k^2\right]+\rho \left[2(K-1)a_kb_k + (K-1)(K-2)b_k^2\right],\\
c &= \theta \left[2a_kb_k+(K-2)b_k^2\right] +\rho \left[ a_k^2 +2(K-2)a_kb_k + ((K-1)(K-2) + 1)b_k^2 \right] \text{ and}\\
\theta R &= \mathcal O \left( \theta^{2k+1} n^{2k} \left( \frac{\epsilon}{\theta} + \theta^k n^{k-1} + \frac{1}{\theta n} \right)  + \theta^{2k+1}n^{2k-1} + \epsilon \theta^k n^k + \theta^{2k}n^{2k-1} + \epsilon \theta^{k-1} n^k \right)\\
& = \mathcal O\left(\theta^{3k+1}n^{3k-1} + \epsilon \theta^{k-1}n^k  + \epsilon \theta^{2k}n^{2k} \right).
\end{align*}
Note that according to Remark \ref{rmk_difference_explicit_general_k} we have
\begin{align*}
a - c &= \theta a_k^2 + \theta b_k^2 -2 \theta a_k b_k + 2\rho a_k b_k - \rho a_k^2 - \rho b_k^2\\
&= (\theta -  \rho )(a_k - b_k)^2\\
&= (\theta - \rho)^{2k+1}.
\end{align*}
In case \(i=j\) the concentration result \eqref{concentration_S_general_k} on \(S\) as well as the upper bound \eqref{upper_bd_D_general_k} for \(D\) are still valid: After conditioning on \( \mathcal S(\mathcal C_i^{\leq k})\) we have \(D\equiv 0\) as the only deterministic edges are \(\mathcal Y_{vv}=0\) for any \(v\in\mathcal C_i^k\). This leads to an additional term \(\mathcal O(\theta^{k+1}n^k)\) that needs to be considered in \eqref{concentration_S_general_k}, however this is much smaller than \(\mathcal O(\theta R)\). 

Using again Lemma \ref{Lemma_concentration_result_general_pathlength}, we compute for the denominator in case \(i\neq j\)
\begin{align*}
 N_{ij}^k &= \left(\sum_\alpha a_k^{i, \alpha} \right)^2 n^{2k}  \left( 1 + \mathcal O\left(\frac{\epsilon}{\theta} + \theta^k n^{k-1} \right)\right)\\
 &= dn^{2k} + \mathcal O\left(\epsilon \theta^{2k-1} n^k + \theta^{3k} n^{3k-1} \right)\\
 &=dn^{2k} + \mathcal O(R)
\end{align*}
for
\begin{align*}
 d&= a_k + (K-1)b_k.
\end{align*}
In case \(i=j\) the above is still valid: We only need to add a term \(\mathcal O(C_i^k) = \mathcal O(\theta^k n^k)\) which is in view of \(\epsilon \theta^{2k-1} n^{2k} \gg \theta^{2k-1} n^{2k-1} \gg \theta^k n^k\) already contained in \(\mathcal O(R)\). 

The rest of the proof is very similar to the proof of Theorem \ref{thm_original_AWCD}: From the above, we conclude
\begin{align*}
\mathbb E[ \widetilde\theta_{ij}| \mathcal S(\mathcal C_i^{\leq k}) \cup  \mathcal S(\mathcal C_j^{\leq k})] &= \begin{cases}\frac{a}{d}\left(1+\mathcal O(\delta)\right)& ,\exists \alpha: i, j\in \mathcal C_\alpha^* \\
                    \frac{c}{d} \left(1+\mathcal O(\delta)\right)&  ,\text{otherwise}
                    \end{cases}\\
\text{and }\mathbb E[ \widetilde\theta_{i\lor j} |  \mathcal S(\mathcal C_i^{\leq k}) \cup  \mathcal S(\mathcal C_j^{\leq k})] &=\begin{cases} \frac{a}{d}\left(1+\mathcal O(\delta)\right) & ,\exists \alpha: i, j\in \mathcal C_\alpha^*\\
   \frac{a + c}{2d} \left(1+\mathcal O(\delta)\right)&  ,\text{otherwise}
                    \end{cases}
\end{align*}
for
\begin{align*}
\delta & = \frac{R}{\theta^{2k} n^{2k}}\\
\end{align*}
and under the assumption \(\delta \ll 1\), which we discuss later. 
For large enough \(n\) in the case of \(i\) and \(j\) belonging to different communities we have
 \begin{align}
\mathbb E[ \widetilde\theta_{i\lor j} -\widetilde \theta_{ij} |  \mathcal S(\mathcal C_i^{\leq k}) \cup  \mathcal S(\mathcal C_j^{\leq k})]&= \frac{a+c}{2d}\left(1+\mathcal O(\delta)\right) - \frac{c}{d}\left(1+\mathcal O(\delta)\right)\nonumber\\
  &=\frac{a-c}{2d}  + \mathcal O(\delta \theta) \nonumber\\
  & \propto  \frac{(\theta-p)^{2k+1}}{\theta^{2k} }  + \mathcal O(\delta \theta) \label{ineq_diff_lower_bd_simplified_general_l}
 \end{align}
and otherwise
\begin{align}
\mathbb E[ \widetilde\theta_{i\lor j} - \widetilde \theta_{ij} | \mathcal S(\mathcal C_i^{\leq k}) \cup  \mathcal S(\mathcal C_j^{\leq k})]  &= \frac{a}{d}\left(1+\mathcal O(\delta)\right) - \frac{a}{d}\left(1+\mathcal O(\delta) \right)\nonumber\\
  &=\mathcal O\left(\delta \theta\right). \label{ineq_diff_upper_bd_simplified_general_k}
\end{align}
Next we apply Lemma \ref{lem_bernstein}: Considering that the sum \(S_{ij}\) consists of \(\mathcal O(\theta^{2k}n^{2k})\) summands, we have with large probability
\begin{align*}
 |S_{ij} - \mathbb E[S_{ij} |  \mathcal S(\mathcal C_i^{\leq k}) \cup  \mathcal S(\mathcal C_j^{\leq k})]| & \lesssim \log n + (\log n)^\frac{1}{2} \theta^{k+\frac{1}{2}}n^k.
\end{align*}
The upper bound simplifies under the condition \(\theta^{2k+1}n^{2k} \gtrsim \log n\) to 
\begin{align*}
 |S_{ij} - \mathbb E[S_{ij} |  \mathcal S(\mathcal C_i^{\leq k}) \cup  \mathcal S(\mathcal C_j^{\leq k})]| & \lesssim (\log n)^\frac{1}{2} \theta^{k+\frac{1}{2}}n^k,
\end{align*}
implying
\begin{align}
  |\widetilde \theta_{ij} - \mathbb E[\widetilde\theta_{ij} |  \mathcal S(\mathcal C_i^{\leq k}) \cup  \mathcal S(\mathcal C_j^{\leq k})]| & \lesssim \frac{(\log n)^\frac{1}{2}}{\theta^{k-\frac{1}{2}}n^k}. \label{upper_bd_stochastic_uncertainty_general_k}
\end{align}
 As \(\theta^{2k+1}n^{2k} \gg \log n\), the upper bound is \(\ll \theta\), implying furthermore
 \begin{align*}
\widetilde\theta_{ij}, \widetilde\theta_{i\lor j} \in \left[\frac{\rho}{2}, \theta + \frac{\rho}{2}\right].
 \end{align*}
Consequently, the corresponding Fisher information is up to bounded constants given by \(\theta^{-1}\). 
Using the quadratic Taylor expansion of the Kullbach-Leibler, we conclude in the case where \(i\) and \(j\) belong to different communities from \eqref{ineq_diff_lower_bd_simplified_general_l} and \eqref{upper_bd_stochastic_uncertainty_general_k}
\begin{align}
\theta^{\frac{1}{2}}\mathcal K^\frac{1}{2}(\widetilde\theta_{ij}, \widetilde\theta_{i\lor j}) \gtrsim  &  |\widetilde \theta_{ij}- \widetilde \theta_{i\lor j}| \nonumber\\
\geq & \left|\mathbb E[\widetilde \theta_{ij} |  \mathcal S(\mathcal C_i^{\leq k}) \cup  \mathcal S( \mathcal C_j^{\leq k})]  -\mathbb E[\widetilde \theta_{i\lor j} |  \mathcal S(\mathcal C_i^{\leq k}) \cup  \mathcal S( \mathcal C_j^{\leq k})]\right| \nonumber\\
& -\left\vert \widetilde\theta_{ij} - \mathbb E[ \widetilde\theta_{ij} |  \mathcal S(\mathcal C_i^{\leq k}) \cup  \mathcal S( \mathcal C_j^{\leq k})]\right\vert \nonumber\\
&- \left\vert \widetilde\theta_{i\lor j} - \mathbb E[ \widetilde\theta_{i\lor j} | \mathcal S(\mathcal C_i^{\leq k}) \cup  \mathcal S( \mathcal C_j^{\leq k})]\right\vert\nonumber\\
\gtrsim & \frac{(\theta - \rho)^{2k+1}}{\theta^{2k}} + \mathcal O\left(\delta \theta +  \frac{(\log n)^\frac{1}{2}}{\theta^{k-\frac{1}{2}}n^k}\right), \label{KL_final_lower_bd_general_k}
\end{align}
whereas in the other case we conclude analogously from \eqref{ineq_diff_upper_bd_simplified_general_k}
\begin{align}
\theta^{\frac{1}{2}}\mathcal K^\frac{1}{2}(\widetilde\theta_{ij}, \widetilde\theta_{i\lor j})&\lesssim \delta \theta +  \frac{(\log n)^\frac{1}{2}}{\theta^{k-\frac{1}{2}}n^k}. \label{KL_final_upper_bd_general_k}
\end{align}
From \eqref{KL_final_upper_bd_general_k} and \eqref{KL_final_lower_bd_general_k} we conclude: A sufficient condition (together with the lower bound from Lemma \ref{Lemma_concentration_result_general_pathlength} on \(\epsilon\) that guarantees the large probability of the concentration results) for concistency of the algorithm is
\begin{align}
 \frac{(\theta - \rho )^{2k+1}}{\theta^{2k}} &\gg 
  \delta \theta +  \frac{(\log n)^\frac{1}{2}}{\theta^{k-\frac{1}{2}}n^k}  \nonumber \\
  \Leftrightarrow (\theta -  \rho )^{2k+1} &\gg \theta Rn^{-2k} + (\log n)^\frac{1}{2} \theta^{k + \frac{1}{2}}n^-k\nonumber\\
\Leftrightarrow (\theta - \rho)^{2k+1} &\gg \theta^{3k+1}n^{k-1} + \epsilon \theta^{k-1}n^{-k}  + \epsilon \theta^{2k} + (\log n)^\frac{1}{2} \theta^{k + \frac{1}{2}}n^{-k}. \label{condition1_general_k}
\end{align}
At the same time, considering \(\theta \gg \frac{\log n}{n}\), all of the above concentration results above only hold with large probability as long as \(\epsilon\) satisfies the lower bound (c.f. Lemma \ref{lem_bernstein})
\begin{align}
 \min \left\{ \frac{\epsilon^2 n}{\theta \log n}, \frac{\epsilon n}{\log n} \right\}  &\gg 1 \nonumber\\
 \Leftrightarrow \epsilon &\gg \frac{\log n}{n} + \sqrt{\frac{\theta \log n}{n}}\nonumber\\
  \Leftrightarrow \epsilon &\gg  \sqrt{\frac{\theta \log n}{n}},\label{condition2_general_k}
\end{align}
whereas the \(\epsilon\)-dependent part of \eqref{condition1_general_k} is equivalent to the upper bound
\begin{align}
\epsilon &\ll \min\left\{\frac{(\theta - \rho)^{2k+1}}{\theta^{2k}}, \frac{(\theta - \rho)^{2k+1}n^k}{\theta^{k-1} } \right\}. \label{condition3_general_k}
\end{align}
An \(\epsilon> 0\) satisfying \eqref{condition2_general_k} as well as \eqref{condition3_general_k} exists if and only if
\begin{align}
\min\left\{\frac{(\theta - \rho)^{2k+1}}{\theta^{2k}}, \frac{(\theta - \rho)^{2k+1}n^k}{\theta^{k-1} } \right\} & \gg \sqrt{\frac{\theta \log n}{n}}  \nonumber\\
\Leftrightarrow (\theta - \rho)^{2k+1} &\gg \theta^{2k+\frac{1}{2}}n^{-\frac{1}{2}} (\log n)^\frac{1}{2} + \theta^{k-\frac{1}{2}} n^{-k-\frac{1}{2}} (\log n)^\frac{1}{2}. \label{condition4_general_k}
\end{align}
Because we only discuss the case \(k\geq 2\), our assumptions ensure \(\theta \ll n^{-\frac{k-1}{k}} \ll n^{-\frac{1}{2}}\) and in particular \(\theta^{k-\frac{1}{2}}n^{-k-\frac{1}{2}} (\log n)^\frac{1}{2} \ll \theta^{k+\frac{1}{2}}n^{-k}(\log n)^\frac{1}{2}\). Thus we can simplify \eqref{condition1_general_k} and \eqref{condition4_general_k} into the following sufficient condition for consistency of the algorithm
\begin{align*}
 \theta - \rho & \gg \max \{A, B, C\}\\
 \end{align*}
for
\begin{align*}
A &= \theta^\frac{3k+1}{2k+1} n^\frac{k-1}{2k+1},\\
B &= \theta^\frac{4k+1}{4k+2}n^{-\frac{1}{4k+2}} (\log n)^\frac{1}{4k+2}\text{ and}\\
C&= \theta^\frac{2k-1}{4k+2}n^{-\frac{1}{2}} (\log n)^\frac{1}{4k+2}.
\end{align*}

Next, let us consider the algorithm AWCD\(^+\). The proof is almost identical - we only need to modify \(D= 0\), leading to
\begin{align*}
 \theta R = \mathcal O(\theta^{3k+1}n^{3k-1}+\epsilon \theta^{2k} n^{2k}).
\end{align*}
Consequently, in the lower bound \eqref{condition1_general_k} we can drop the term \(\epsilon \theta^{k-1} n^{-k}\) leading to the following analogous condition
\begin{align*}
 (\theta - \rho)^{2k+1} &\gg \theta^{3k+1}n^{k-1}  + \epsilon \theta^{2k} + (\log n)^\frac{1}{2} \theta^{k + \frac{1}{2}}n^{-k}
\end{align*}
and \eqref{condition4_general_k} simplifies to
\begin{align*}
 (\theta - \rho)^{2k+1} &\gg \theta^{2k+\frac{1}{2}}n^{-\frac{1}{2}} (\log n)^\frac{1}{2}.
\end{align*}
We end up with the final sufficient condition
\begin{align*}
\theta -\rho \gg \max\{A, B, C^+\}
\end{align*}
for
\begin{align*}
C^+&= \theta^\frac{1}{2}n^{-\frac{k}{2k+1}} (\log n)^\frac{1}{4k+2}.
\end{align*}

Finally, let us consider the algorithm AWCD$_k^\circ$ without any correction of the bias term. The deterministic part of the conditional expectation of \(S_{ij}\) increases to \(D=\mathcal O(\theta^kn^k)\) implying
\begin{align*}
 \theta R = \mathcal O(\theta^{3k+1}n^{3k-1}+\epsilon \theta^{2k} n^{2k} + \theta^kn^k).
\end{align*}
We end up with the following lower bound corresponding to \eqref{condition1_general_k}
\begin{align*}
(\theta-\rho)^{2k+1} & \gg \theta^{3k+1}n^{k-1} + \epsilon \theta^{2k} + \theta^k n^{-k} + (\log n)^\frac{1}{2} \theta^{k+\frac{1}{2} }n^{-k}\\
\Leftrightarrow (\theta-\rho)^{2k+1} & \gg \theta^{3k+1}n^{k-1} + \epsilon \theta^{2k} + \theta^k n^{-k} 
\end{align*}
and instead of \eqref{condition4_general_k} we have again \begin{align*}
 (\theta - \rho)^{2k+1} &\gg \theta^{2k+\frac{1}{2}}n^{-\frac{1}{2}} (\log n)^\frac{1}{2}.
\end{align*}
This leads to the final sufficient condition
\begin{align*}
\theta - \rho & \gg \max\{A, B, C^\circ\}
\end{align*}
with
\begin{align*}
 C^\circ &= \theta^\frac{k}{2k+1}n^{-\frac{k}{2k+1}}.
\end{align*}

\end{proof}


\begin{proof}[Proof of Proposition \ref{Proposition_different_block_size}]
We proceed analogously as in the case of identical block size. Suppose \(\epsilon > 0\) and \(\alpha\in\{1, 2, \dots, K\}\). Moreover, we use the notation \(\alpha_i\) for a member \(i\) to denote the corresponding community index such that \(i\in\mathcal C_{\alpha_i}^*\). According to Lemma \ref{lem_bernstein_homogeneous} we have
\begin{align} |\mathcal C_i^j\cap \mathcal C^*_\alpha| = \theta_i^\alpha n_\alpha + \mathcal O\left(\epsilon n_{\max} + \theta \right)   \label{dbs_eqn_size_one_nbhd}\end{align}
for \begin{align}\theta_i^\alpha = \begin{cases}
                         \theta &, i \in \mathcal C^*_\alpha\\
                         \rho &, i\notin \mathcal C^*_\alpha
                        \end{cases}
\nonumber
\end{align}
on an event of probability at least
\[1 - 2\exp \left( - \frac{\epsilon^2 n_{\min}}{2\theta + \epsilon}\right)\text{.}\]
Similarly, we conclude for \(i\neq j\)
\begin{align}
|\mathcal C_i^j \cap \mathcal C_j^j|= \mathcal O(\theta^2 n_{\max} + \epsilon n_{\max}) \label{dbs_eqn_size_overlap} 
\end{align}
on an event of probability of at least the same lower bound. In the following, we will restrict to an event where \eqref{dbs_eqn_size_one_nbhd} and \eqref{dbs_eqn_size_overlap} are satisfied for any \(i \neq j\). By union bound, this event of asymptotic probability 1 as long as
\begin{align}
 \min \left\{ \frac{\epsilon^2 n_{\min}}{\theta}, \epsilon n_{\min} \right\} \gg \log n_{\min}.
 \nonumber
\end{align}
Before moving on, we introduce the following notation:
\begin{align*}
 T &= \sum_{\alpha \notin \{\alpha_i, \alpha_j\}} n_\alpha\\
a_i^\circ &= \theta^3 n_i^2 + \theta \rho^2(n_j^2 + 2n_i n_j) \\
 c_{ij}^\circ &= \theta^2 \rho (n_i^2+n_j^2+n_in_j)+ \rho^3 n_i n _j\\
 r_i &= \theta \rho^2 (\sum_{\alpha\notin\{\alpha_i, \alpha_j\}} n_\alpha^2 +2n_i T) +2 \rho^3 n_j T + \rho^3\sum_{\alpha\neq \beta, \alpha, \beta \notin
 \{\alpha_i, \alpha_j\} } n_\alpha n_\beta\\
 r_{ij} &=  \theta \rho^2 \left(\sum_{\alpha \notin \{\alpha_i, \alpha_j\}} n_\alpha^2 + (n_i+n_j) T\right) + \rho^3 (n_i + n_j) T + \rho^3\sum_{\alpha\neq \beta, \alpha, \beta \notin
 \{\alpha_i, \alpha_j\} } n_\alpha n_\beta
\end{align*}
Note that according to our assumptions \(\theta n_{\min} \to \infty\). We conclude from \eqref{dbs_eqn_size_one_nbhd} and \eqref{dbs_eqn_size_overlap} for \(\epsilon < \theta\)
\begin{align}
 \mathbb E [S_{ij}|\mathcal C_i^j, \mathcal C_j^i] 
 &= \theta \sum_\alpha |\mathcal C_i^j \cap \mathcal C_\alpha^*| |C_j^i \cap \mathcal C_\alpha^*| + \rho \sum_{\alpha\neq \beta} |\mathcal C_i^j\cap\mathcal C_\alpha^*||\mathcal C_j^i\cap\mathcal C_\alpha^*| + \mathcal O(\theta |\mathcal C_i^j\cap\mathcal C_j^i|\nonumber)\\
 &=\theta\sum_\alpha \theta_i^\alpha \theta_j^\alpha n_\alpha^2 + \rho \sum_{\alpha\neq\beta} \theta_i^\alpha \theta_j^\beta n_\alpha n_\beta +\mathcal O(\theta^3 n_{\max}+ \theta^2 \epsilon n_{\max}^2)\nonumber \\
 &= \begin{cases}
 a_i   +\mathcal O(\theta^3 n_{\max}+ \theta^2 \epsilon n_{\max}^2)& ,\exists \alpha: i, j\in \mathcal C_\alpha^*\\
    c_{ij} + \mathcal O(\theta^3 n_{\max} + \theta^2 \epsilon n_{\max}^2) & ,\text{otherwise}
    \end{cases}\nonumber 
\end{align}
for 
\begin{align*}
 a_i &\defined a_i^\circ + r_i\text{ and}\\
c_{ij} &\defined c_{ij}^\circ + r_{ij}. 
\end{align*}
Similarly 
\begin{align*}
 \mathbb E [S_{ii}|\mathcal C_i] &= a_i + \mathcal O(\theta^2 n_{\max} + \theta^2 \epsilon n_{\max}^2)\nonumber
\end{align*}
as well as
\begin{align*}
\mathbb E[N_{ij}|\mathcal C_i, \mathcal C_j]
&= N_{ij}\nonumber \\
&= \left(\sum_\alpha |\mathcal C_i\cap \mathcal C_\alpha^*|\right)\left(\sum_\alpha |\mathcal C_j\cap \mathcal C_\alpha^*|\right)\nonumber\\
&= d_{ij}  + \mathcal O(\theta^2 n_{\max} + \epsilon \theta n_{\max}^2)\nonumber
\end{align*}
for
\begin{align*}
d_{ij} &= (N_i^\circ + \rho T)(N_j^\circ + \rho T),\\
N_i^\circ &= \theta n_i + \rho n_j \text{ and}\\
N_j^\circ &= \rho n_i + \theta n_j.
\end{align*}
Analogously as before, we introduce \begin{align}
    R\defined \theta n_{\max} + \epsilon \theta n_{\max}^2\nonumber
    \end{align}
 and get
\begin{align*}
\mathbb E[ \widetilde \theta_{ii}|\mathcal C_i] &=  \frac{a_i +\mathcal O\left(\theta R \right)}{d_{ii} +\mathcal O\left(R\right)}.
\end{align*}
For \begin{align}
  \delta\defined \max\left\{\frac{1}{\theta n_{\max}}, \frac{\epsilon}{\theta} \right\}\nonumber
    \end{align}
 and under the assumption \(\delta \ll 1\) we can rewrite 
\begin{align*}
\mathbb E[ \widetilde \theta_{ii}|\mathcal C_i]&=\frac{a_i}{d_{ii}}\left(1+\mathcal O(\delta)\right)
\end{align*}
and analogously 
\begin{align*}
\mathbb E[ \widetilde \theta_{ij}|\mathcal C_i, \mathcal C_j] &= \begin{cases}\frac{a_i}{d_{ii}}\left(1+\mathcal O(\delta)\right)& ,\exists \alpha: i, j\in \mathcal C_\alpha^* \\
                    \frac{c_{ij}}{d_{ij}} \left(1+\mathcal O(\delta)\right)&  ,\text{otherwise.}
                    \end{cases}
\end{align*}
Let us consider the case \(i\in\alpha_i\). In this case, we have
\begin{align*}
 \mathbb E[ \widetilde \theta_{ii} - \widetilde \theta_{ij} | \mathcal C_i, \mathcal C_j]  &= \frac{a_i}{d_{ii}}\left(1+\mathcal O(\delta)\right) - \frac{a_i}{d_{ii}}\left(1+\mathcal O(\delta) \right)\nonumber\\
  &=\mathcal O\left(\delta \theta\right).\nonumber
\end{align*}
Since \(\widetilde \theta_{i\lor j}\) is the weighted mean of \(\widetilde \theta_{ii}\), \(\widetilde \theta_{jj}\) and \(\widetilde \theta_{ij}\), we conclude also
\begin{align}
\mathbb E[ \widetilde \theta_{i\lor j} - \widetilde \theta_{ij} | \mathcal C_i, \mathcal C_j]  &= \mathcal O\left(\delta \theta\right). \label{dbs_ineq_diff_upper_bd_simplified}
\end{align}
Next, we consider the other case \(i\notin\alpha_i\). First, we need to calculate a lower bound for the difference \(\frac{a_i}{d_{ii}}-\frac{c_i}{d_{ij}}\). We have
\begin{align}
\frac{a_i}{d_{ii}}-\frac{c_i}{d_{ij}} &= \frac{a_i^\circ + r_i}{(N_i^\circ + \rho T)^2}-\frac{c_{ij}^\circ + r_{ij}}{(N_i^\circ + \rho T)(N_j^\circ + \rho T)} \nonumber\\
&= \frac{1}{(N_i^\circ + \rho T)^2(N_j^\circ + \rho T)} \left(A_{ij} + B_{ij}+C_{ij} \right)\nonumber
\end{align}
for
\begin{align*}
A_{ij} &= a_i^\circ N_j^\circ - c_{ij}^\circ N_i^\circ , \\
B_{ij} &=\rho T(a_i^\circ - c_{ij}^\circ)\text{ and}\\
C_{ij} &= r_i(N_j^\circ - N_i^\circ) +(r_i - r_{ij}) (N_i^\circ + \rho T).\\
\end{align*}
We have 
\begin{align*}
 r_i-r_{ij} &= \theta \rho^2 (n_i - n_j) T + \rho^3 (n_j - n_i ) T\\
 &= \rho^2 T(\theta n_i + \rho n_j - (\theta n_j +\rho n_i))\\
 &= \rho^2T(N_i^\circ - N_j^\circ).
\end{align*}
Consequently,
\begin{align*}
C_{ij} &= (N_i^\circ - N_j^\circ) (\rho^2T N_i^\circ + \rho^3 T^2 - r_i)\\
 &= -(N_i^\circ - N_j^\circ) \left((\theta - \rho) \rho^2 \sum_{\alpha\notin\{\alpha_i, \alpha_j\}} n_\alpha^2 + \theta \rho^2 n_i T + \rho^3 n_j T\right)\\
 &= -(N_i^\circ - N_j^\circ) \left((\theta - \rho) \rho^2 \sum_{\alpha\notin\{\alpha_i, \alpha_j\}} n_\alpha^2 + \rho^2 N_i^\circ T \right)
\end{align*}
and
\begin{align*}
 C_{ij} + C_{ji} &= (N_i^\circ - N_j^\circ) (-\rho^2 N_i^\circ T + \rho^2 N_j^\circ T)\\
 &=-(N_i^\circ - N_j^\circ)^2 \rho^2 T\\
  &= - (\theta - \rho)^2 (n_i - n_j)^2 \rho^2 T.
\end{align*}
Moreover,
\begin{align*}
 B_{ij} + B_{ji} &= \rho T(a_i^\circ + a_j^\circ - 2c_{ij}^\circ)\\
 &= \rho  T\left[\theta^3 (n_i^2+n_j^2)+\theta \rho^2 (n_i^2 + n_j^2 + 4n_i n_j) - 2 \theta^2 \rho(n_i^2 + n_j^2 + n_i n_j) - 2 \rho^3 n_i n_j\right]\\
 &=\rho T\left[\rho (\theta - \rho )^2 (n_i - n_j)^2 + (\theta - \rho )^3 (n_i^2 + n_j^2)\right].
\end{align*}
The first term \(A_{ij}\) simplifies to
\begin{align*}
 A_{ij} &= a_i^\circ N_j^\circ - c_{ij}^\circ N_i^\circ \\
 &= [ \theta^3 n_i^2 + \theta \rho^2(n_j^2 + 2n_i n_j) ] (\rho n_i + \theta n_j) - [ \theta^2 \rho (n_i^2+n_j^2+n_in_j)+ \rho^3 n_i n _j] (\theta n_i + \rho n_j)\\
 &=\theta^4 n_i^2 n_j + \theta^2 \rho^2 n_i n_j^2 + \theta \rho^3 n_i n_j^2 +\theta \rho^3 n_i^2 n_j - \theta^3 \rho n_i n_j^2 - \theta^3 \rho n_i^2 n_j - \theta^2 \rho^2 n_i^2 n_j- \rho^4 n_i n_j^2\\
 &=n_i n_j^2 \underbrace{(\theta ^4 - 2\theta^3 \rho + 2\theta \rho^3 - \rho^4)}_{=(\theta + \rho)(\theta - \rho)^3} + n_i n_j (n_i - n_j)\underbrace{(\theta^4 - \theta^3 \rho- \theta^2 \rho^2 + \theta \rho^3)}_{=\theta (\theta + \rho )(\theta - \rho)^2}\\
 &= n_i n_j^2 (\theta + \rho) (\theta - \rho)^3 + D_{ij}
\end{align*}
for 
\begin{align*}
 D_{ij} &= (n_i - n_j) n_i n_j \theta  (\theta + \rho) (\theta - \rho) ^2.
\end{align*}
Note that
\begin{align*}
\underbrace{D_{ij} + D_{ji}}_{=0} + B_{ij} + B_{ji} + C_{ij} + C_{ji} &= \rho(\theta - \rho)^3 (n_i^2 + n_j^2) T.\\
\end{align*}
So w.l.o.g. we can assume \[D_{ij} + B_{ij} + C_{ij} \geq  \frac{1}{2} \rho(\theta - \rho)^3 (n_i^2 + n_j^2) T, \] implying
\begin{align*}
A_{ij}+B_{ij}+C_{ij} > \frac{1}{2} \rho(\theta - \rho)^3  (n_i^2 + n_j^2) T.
\end{align*}
We conclude w.l.o.g. (possibly after exchanging the indices \(i\) and \(j\)) 
\begin{align*}
 \frac{a_i}{d_{ii}} - \frac{c_i}{d_{ij}} &\gtrsim \frac{(\theta - \rho)^3}{\theta^2} \frac{n_{\min}^2}{n_{\max}^2}
\end{align*}
and thus
\begin{align*}
 \mathbb E[ \widetilde \theta_{ii} - \widetilde \theta_{ij} | \mathcal C_i, \mathcal C_j]  &= \frac{a_i}{d_{ii}}\left(1+\mathcal O(\delta)\right) - \frac{c_{ij}}{d_{ij}}\left(1+\mathcal O(\delta) \right)\nonumber\\
  & \gtrsim \frac{(\theta - \rho)^3}{\theta^2} \frac{n_{\min}^2}{n_{\max}^2} + \mathcal O(\delta \theta).
\end{align*}
Since \(\widetilde \theta_{i\lor j}\) is the weighted mean of \(\widetilde \theta_{ii}\), \(\widetilde \theta_{jj}\) and \(\widetilde \theta_{ij}\), there exists indices \((i', j')\in \{(i, i), (j, j), (i, j)\}\) such that
\begin{align}
\mathbb E[ \widetilde \theta_{i\lor j} - \widetilde\theta_{i'j'} | \mathcal C_i, \mathcal C_j]  &\gtrsim \frac{(\theta - \rho)^3}{\theta^2} \frac{n_{\min}^2}{n_{\max}^2} + \mathcal O(\delta \theta).\label{dbs_ineq_diff_lower_bd_simplified}
\end{align}
Similar as in the proof of Theorem \ref{thm_original_AWCD}, we deduce from our assumptions \(\theta^3 n_{\min}^2 \gg \log n_{\min}\) and \(\theta \leq \frac{1}{2}\) as well as Lemma \ref{lem_bernstein} that
\begin{itemize}
 \item the estimates \(\widetilde \theta_{\cdot}\) (including \(\widetilde \theta_{ii}\), \(\widetilde \theta_{jj}\), \(\widetilde \theta_{ij}\) and \(\widetilde \theta_{i\lor j}\)) are with large probability inside an intervall of the form \([\frac{\rho}{2},  \theta+\frac{\rho}{2}]\), implying that the corresponding Fisher information of a Bernoulli variable with such a mean is up to bounded constants given by \(\theta^{-1}\) and
 \item with large probability, \begin{align}
 |\widetilde \theta_{\cdot} -   \mathbb E [\widetilde \theta_{\cdot}|\mathcal C_i, \mathcal C_j] |&\lesssim (\log n_{\min})^\frac{1}{2} \theta^{-\frac{1}{2}} n_{\min}^{-1}. \nonumber
\end{align}
\end{itemize}
Combining the above with the quadratic Taylor expansion of the Kullbach-Leibler and \eqref{dbs_ineq_diff_lower_bd_simplified}, we conclude in the case where \(i\) and \(j\) belong to different communities 
\begin{align}
\theta^{\frac{1}{2}}\mathcal K^\frac{1}{2}(\widetilde \theta_{i'j'}, \widetilde\theta_{i\lor j})&\gtrsim    |\widetilde \theta_{i'j'}- \widetilde \theta_{i\lor j}| \nonumber\\
&\geq \left|\mathbb E[\widetilde\theta_{i'j'} | \mathcal C_i, \mathcal C_j]-\mathbb E[\widetilde\theta_{i\lor j} | \mathcal C_i, \mathcal C_j]\right| -\left\vert\widetilde\theta_{i'j'} - \mathbb E[\widetilde\theta_{i'j'} | \mathcal C_i, \mathcal C_j]\right\vert - \left\vert\widetilde\theta_{i\lor j} - \mathbb E[\widetilde\theta_{i\lor j} | \mathcal C_i, \mathcal C_j]\right\vert\nonumber\\
&\gtrsim  \frac{(\theta - \rho)^3}{\theta^2} \frac{n_{\min}^2}{n_{\max}^2} + \mathcal O\left(\delta \theta + (\log n_{\min})^\frac{1}{2}\theta^{-\frac{1}{2}} n_{\min}^{-1}\right),\nonumber
\end{align}
whereas in the other case we conclude analogously from \eqref{dbs_ineq_diff_upper_bd_simplified}
\begin{align}
\theta^{\frac{1}{2}}\mathcal K^\frac{1}{2}(\widetilde\theta_{\cdot}, \widetilde\theta_{i\lor j})&\lesssim \delta\theta + (\log n_{\min})^\frac{1}{2}\theta^{-\frac{1}{2}} n_{\min}^{-1}.\nonumber
\end{align}
To ensure that the Kullback-Leibler divergence in the second case is significantly smaller than in the first case, it will suffice to check 
\begin{align}
 \delta \theta&\ll \frac{(\theta - \rho)^3}{\theta^2} \frac{n_{\min}^2}{n_{\max}^2}\nonumber\\
\Leftrightarrow \qquad \max\left\{\frac{1}{\theta n_{\max}}, \frac{\epsilon}{\theta} \right\}&\ll \left(\frac{\theta-\rho}{\theta}\right)^3 \frac{n_{\min}^2}{n_{\max}^2}\label{dbs_condition1}
\end{align}
and
\begin{align}
 \sqrt{\frac{\log n_{\min}}{\theta n_{\min}^2}} &\ll \frac{(\theta - \rho)^3}{\theta^2} \frac{n_{\min}^2}{n_{\max}^2}. \label{dbs_condition2}
\end{align}
Those are all conditions in the proof that remain to be checked except the large probability assumption 
\begin{align}
 \min \left\{ \frac{\epsilon^2 n_{\min}}{\theta \log n_{\min}}, \frac{\epsilon n_{\min}}{\log n_{\min}} \right\} \gg 1. \label{dbs_condition3}
\end{align}
An \(\epsilon > 0\) satisfying \eqref{dbs_condition1} and \eqref{dbs_condition3} exists as long as 
\begin{align}
 \frac{\log n_{\min}}{n_{\min}} + \sqrt{\frac{\theta \log n_{\min}}{n_{\min}}} \ll \frac{(\theta - \rho )^3}{\theta^2} \frac{n_{\min}^2}{n_{\max}^2},\label{dbs_condition21}
\end{align}
whereas the \(\epsilon\)-independent part of conditions \eqref{dbs_condition1} and \eqref{dbs_condition2} can be summarized by 
\begin{align}
\frac{1}{n_{\max}} + \sqrt{\frac{\log n_{\min}}{\theta n_{\min}^2}} \ll \frac{(\theta - \rho)^3}{\theta^2} \frac{n_{\min}^2}{n_{\max}^2}. \label{dbs_condition22}
\end{align}
By simple calculus, we can verify that our assumption
\[\theta - \rho\gg \left(\frac{n_{\max}}{n_{\min}}\right)^\frac{2}{3}\max\{n_{\min}^{-\frac{1}{3}} \theta^\frac{1}{2}(\log n_{\min})^\frac{1}{6},   n_{\min}^{-\frac{1}{6}} \theta^\frac{5}{6} (\log n_{\min})^\frac{1}{6}\}\]
implies conditions \eqref{dbs_condition21} and \eqref{dbs_condition22}. Note that \(d_{ii}\), \(d_{jj}\) and \(d_{ij}\) only differ by bounded factors. Taking into account the concentration result \(N_{\cdot} = d_{\cdot}(1+\mathcal O(\delta))\) with \(\delta \ll 1\), this also ensures consistency of the test - provided a proper threshold is given. 
\end{proof}


\begin{proof}[Proof of Proposition \ref{Proposition_different_block_size_and_parameter}]
We proceed analogously as previously. Suppose \(\epsilon > 0\) and \(\alpha\in\{1, 2\}\). According to Lemma \ref{lem_bernstein_homogeneous} we have
\begin{align} |\mathcal C_i^j\cap \mathcal C^*_\alpha| = \theta_i^\alpha n_\alpha + \mathcal O\left(\epsilon n_{\max} + \theta_{\max} \right)   \label{dp_eqn_size_one_nbhd}\end{align}
for \begin{align}\theta_i^\alpha = \begin{cases}
                         \theta_\alpha &, i \in \mathcal C^*_\alpha\\
                         \rho &, i\notin \mathcal C^*_\alpha
                        \end{cases}
\nonumber
\end{align}
on an event of probability at least
\[1 - 2\exp \left( - \frac{\epsilon^2 (n_{\min} - 2)}{2\theta_{\max} + \epsilon}\right)\text{.}\]
Similarly, we conclude for \(i\neq j\)
\begin{align}
|\mathcal C_i^j \cap \mathcal C_j^j|= \mathcal O(\theta_{\max}^2 n_{\max} + \epsilon n_{\max}) \label{dp_eqn_size_overlap} 
\end{align}
on an event of probability of at least the same lower bound. In the following, we will restrict to an event where \eqref{dp_eqn_size_one_nbhd} and \eqref{dp_eqn_size_overlap} are satisfied for any \(i \neq j\). By union bound, this event of asymptotic probability 1 as long as
\begin{align}
 \min \left\{ \frac{\epsilon^2 n_{\min}}{\theta_{\max}}, \epsilon n_{\min} \right\} \gg \log n_{\min}.
\nonumber 
\end{align}
Note that according to our assumptions \(\theta_{\min} n_{\min}\to \infty\). We conclude from \eqref{dp_eqn_size_one_nbhd} and \eqref{dp_eqn_size_overlap} for \(\epsilon < \theta_{\min}\)
\begin{align}
 \mathbb E [S_{ij}|\mathcal C_i^j, \mathcal C_j^i] 
 &= \sum_\alpha \theta_\alpha |\mathcal C_i^j \cap  \mathcal C_\alpha^*| | \mathcal C_j^i \cap \mathcal C_\alpha^*| + \rho \sum_{\alpha\neq \beta} |\mathcal C_i^j\cap \mathcal C_\alpha^*||\mathcal C_j^i\cap \mathcal C_\alpha^*| + \mathcal O(\theta_{\max} |\mathcal C_i^j\cap \mathcal C_j^i|\nonumber)\\
 &=\sum_\alpha \theta_\alpha \theta_i^\alpha \theta_j^\alpha n_\alpha^2 + p \sum_{\alpha\neq\beta} \theta_i^\alpha \theta_j^\beta n_\alpha n_\beta +\mathcal O(\theta_{\max}^3 n_{\max}+ \theta_{\max}^2 \epsilon n_{\max}^2)\nonumber \\
 &= \begin{cases}
 a_{11}   +\mathcal O(\theta^3 n_{\max}+ \theta^2 \epsilon n_{\max}^2)&, i, j\in \mathcal C_1^*\\
 a_{22}   +\mathcal O(\theta^3 n_{\max}+ \theta^2 \epsilon n_{\max}^2)&, i, j\in \mathcal C_2^*\\
a_{12} + \mathcal O(\theta^3 n_{\max} + \theta^2 \epsilon n_{\max}^2) & ,\text{otherwise}
    \end{cases} \nonumber
\end{align}
for 
\begin{align*}
 a_{11} &\defined \theta_1^3 n_1^2+ 2\theta_1 \rho^2 n_1 n_2 + \theta_2 \rho^2 n_2^2, \\
  a_{22} &\defined \theta_2^3 n_2^2+ 2\theta_2 \rho^2 n_1 n_2 + \theta_1 \rho^2 n_1^2\text{ and} \\
a_{12} &\defined \theta_1^2 \rho n_1^2 + \theta_2^2\rho n_2^2 + \theta_1 \theta_2 \rho n_1 n_2 + \rho^3 n_1 n_2.
\end{align*}
Similarly 
\begin{align*}
 \mathbb E [S_{ii}|\mathcal C_i] &= a_i + \mathcal O(\theta_{\max}^2 n_{\max} + \theta_{\max}^2 \epsilon n_{\max}^2)\nonumber
\end{align*}
as well as
\begin{align*}
\mathbb E[N_{ij}|\mathcal C_i, \mathcal C_j]
&= N_{ij}\nonumber \\
&= \left(\sum_\alpha |\mathcal C_i\cap  \mathcal C_\alpha^*|\right)\left(\sum_\alpha |\mathcal C_j\cap  \mathcal C_\alpha^*|\right)\nonumber\\
&= \begin{cases}
    d_{11}  + \mathcal O(\theta_{\max}^2 n_{\max} + \epsilon \theta_{\max} n_{\max}^2)\nonumber &,i,j\in\mathcal C_1^*\\
        d_{22}  + \mathcal O(\theta_{\max}^2 n_{\max} + \epsilon \theta_{\max} n_{\max}^2)\nonumber &,i,j\in\mathcal C_2^*\\
    d_{12}  + \mathcal O(\theta_{\max}^2 n_{\max} + \epsilon \theta_{\max} n_{\max}^2)\nonumber &,\text{otherwise}
   \end{cases}
\end{align*}
for
\begin{align*}
d_{11} &\defined  (\theta_1n_1 + \rho n_2)^2,\\
d_{22} &\defined  (\theta_2n_2 + \rho n_1)^2\text{ and}\\
d_{12} &\defined  (\theta_1n_1 + \rho n_2) (\theta_2n_2 + \rho n_1).
\end{align*}
For \begin{align}
    R\defined \theta_{\max} n_{\max} + \epsilon \theta_{\max} n_{\max}^2 \nonumber
    \end{align}
 we conclude
\begin{align*}
\mathbb E[ \widetilde \theta_{ii}|\mathcal C_i] &=  \frac{a_i +\mathcal O\left(\theta_{\max} R \right)}{d_{ii} +\mathcal O\left(R\right)}.
\end{align*}
For \begin{align}
  \delta\defined \left(\frac{\theta_{\max}}{\theta_{\min}}\right)^2 \left( \frac{n_{\max}}{n_{\min}}\right)^2 \max\left\{\frac{1}{\theta_{\min} n_{\max}}, \frac{\epsilon}{\theta_{\min}} \right\}\nonumber 
    \end{align}
 and under the assumption \(\delta \ll 1\) we can rewrite 
\begin{align*}
\mathbb E[ \widetilde \theta_{ij}|\mathcal C_i, \mathcal C_j] &= 
\begin{cases}\frac{a_{11}}{d_{11}}\left(1+\mathcal O(\delta)\right)& ,i, j\in \mathcal C_1^* \\
\frac{a_{22}}{d_{22}}\left(1+\mathcal O(\delta)\right)& ,i, j\in \mathcal C_2^* \\
                  \frac{a_{12}}{d_{22}}\left(1+\mathcal O(\delta)\right)&  ,\text{otherwise.}
                    \end{cases}
\end{align*}
Let us consider the case \(i\in\alpha_j\). W.l.o.g. let us assume \(\alpha_i = 1\). Then
\begin{align*}
 \mathbb E[ \widetilde \theta_{ii} - \widetilde \theta_{ij} | \mathcal C_i, \mathcal C_j]  &= \frac{a_{11}}{d_{11}}\left(1+\mathcal O(\delta)\right) - \frac{a_{11}}{d_{11}}\left(1+\mathcal O(\delta) \right)\nonumber\\
  &=\mathcal O\left(\delta \theta_{\max}\right).\nonumber
\end{align*}
Since \(\widetilde \theta_{i\lor j}\) is the weighted mean of \(\widetilde \theta_{ii}\), \(\widetilde \theta_{jj}\) and \(\widetilde \theta_{ij}\), we also conclude 
\begin{align}
\mathbb E[ \widetilde \theta_{i\lor j} - \widetilde \theta_{ij} | \mathcal C_i, \mathcal C_j]  &= \mathcal O\left(\delta \theta_{\max}\right). \label{dp_ineq_diff_upper_bd_simplified}
\end{align}
Next, we consider the other case \(i\notin\alpha_i\). W.l.o.g. let us assume \(\alpha_i = 1\) and \(\alpha_j = 2\). 
We compute
\begin{align*}
\left(\frac{a_{11}}{d_{11}} - \frac{a_{22}}{d_{22}}\right) + \left(\frac{a_{11}}{d_{11}} - \frac{a_{12}}{d_{12}}\right)  &=\frac{x}{  (\theta_1n_1 + \rho n_2)^2 (\theta_2n_2 + \rho n_1)^2}
\end{align*}
with
\begin{align*}
 x &= 2a_{11}  (\theta_2n_2 + \rho n_1)^2 - a_{22} (\theta_1n_1 + \rho n_2)^2 - a_{12} (\theta_1n_1 + \rho n_2) (\theta_2n_2 + \rho n_1)\\
 &= n_1n_2(\theta_1\theta_2 - \rho^2) \left[(\theta_2 - \rho)(\theta_1^2 - \rho^2) n_1 n_2 + R_{ij}\right]
\end{align*}
for 
\begin{align*}
R_{ij} &= 3\rho (\theta_2 - \rho) (\theta_1 n_1^2 - \theta_2 n_2^2) + (\theta_1-\theta_2) \left[n_1n_2(\theta_1\theta_2 + \theta_1 \rho + \rho^2) + 3\rho\theta_1n_1^2 \right].
\end{align*}
Note that
\begin{align*}
 R_{ij} + R_{ji} &= 3\rho (\theta_2 - \theta_1)(\theta_1 n _1^2 - \theta_2 n_2^2) + (\theta_1-\theta_2) \left[\rho n_1 n_2(\theta_1- \theta_2) + 3\rho(\theta_1 n_1^2- \theta_2 n_2^2)\right]\\
 &= \rho n_1 n_2 (\theta_1 - \theta_2)^2\\
 & \geq 0.
\end{align*}
So w.l.o.g. (after possibly exchanging the two indices) we can assume 
\begin{align*}
\left(\frac{a_{11}}{d_{11}} - \frac{a_{22}}{d_{22}}\right) + \left(\frac{a_{11}}{d_{11}} - \frac{a_{12}}{d_{12}}\right)  &\geq \frac{(\theta_1\theta_2 - \rho^2) (\theta_1 - \rho)(\theta_2 - \rho)(\theta_1 +\rho) n_1^2 n_2^2}{  (\theta_1n_1 + \rho n_2)^2 (\theta_2n_2 + \rho n_1)^2}\\
& \gtrsim \left(\frac{\theta_{\min}}{\theta_{\max}}\right)^2 \left(\frac{ n_{\min} }{ n_{\max}}\right)^2 \frac{(\theta_{\max}  - \rho)(\theta_{\min} - \rho)^2 }{\theta_{\max} ^2}
\end{align*}
and thus
\begin{align*}
 &\mathbb E[( \widetilde \theta_{ii} - \widetilde \theta_{jj})+ ( \widetilde \theta_{ii} - \widetilde \theta_{ij} )| \mathcal C_i, \mathcal C_j]  \\
 =   & \left(\frac{a_{11}}{d_{11}} \left(1+\mathcal O(\delta) \right) - \frac{a_{22}}{d_{22}} \left(1+\mathcal O(\delta) \right) \right)+  \left(\frac{a_{11}}{d_{11}} \left(1+\mathcal O(\delta) \right)  - \frac{a_{12}}{d_{12}} \left(1+\mathcal O(\delta) \right) \right)  \nonumber\\
 \gtrsim  &  \left(\frac{\theta_{\min}}{\theta_{\max}}\right)^2 \left(\frac{ n_{\min} }{ n_{\max}}\right)^2 \frac{(\theta_{\max}  - \rho)(\theta_{\min} - \rho)^2 }{\theta_{\max} ^2}+ \mathcal O(\delta \theta_{\max}).
\end{align*}
Since \(\widetilde \theta_{i\lor j}\) is the weighted mean of \(\widetilde \theta_{ii}\), \(\widetilde \theta_{jj}\) and \(\widetilde \theta_{ij}\), there exists indices \((i', j')\in \{(i, i), (j, j), (i, j)\}\) such that
\begin{align}
\mathbb E[ \widetilde \theta_{i\lor j} - \widetilde\theta_{i'j'} | \mathcal C_i, \mathcal C_j]  &\gtrsim\left(\frac{\theta_{\min}}{\theta_{\max}}\right)^2 \left(\frac{ n_{\min} }{ n_{\max}}\right)^2 \frac{(\theta_{\max}  - \rho)(\theta_{\min} - \rho)^2 }{\theta_{\max} ^2}+ \mathcal O(\delta \theta_{\max}).\label{dp_ineq_diff_lower_bd_simplified}
\end{align}
Similar as in the proof of Theorem \ref{thm_original_AWCD}, we deduce from our assumptions \(\theta_{\min}^3 n_{\min}^2 \gg \log n_{\min}\) and \(\theta_{\max} \leq \frac{1}{2}\) as well as Lemma \ref{lem_bernstein} that
\begin{itemize}
 \item the estimates \(\widetilde \theta_{\cdot}\) (including \(\widetilde \theta_{ii}\), \(\widetilde \theta_{jj}\), \(\widetilde \theta_{ij}\) and \(\widetilde \theta_{i\lor j}\)) are with large probability inside an intervall of the form \([\frac{\rho}{2}, \theta_{\max}+\frac{\rho}{2}]\), implying that the corresponding Fisher information of a Bernoulli variable with such a mean is up to bounded constants bounded from below by \(\theta_{\max}^{-1}\) and
 \item with large probability, \begin{align}
 |\widetilde \theta_{\cdot} -   \mathbb E [\widetilde \theta_{\cdot}|\mathcal C_i, \mathcal C_j] |&\lesssim (\log n_{\min})^\frac{1}{2} \theta_{\min}^{-\frac{1}{2}} n_{\min}^{-1}.\nonumber 
\end{align}
\end{itemize}
Combining the above with the quadratic taylor expansion of the Kullbach-Leibler and \eqref{dp_ineq_diff_lower_bd_simplified}, we conclude in the case where \(i\) and \(j\) belong to different communities 
\begin{align}
\theta_{\max}^{\frac{1}{2}}\mathcal K^\frac{1}{2}(\widetilde \theta_{i'j'}, \widetilde\theta_{i\lor j})&\gtrsim    |\widetilde \theta_{i'j'}- \widetilde \theta_{i\lor j}| \nonumber\\
&\geq \left|\mathbb E[\widetilde\theta_{i'j'} | \mathcal C_i, \mathcal C_j]-\mathbb E[\widetilde\theta_{i\lor j} | \mathcal C_i, \mathcal C_j]\right| -\left\vert\widetilde\theta_{i'j'} - \mathbb E[\widetilde\theta_{i'j'} | \mathcal C_i, \mathcal C_j]\right\vert - \left\vert\widetilde\theta_{i\lor j} - \mathbb E[\widetilde\theta_{i\lor j} | \mathcal C_i, \mathcal C_j]\right\vert\nonumber\\
&\gtrsim  \left(\frac{\theta_{\min}}{\theta_{\max}}\right)^2 \left(\frac{ n_{\min} }{ n_{\max}}\right)^2 \frac{(\theta_{\max}  - \rho)(\theta_{\min} - \rho)^2 }{\theta_{\max} ^2} + \mathcal O\left(\delta \theta_{\max} + (\log n_{\min})^\frac{1}{2}\theta_{\min}^{-\frac{1}{2}} n_{\min}^{-1}\right),\nonumber
\end{align}
whereas in the other case we conclude analogously from \eqref{dp_ineq_diff_upper_bd_simplified}
\begin{align}
\theta_{\max}^{\frac{1}{2}}\mathcal K^\frac{1}{2}(\widetilde\theta_{\cdot}, \widetilde\theta_{i\lor j})&\lesssim \delta \theta_{\max} + (\log n_{\min})^\frac{1}{2}\theta_{\min}^{-\frac{1}{2}} n_{\min}^{-1}. \nonumber
\end{align}
Recall that the concentration result \(N_{\cdot} = d_{\cdot}(1+\mathcal O(\delta))\) and with \(\delta \ll 1\) and note that \(d_{ii}\), \(d_{jj}\) and \(d_{ij}\) differ only by a factor \(\lesssim \left(\frac{\theta_{\max}}{\theta_{\min}}\right)^2 \left(\frac{n_{\max}}{n_{\min}}\right)^2\). We conclude that to show the consistency of the test under a proper choice of the threshold, it will suffice to check 
\begin{align}
 \delta \theta_{\max}&\ll  \left(\frac{\theta_{\min}}{\theta_{\max}}\right)^4 \left(\frac{ n_{\min} }{ n_{\max}}\right)^4 \frac{(\theta_{\max}  - \rho)(\theta_{\min} - \rho)^2 }{\theta_{\max} ^2}
 \nonumber\\
\Leftrightarrow \qquad  \max\left\{\frac{1}{\theta_{\min} n_{\max}}, \frac{\epsilon}{\theta_{\min}} \right\}&\ll \left(\frac{\theta_{\min}}{\theta_{\max}}\right)^6 \left(\frac{ n_{\min} }{ n_{\max}}\right)^6 \frac{(\theta_{\max}  - \rho)(\theta_{\min} - \rho)^2 }{\theta_{\max}^3} \label{dp_condition1}
\end{align}
and
\begin{align}
 \sqrt{\frac{\log n_{\min}}{\theta_{\min} n_{\min}^2}} &\ll  \left(\frac{\theta_{\min}}{\theta_{\max}}\right)^4 \left(\frac{ n_{\min} }{ n_{\max}}\right)^4 \frac{(\theta_{\max}  - \rho)(\theta_{\min} - \rho)^2 }{\theta_{\max} ^2}, \label{dp_condition2}
\end{align}
while also considering the large probability assumption 
\begin{align}
 \min \left\{ \frac{\epsilon^2 n_{\min}}{\theta_{\max} \log n_{\min}}, \frac{\epsilon n_{\min}}{\log n_{\min}} \right\} \gg 1. \label{dp_condition3}
\end{align}
An \(\epsilon > 0\) satisfying \eqref{dbs_condition1} and \eqref{dbs_condition3} exists as long as 
\begin{align}
 \frac{\log n_{\min}}{n_{\min}} + \sqrt{\frac{\theta_{\max} \log n_{\min}}{n_{\min}}} \ll \left(\frac{\theta_{\min}}{\theta_{\max}}\right)^7 \left(\frac{ n_{\min} }{ n_{\max}}\right)^6 \frac{(\theta_{\max}  - \rho)(\theta_{\min} - \rho)^2 }{\theta_{\max}^2}, \label{dp_condition21}
\end{align}
whereas the \(\epsilon\)-independent part of conditions \eqref{dbs_condition1} and \eqref{dbs_condition2} is implied by
\begin{align}
\left( \frac{\theta_{\max}}{\theta_{\min}} \right)\frac{1}{n_{\max}} + \sqrt{\frac{\log n_{\min}}{\theta_{\min} n_{\min}^2}} \ll \left(\frac{\theta_{\min}}{\theta_{\max}}\right)^6 \left(\frac{ n_{\min} }{ n_{\max}}\right)^6 \frac{(\theta_{\max}  - \rho)(\theta_{\min} - \rho)^2 }{\theta_{\max} ^2}. \label{dp_condition22}
\end{align}
By simple calculus we can verify that our assumption \[\theta_{\min} - \rho\gg \left(\frac{\theta_{\max}}{\theta_{\min}}\right)^\frac{7}{3}\left(\frac{n_{\max}}{n_{\min}}\right)^2\max\left\{n_{\min}^{-\frac{1}{3}} \theta_{\max}^\frac{1}{2}(\log n_{\min})^\frac{1}{6},   n_{\min}^{-\frac{1}{6}} \theta_{\max}^\frac{5}{6} (\log n_{\min})^\frac{1}{6}\right\}\] implies conditions \eqref{dp_condition21} and \eqref{dp_condition22}.
\end{proof}


\appendix

\section{Equivalent formulations and log$_n$-log$_n$ plot for rates of consistency}
 \label{ref_appendix}

Recall that for $k=1$, Theorem \ref{thm_original_AWCD} guarantees strong consistency as long as 
 \begin{equation}\theta - \rho\gg \max\{n^{-\frac{1}{3}} \theta^\frac{1}{2}(\log n)^\frac{1}{6},   n^{-\frac{1}{6}} \theta^\frac{5}{6} (\log n)^\frac{1}{6}\}. \label{condition232} \end{equation}
Considering \(\theta - \rho < \theta\), condition \eqref{condition232} implies \(\theta \gg n^{-\frac{2}{3}} (\log n)^\frac{1}{3}\). This is equivalent to \eqref{condition232} in case \(\frac{\theta}{\rho}\equiv C\). Furthermore, in view of
\begin{align*}
 n^{-\frac{1}{3}} \theta^\frac{1}{2}(\log n)^\frac{1}{6} &>   n^{-\frac{1}{6}} \theta^\frac{5}{6} (\log n)^\frac{1}{6}\\ \Leftrightarrow \theta &< n{-\frac{1}{2}}
\end{align*}
we can rewrite the consistency condition \eqref{condition232} as
\begin{align*}
\theta - \rho \gg \begin{cases}
n^{-\frac{1}{6}} \theta^\frac{5}{6} (\log n)^\frac{1}{6} &, n^{-\frac{1}{2}} \lesssim \theta < \frac{1}{2} \\
n^{-\frac{1}{3}} \theta^\frac{1}{2}(\log n)^\frac{1}{6} &, n^{-\frac{2}{3}} (\log n)^\frac{1}{3} \lesssim \theta \lesssim n^{-\frac{1}{2}},
               \end{cases}
\end{align*}
or equivalently, using the notation \(\theta = n^\gamma\)
\begin{align*}
&\log_n(\theta- \rho)  \color{gray}-\frac{c}{\log n}\color{black} \geq \\ &\begin{cases}
 \frac{5 \gamma - 1}{6} \color{gray}+ \frac{1}{6} \log_n(\log n) \color{black}&, -\frac{1}{2} \color{gray}+ \mathcal O((\log n)^{-1})\color{black} \leq \gamma <0 + \color{gray} \log_n\left(\frac{1}{2}\right)\color{black}\\
\frac{3 \gamma - 2}{6} + \frac{1}{6} \log_n (\log n) & ,-\frac{2}{3} \color{gray}+ \frac{1}{3}\log_n(\log n) \color{black}\leq \gamma \color{gray} + \mathcal O((\log n)^{-1})\color{black}\leq -\frac{1}{2}                       
\end{cases}
\end{align*}
for a large enough constant \(c\in\mathbb R\).
Up to the logarithmic terms (which are vanishing for \(n\to\infty\)), this can be visualized in a plot of \(\log_n\left(\frac{\theta-\rho}{\theta}\right)\) versus \(\log_n(\theta)\) by the area of a polygon consisting of the following points:
\begin{align*}
 &\left( 0 , 0 \right)  \\
&  \left( 0, -\frac{1}{6} \right)  \\
  & \left( -\frac{1}{2}, -\frac{1}{12} \right)  \\
  &  \left( -\frac{2}{3}, 0 \right)  
\end{align*}
 
Using the original algorithm AWCD\(^\circ\) instead, we have the sufficient condition for strong consistency
\begin{align}
 \theta - \rho \gg \max \{ n^{-\frac{1}{3}}\theta^\frac{1}{3}, n^{-\frac{1}{6}}\theta^\frac{5}{6} (\log n)^\frac{1}{6}\}. \label{condition242}
\end{align}
Considering \(\theta - \rho < \theta\), condition \eqref{condition242} implies \(\theta \gg n^{-\frac{1}{2}}\). This is equivalent to \eqref{condition242} in case \(\frac{\theta}{\rho}\equiv C\). Furthermore, because of
\begin{align*}
 n^{-\frac{1}{3}} \theta^\frac{1}{3} &>   n^{-\frac{1}{6}} \theta^\frac{5}{6} (\log n)^\frac{1}{6}\\ \Leftrightarrow \theta &< n^{-\frac{1}{3}} (\log n)^{-\frac{1}{3}}
\end{align*}
we can rewrite the consistency condition \eqref{condition242} as
\begin{align*}
\theta - \rho \gg \begin{cases}
n^{-\frac{1}{6}} \theta^\frac{5}{6} (\log n)^\frac{1}{6} &, n^{-\frac{1}{3}} (\log n)^{-\frac{1}{3}} \lesssim \theta < \frac{1}{2} \\
n^{-\frac{1}{3}} \theta^\frac{1}{3} &, n^{-\frac{1}{2}}  \lesssim \theta \lesssim n^{-\frac{1}{3}} (\log n)^{-\frac{1}{3}}
               \end{cases}
\end{align*}
or equivalently, using the notation \(\theta = n^\gamma\)
\begin{align*}
&\log_n(\theta- \rho)  \color{gray}-\frac{c}{\log n}\color{black}\geq\\ & \begin{cases}
 \frac{5 \gamma - 1}{6} \color{gray}+ \frac{1}{6} \log_n(\log n) &\color{black}, -\frac{1}{3}  \color{gray}-\frac{1}{3}\log_n(\log n) + \mathcal O((\log n)^{-1}) \color{black} \leq \gamma < 0 + \color{gray}\log_n\left(\frac{1}{2}\right)\color{black}\\
\frac{ \gamma -1}{3}  & ,-\frac{1}{2}  \leq \gamma  \color{gray} + \mathcal O((\log n)^{-1}) \color{black} \leq -\frac{1}{3}  \color{gray} -\frac{1}{3}\log_n(\log n)      \color{black}      .   
\end{cases}
\end{align*}
Again ignoring the logarithmic terms, this can be visualized in a plot of \(\log_n\left(\frac{\theta-\rho}{\theta}\right)\) versus \(\log_n(\theta)\) by the area of a polygon consisting of the following points:
\begin{align*}
 &\left( 0 , 0 \right)  \\
&  \left( 0, -\frac{1}{6} \right)  \\
  & \left( -\frac{1}{3}, -\frac{1}{9} \right)  \\
  &  \left( -\frac{1}{2}, 0 \right)  
\end{align*}

Next, we discuss the case \(k\geq 2\), starting with main version AWCD. Recall from Theorem \ref{thm_consistency_AWCD_general_k} the sufficient consistency condition
\begin{align*}
 \theta - \rho & \gg \max \{A, B, C\}\\
 \end{align*}
for
\begin{align*}
A &= \theta^\frac{3k+1}{2k+1} n^\frac{k-1}{2k+1},\\
B &= \theta^\frac{4k+1}{4k+2}n^{-\frac{1}{4k+2}} (\log n)^\frac{1}{4k+2}\text{ and}\\
C&= \theta^\frac{2k-1}{4k+2}n^{-\frac{1}{2}} (\log n)^\frac{1}{4k+2}.
\end{align*}
Since \(\frac{B}{A} = \theta^{-\frac{1}{2}}n^{-\frac{2k-1}{4k+2}} (\log n)^\frac{1}{4k+2}\), we have
\begin{align*}
B>A \Leftrightarrow \theta < n^{-\frac{2k-1}{2k+1}} (\log n)^\frac{1}{2k+1}.
\end{align*}
Since \(\frac{C}{B} = \theta^{-\frac{k+1}{2k+1}}n^{-\frac{k}{2k+1}}\), we have
\begin{align*}
 C>B \Leftrightarrow \theta < n^{-\frac{k}{k+1}}.
\end{align*}
Furthermore, \(\theta \gg A\) implies the upper bound
\begin{align*}
 \theta & \ll n^{-\frac{k-1}{k}}
\end{align*}
and \(\theta \gg C\) implies the lower bound
\begin{align*}
 \theta & \gg n^{-\frac{2k+1}{2k+3}} (\log n)^\frac{1}{2k+3}.
\end{align*}
This allows us to rewrite the consistency condition as
\begin{align*}
 \theta - \rho \gg \begin{cases}
                 A & ,n^{-\frac{2k-1}{2k+1}} (\log n)^\frac{1}{2k+1} \lesssim \theta \lesssim n^{-\frac{k-1}{k}}\\
                 B& ,n^{-\frac{k}{k+1}}\lesssim \theta \lesssim n^{-\frac{2k-1}{2k+1}} (\log n)^\frac{1}{2k+1} \\
                 C & ,n^{-\frac{2k+1}{2k+3}} (\log n)^\frac{1}{2k+3} \lesssim \theta \lesssim  n^{-\frac{k}{k+1}}
                \end{cases}
\end{align*}
for
\begin{align*}
 n^{-\frac{2k+1}{2k+3}} (\log n)^\frac{1}{2k+3} \ll \theta \ll n^{-\frac{k-1}{k}}.
\end{align*}
Note that the latter is in fact equivalent to the consistency condition in case \(\frac{\theta}{\rho}\equiv C\).
Using the notation \[\theta = n^{\gamma},\] we can write equivalently
 \begin{align*}
&  \log_n (\theta - \rho) \color{gray}-\frac{c}{\log n}\color{black}\geq\\
&  \begin{cases}
      \frac{\gamma(3k+1) +  k-1}{2k+1}&, \gamma \color{gray} + \mathcal O((\log n)^{-1}) \color{black}\in [ -\frac{2k-1}{2k+1} \color{gray} + \frac{1}{2k+1} \log_n(\log n)\color{black}, -\frac{k-1}{k}]
 \\
  \frac{\gamma(4k +1) - 1}{4k+2}   \color{gray} + \frac{1}{4k+2}\log_n(\log n) & \color{black}, \gamma \color{gray} + \mathcal O((\log n)^{-1}) \color{black} \in [-\frac{k}{k+1} , -\frac{2k-1}{2k+1} \color{gray}+ \frac{1}{2k+1} \log_n(\log n)\color{black}]\\
 \frac{\gamma(2k-1)}{4k+2} - \frac{1}{2} \color{gray} + \frac{1}{4k+2}\log_n(\log n) & \color{black}, \gamma \color{gray} + \mathcal O((\log n)^{-1}) \color{black} \in [-\frac{2k+1}{2k+3} \color{gray}+ \frac{1}{2k+3} \log_n(\log n)\color{black}, -\frac{k}{k+1}]
                              \end{cases}
 \end{align*}
for a large enough constant \(c\in\mathbb R\). Up to the logarithmic terms (which are vanishing for \(n\to\infty\)), we can visualize the consistency regime in a plot of \(\log_n\left(\frac{\theta - \rho}{\theta}\right)\) versus \(\log_n(\theta)\) by the area of a polygon consisting of the following points:
\begin{align*}
\left(   - \left(\frac{k-1}{k}\right),  0  \right) \\
\left(  -\left(\frac{2k-1}{2k+1} \right),  -\left(\frac{1}{(2k+1)^2} \right) \right) \\
\left(  -\left(\frac{k}{k+1}\right) , -\frac{1}{(4k+2)(k+1)}   \right) \\
\left(  -\left(\frac{2k+1}{2k+3}\right) ,   0  \right) \\
\end{align*}
Next, we consider AWCD\(^+\). Recall the consistency condition \begin{align*}
\theta -\rho \gg \max\{A, B, C^+\}
\end{align*}
for
\begin{align*}
C^+&= \theta^\frac{1}{2}n^{-\frac{k}{2k+1}} (\log n)^\frac{1}{4k+2}.
\end{align*}
Since \(\frac{C^+}{B} = \theta^{-\frac{k}{2k+1}}n^{-\frac{2k-1}{4k+2}}\) we have
\begin{align*}
C^+ > B \Leftrightarrow \theta < n^{-\frac{2k-1}{2k}}.
\end{align*}
Moreover, \(\theta \gg C^+\) implies the lower bound
\begin{align*}
 \theta \gg n^{-\frac{2k}{2k+1}} (\log n)^\frac{1}{2k+1}.
\end{align*}
The condition
\begin{align*}
 n^{-\frac{k}{2k+1}} \ll \theta \ll n^{-\frac{k-1}{k}}
\end{align*}
is in fact equivalent to the final sufficient condition above in case of \(\frac{\theta}{\rho}\equiv C\).
Moreover, according to the above, the consistency condition can be rewritten as
\begin{align*}
 \theta - \rho \gg \begin{cases}
                 A & ,n^{-\frac{2k-1}{2k+1}} (\log n)^\frac{1}{2k+1} \lesssim \theta \lesssim n^{-\frac{k-1}{k}}\\
                 B& ,n^{-\frac{2k-1}{2k}}\lesssim \theta \lesssim n^{-\frac{2k-1}{2k+1}} (\log n)^\frac{1}{2k+1} \\
                 C^+ & ,n^{-\frac{2k}{2k+1}} (\log n)^\frac{1}{2k+1} \lesssim \theta \lesssim  n^{-\frac{2k-1}{2k}}.
                \end{cases}
\end{align*}
Using the notation \[\theta = n^{\gamma},\] we can write equivalently
 \begin{align*}
 & \log_n (\theta - \rho) \color{gray}- \frac{c}{\log n}\color{black}\geq \\
 &\begin{cases}
      \frac{\gamma(3k+1) +  k-1}{2k+1}&, \gamma \color{gray}+\mathcal O((\log n)^{-1})\color{black} \in [ -\frac{2k-1}{2k+1} \color{gray} + \frac{1}{2k+1} \log_n(\log n)\color{black}, -\frac{k-1}{k}]
 \\
  \frac{\gamma(4k +1) - 1}{4k+2}   \color{gray} + \frac{1}{4k+2}\log_n(\log n) & \color{black}, \gamma \color{gray}+\mathcal O((\log n)^{-1})\color{black} \in [-\frac{2k-1}{2k} , -\frac{2k-1}{2k+1} \color{gray}+ \frac{1}{2k+1} \log_n(\log n)\color{black}]\\
 \frac{\gamma}{2} - \frac{k}{2k+1} \color{gray} + \frac{1}{4k+2}\log_n(\log n) & \color{black}, \gamma \color{gray}+\mathcal O((\log n)^{-1})\color{black} \in [-\frac{2k}{2k+1} \color{gray}+ \frac{1}{2k+1} \log_n(\log n)\color{black}, -\frac{2k-1}{2k}].
                              \end{cases}
 \end{align*}
The corresponding polygon in a plot of \(\log_n\left(\frac{\theta - \rho}{\theta}\right)\) versus \(\log_n(\theta)\) is defined by the following points:
\begin{align*}
\left(   - \left(\frac{k-1}{k}\right),  0  \right) \\
\left(  -\left(\frac{2k-1}{2k+1} \right),  -\left(\frac{1}{(2k+1)^2} \right) \right) \\
\left(  -\left(\frac{2k-1}{2k}\right) , -\frac{1}{2k(4k+2)}   \right) \\
\left(  -\left(\frac{2k}{2k+1}\right) ,   0  \right) \\
\end{align*}
Lastly, we also discuss the version AWCD\(^\circ\) without bias correction. Recall the consistency condition
\begin{align*}
\theta - \rho & \gg \max\{A, B, C^\circ\}
\end{align*}
for
\begin{align*}
 C^\circ &= \theta^\frac{k}{2k+1}n^{-\frac{k}{2k+1}}.
\end{align*}
Since \(\frac{C^\circ}{B} = \theta^{-\frac{1}{2}}n^{-\frac{2k-1}{4k+2}} (\log n)^{-\frac{1}{4k+2}}\), we have
\begin{align*}
C^\circ > B \Leftrightarrow \theta < n^{-\frac{2k-1}{2k+1}} (\log n)^{-\frac{1}{2k+1}}.
\end{align*}
Moreover, \(\theta \gg C^\circ\) implies the lower bound
\begin{align*}
 \theta \gg n^{-\frac{k}{k+1}} .
\end{align*}
The condition
\begin{align*}
 n^{-\frac{k}{2k+1}} \ll \theta \ll n^{-\frac{k-1}{k}}
\end{align*}
is equivalent to the final sufficient condition above in case of \(\frac{\theta}{\rho}\equiv C\).
Moreover, we can rewrite the consistency condition as
\begin{align*}
 \theta - \rho \gg \begin{cases}
                 A & ,n^{-\frac{2k-1}{2k+1}} (\log n)^\frac{1}{2k+1} \lesssim \theta \lesssim n^{-\frac{k-1}{k}}\\
                 B& ,n^{-\frac{2k-1}{2k+1}} (\log n)^{-\frac{1}{2k+1}}\lesssim \theta \lesssim n^{-\frac{2k-1}{2k+1}} (\log n)^\frac{1}{2k+1} \\
                 C^\circ & ,n^{-\frac{k}{k+1}} \lesssim \theta \lesssim  n^{-\frac{2k-1}{2k+1}} (\log n)^{-\frac{1}{2k+1}}.
                \end{cases}
\end{align*}
The corresponding polygon in a plot of \(\log_n\left(\frac{\theta - \rho}{\theta}\right)\) versus \(\log_n(\theta)\) is defined by the following points:
\begin{align*}
\left(   - \left(\frac{k-1}{k}\right),  0  \right) \\
\left(  -\left(\frac{2k-1}{2k+1} \right),  -\left(\frac{1}{(2k+1)^2} \right) \right) \\
\left(  -\left(\frac{k}{k+1}\right) ,   0  \right) \\
\end{align*}

\bibliography{mybib}

\end{document}